\documentclass[a4paper,10pt]{article}
\usepackage{titlesec}

\usepackage{fullpage}
\usepackage{graphicx,amssymb,amstext,amsmath, amsthm,latexsym,amscd,amsfonts,bbm}
\usepackage{enumitem}
\usepackage{mathrsfs}
\usepackage{commath}
\usepackage{upgreek}
\usepackage{dsfont}
\usepackage{hyperref}

\usepackage{cleveref}
\usepackage{bm}
\usepackage{marginnote}
\numberwithin{equation}{section}
\usepackage{braket,amsfonts}
\usepackage{mathtools}
\usepackage{array}
\usepackage{dsfont}
\usepackage{hhline}
\usepackage[mathscr]{euscript}
\usepackage{textcomp}
\usepackage[toc,page]{appendix} 
\usepackage[a4paper]{geometry}
\usepackage[caption=false]{subfig}
\usepackage[numbers,sort&compress]{natbib}

\usepackage{algorithmic}

\usepackage{graphicx,epstopdf}
\Crefname{ALC@unique}{Line}{Lines}

\allowdisplaybreaks

\newtheorem{theorem}{Theorem}[section]

\newtheorem{lemma}[theorem]{Lemma}
\newtheorem{remark}{Remark}
\newtheorem{defi}[theorem]{Definition}

\usepackage{amsthm}
\usepackage{tikz}
\usepackage{pgfplots}
\usepackage{subfig}
\usetikzlibrary{matrix,external}
\usepackage{mwe}

\usepackage{amsopn}
\DeclareMathOperator*{\esssup}{ess\,sup}

\newcommand\B{\mathcal{B}}

\newcommand\C{\mathcal{C}}
 
\newcommand\V{\mathbf{V}}

\newcommand\Ss{\mathbb{S}}

\newcommand\eps{\epsilon}
\newcommand\veps{\varepsilon}

\newcommand{\boldvphi}{\boldsymbol{\vphi}}  

\newcommand{\disp}{\displaystyle}

\newcommand{\uvect}{\mathbf{u}}
\newcommand{\qvect}{\mathbf{q}} 
\newcommand{\Div}{\text{div}}

\newcommand{\vphi}{\varphi}

\def\dx{\,\textnormal{d}x}
\def\dt{\textnormal{d}t}
\def\d{\,\textnormal{d}}

\date{}

\title{\bf Weak solutions to the heat conducting compressible self-gravitating flows in  time-dependent domains}

\author{Kuntal Bhandari \thanks{Institute of Mathematics of the  Czech Academy of Sciences, \v{Z}itn\'{a} 25, 11567 Praha 1, Czech Republic;
		 bhandari@math.cas.cz.}
	\and Bingkang Huang \thanks{School of Mathematics,
		Hefei University of Technology,
		Hefei 230009,
		China; bkhuang@whu.edu.cn, bkhuang92@hotmail.com}
	\and \v{S}\'arka Ne\v{c}asov\'a \thanks{Corresponding author, Institute of Mathematics of the  Czech Academy of Sciences, \v{Z}itn\'{a} 25, 11567 Praha 1, Czech Republic; matus@math.cas.cz.}
}

\begin{document}
	\maketitle

\begin{abstract}
	In this paper, we consider 
	the heat-conducting   compressible self-gravitating fluids in time-dependent domains, which typically describe the motion of viscous gaseous stars.
	The flow is governed by the 3-D Navier-Stokes-Fourier-Poisson equations  where  the velocity is supposed to fulfil  the full-slip boundary condition and the temperature on the boundary is given by a non-homogeneous Dirichlet condition. We establish the global-in-time weak solution to the system. Our approach is based on the penalization of the boundary behavior, viscosity, and the pressure in the weak formulation. 
	Moreover, to accommodate the non-homogeneous boundary heat flux, the concept of {\em ballistic energy} is utilized in this work.
\end{abstract}


	

\section{Introduction and general setting}

\subsection{Problem statement}\label{section-problm-state}

In this article, we study the existence of global-in-time weak solutions to the  flow of {\em viscous gaseous stars} where the spatial domain  is varying with respect to time by means of some given velocity field.
It is by now well-understood that  the stars may be considered as a compressible fluid  (e.g., \cite{S-N-Shorey-book}) and their dynamics are very often shaped and controlled by self-gravitation and high-temperature radiation effects (see for instance \cite{Cox, Ghatak-et-al}).  The mathematical model of such flows is governed  by the 3-D compressible {\em Navier-Stokes-Fourier-Poisson} system (see for instance, \cite{Eduard-Ducomet-2} by B. Ducomet and E. Feireisl): consisting of {\em equation of continuity}, {\em equation of momentum}, {\em energy equation} and {\em Poisson equation}, respectively given by
\begin{align}
	&\partial_t \rho + \Div_x ( \rho \uvect) = 0  ,   \label{continuity-eq} 
	\\
	&\partial_t (\rho \uvect) + \Div_x (\rho \uvect \otimes \uvect)  + \nabla_x p (\rho, \theta)   = \Div_x \Ss  + \rho \nabla_x \Psi  
	,  \label{momentum-eq} \\
	&\partial_t  (\rho e(\rho, \theta)) +  \Div_x(\rho e(\rho, \theta) \uvect)  + 
	\Div_x \qvect + p(\rho, \theta) \Div_x \uvect  = \Ss: \nabla_x \uvect  , \label{temp-eq} 
	\\
	&-\Delta_x \Psi = 4\pi g  \rho  \label{pois-eq}  , \ \ \  \int_{\Omega_t} \Psi \dx  =0 .
\end{align}   
The  density $\rho$, velocity $\uvect$, and absolute temperature $\theta$ are three typical macroscopic quantities that describe the motion of the  fluid, while  $p(\rho, \theta)$ is the pressure, 	$s(\rho, \theta)$ is the specific entropy and  $e(\rho, \theta)$ is the specific internal energy of the fluid, interrelated through the Gibb's equation 
\begin{align}\label{equ-gibbs}
	\theta D s  = D e + p D (1/\rho) ,
\end{align}
where $D$ stands for the total differential.


The fluid under consideration is assumed to be Newtonian, meaning that the viscous stress
tensor $\Ss$ depends linearly on the velocity's gradient. Precisely,   $\Ss$  is given by 
\begin{align}\label{stress_tensor}
	\Ss(\theta, \nabla_x \uvect) = \mu(\theta) \left( \nabla_x \uvect + \nabla^{\top}_x \uvect -\frac{2}{3} \Div_x \uvect \mathbb I  \right) +\eta(\theta) \Div_x \uvect \mathbb I,
\end{align}
with shear viscosity coefficient $\mu(\theta)>0$ and bulk viscosity coefficient $\eta(\theta)\geq 0$. 

In the momentum equation \eqref{momentum-eq}, $\rho \nabla_x \Psi$ is an external force acting on the fluid due to self-gravitation, and the gravitational potential $\Psi$ (of the star) solves  the Poisson equation given by \eqref{pois-eq} with the gravity $g>0$.

Moreover, the heat flux $\qvect$ is determined by the Fourier's law,
\begin{align}\label{Fourier_law}
	\qvect (\theta, \nabla_x \theta)= - \kappa (\theta) \nabla_x \theta , 
\end{align}
with the heat conductivity coefficient $\kappa(\theta)>0$.

\paragraph{Description of the time-dependent domain.}
Let us consider a regular domain  $\Omega_0 \subset \mathbb R^3$ occupied by the fluid at the initial time $t=0$. Then, we investigate the domain $\Omega_\tau$ (at time $t=\tau$) described by a given velocity field $\V(t,x)$ with $t\geq 0$ and $x\in \mathbb R^3$. More specifically, when $\V$ is regular enough,  we solve the associated system of differential equations
\begin{align}\label{trajec-eq}
	\frac{d}{dt} \mathbf X (t,x) =  \V(t,x) , \quad t>0 , \quad \mathbf X(0,x)=x ,
\end{align}
and set 
\begin{align}\label{domains}
	\begin{dcases} 
		\Omega_\tau = \mathbf X(\tau, \Omega_0) , \quad  \Gamma_\tau : = \partial \Omega_\tau, \quad \text{and} \\
		Q_\tau := \cup_{t\in (0,\tau)} \{t\} \times \Omega_t . 	
	\end{dcases}
\end{align}
We assume that the volume of the domain cannot degenerate in time, meaning that 
\begin{align}\label{measurable-condition}
	\exists \, M_0>0 \ \text{ such that } \ |\Omega_\tau|\geq M_0 \quad \forall \tau \in [0,T].
\end{align}
Moreover, we make the following assumption
\begin{align}\label{condition-V}
	\Div_x \V = 0  \ \text{ in the neighborhood of } \Gamma_\tau    \quad \forall \tau \in [0,T].
\end{align}
The condition \eqref{condition-V} is not restrictive.
In fact,  it has been indicated in  \cite[Remark 5.3]{Sarka-et-al-ZAMP} that for a general $\V \in  \C^1([0,T];\C^3_c(\mathbb R^3, \mathbb R^3))$, one 
can find $w \in W^{1,\infty}(Q_T)$ such that $(\V- w)|_{\Gamma_\tau} = 0$ for all $\tau \in [0,T]$ and $\Div_x w=0$ on some
neighborhood of $\Gamma_\tau$;  see also \cite[Section 4.3.1]{Sarka-Kreml-Neustupa-Feireisl-Stebel}.

\paragraph{Boundary conditions.} Here we prescribe the  boundary conditions for the original system \eqref{continuity-eq}--\eqref{pois-eq}. 

\vspace*{.1cm}
\noindent 
$\bullet$ We first impose the Navier-slip boundary conditions
\begin{align}\label{Navier-slip}
	[\Ss \mathbf n]_{\tan} + \alpha [\uvect - \V]_{\tan} = 0 ,   \quad \text{in } \Gamma_t, \text{ for any } t \in [0,T], 
\end{align}
where $\alpha\geq 0$ represents a {\em friction} coefficient and $\Ss$ is the viscous stress tensor. For simplicity, we take $\alpha=0$ which gives the {\em full-slip} condition. 
Furthermore, the impermeability condition for $\uvect$ is given by 
\begin{align}\label{imperm}
	(\uvect - \V)\cdot \mathbf n = 0 ,  \quad \text{in } \Gamma_t, \text{ for any } t \in [0,T].
\end{align}

\vspace*{.1cm}
\noindent 
$\bullet$ The fluid temperature on the lateral boundary of the domain is given by
\begin{align}\label{boundary-theta-1} 
	\theta \big|_{\bigcup_{t\in (0,T)} \big(\{t\}\times  \Gamma_t\big)} = \theta_B,
\end{align}
where $\theta_B=\theta_B(t,x)$ is a strictly positive smooth function. 

\vspace*{.1cm}
\noindent 
$\bullet$ The gravitational potential  $\Psi$ satisfies the Neumann boundary condition 
\begin{align}\label{boundary-Psi}
	\frac{\partial \Psi}{\partial {\mathbf n}} = 0  \quad \text{on } \Gamma_t \ \ \text{for each } t\in [0,T].
\end{align}

\paragraph{Initial conditions.}
The system \eqref{continuity-eq}--\eqref{pois-eq}  is also supplemented with the initial conditions
\begin{align}\label{initial_conditions}
	\begin{dcases} 
		\rho(0, \cdot)= \rho_0 \in L^{\frac{5}{3}}(\Omega_0) , \ \text{ and }   \\
		(\rho \uvect)(0,\cdot) = (\rho \uvect)_0, \ \
		\theta(0,\cdot)=\theta_0 \quad \text{in } \Omega_0 ,
	\end{dcases} 
\end{align}
where we assume that the fluid density is zero outside the domain $\Omega_0$, more precisely,
\begin{align}\label{condition-density}
	\rho_0 \geq 0 \ \ \text{in } \Omega_0, \ \ \rho_0  \not\equiv 0 \  \text{ and } \   \rho_0 =0 \ \text{ in } \ \mathbb R^3 \setminus \Omega_0.
\end{align}
Further, it holds that $0< \underline \theta \leq \theta_0 \leq \overline \theta$  for some positive constants $\underline \theta$ and $\overline \theta$, and 
\begin{align}
	(\rho s)_0 = \rho_0 s(\rho_0, \theta_0) \in L^1(\Omega_0) . 
\end{align}
Beside that, we assume 
\begin{align}\label{initial-energy}
	\mathcal E_0:= \int_{\Omega_0}\left( \frac{1}{2\rho_0} |(\rho \uvect)_0 |^2 +  \rho_0 e(\rho_0,\theta_0)  \right) <+\infty .
\end{align} 

Our goal is to establish the global-in-time existence of weak solutions to the whole system 
\eqref{continuity-eq}--\eqref{pois-eq} in the domain $Q_T$ with the boundary conditions \eqref{Navier-slip}--\eqref{boundary-Psi} and initial conditions \eqref{initial_conditions}--\eqref{initial-energy}.

\subsection{Bibliographic comments and main goal of our work}
The self-gravitating flows have wide applications in astrophysics and the theory of nuclear fluids. In that regard, we mention that
the global in time weak solutions for the compressible barotropic self-gravitating fluids governed by {\em Navier-Stokes-Poisson} equations  has been initially studied by B. Ducomet and E. Feireisl \cite{Eduard-Ducomet-1} in the fixed spatial domain. Later on, they  established the existence theory  of weak solutions for the compressible {\em Navier-Stokes-Fourier-Poisson} (in short N-S-F-P) system in  \cite{Eduard-Ducomet-2}. On the other hand, B. Ducomet et al \cite{Ducoment-Sarka-1} considered a compressible N-S-F-P system describing a motion of a viscous heat-conducting rotating fluid on a thin domain $\Omega_\epsilon=\omega \times (0,\epsilon)$ with a 2-D domain $\omega$ and positive $\epsilon.$ More precisely, the authors in \cite{Ducoment-Sarka-1} proved that the weak solutions in the 3-D domain converge to the strong solutions of the 2-D system  as $\epsilon\to 0$ in the time interval where the strong solution exists.    

\vspace*{.1cm}

In the context of compressible fluids in time dependent domains, we first address the work \cite{Sarka-Kreml-Neustupa-Feireisl-Stebel} by E. Feireisl et al where the existence of weak solutions of the barotropic compressible  Navier-Stokes systems has been addressed on time dependent domains as prescribed in \eqref{trajec-eq}--\eqref{measurable-condition}. Their approach is based on the penalization of the boundary behavior, viscosity and pressure in the weak formulation.  Later on, the existence of weak solutions to  the full {\em Navier-Stokes-Fourier} (in short N-S-F) systems in the time dependent domain has been treated by  O. Kreml et al \cite{sarka-aneta-al-JMPA}, see also \cite{Sarka-et-al-ZAMP}.  We also note here that the compressible micropolar fluids on a time-dependent domain with slip boundary conditions was considered in \cite{HBZ-2022}. Furthermore, the local-in-time existence of strong solutions to the compressible Navier-Stokes on the moving domains was given in \cite{KNP-20}. The global well-posedness of compressible Navier-Stokes equations on a moving domain in the $L^p$-$L^q$ frame was investigated in \cite{KNP-20-springer}.  Recently, the authors in  \cite{Macha-Muha-Necasova-Roy-Srjdan}  studied the existence of a weak solution to a nonlinear fluid-structure interaction model with heat exchange where the shell is governed
by linear thermoelasticity equations and encompasses a time-dependent domain that is filled with a fluid governed by the full N-S-F system. 
In this regard, we also mention the work \cite{Kalousek2023existence}, where the authors
analyze  a system governing the interaction between two compressible mutually noninteracting fluids and a shell of Koiter type that actually gives rise to a time-dependent 3-D domain filled by the fluids.

\vspace*{.1cm}

In fixed spatial domains, the existence theory of compressible barotropic Navier-Stokes systems was developed by P. L. Lions \cite{Lions1996mathematical} and later it has been extended in \cite{Feireisl2001existence}  to a class of physically relevant pressure-density state equations. The existence of weak solutions to the full N-S-F system has been then established by E. Feireisl \cite{Feireisl-NSF-1, Feireisl-NSF-2} and by E. Feireisl and A. Novotn\'{y} in \cite{Feireisl-Novotny-book}.

\vspace*{.1cm}

The global in time weak solutions to the  Navier-Stokes-Fourier system with nonhomogeneous Dirichlet data (for both velocity and temperature) in fixed spatial domain has been rigorously studied by N. Chaudhuri and E. Feireisl \cite{Chauduri-Feireisl}; they also investigated the weak-strong uniqueness result for their system. To handle the nonhomogeneous Dirichlet data  for temperature and boundary heat flux, the authors in \cite{Chauduri-Feireisl} introduced the  concept of 
{\em ballistic energy} (see also \cite{Danica}).

\vspace*{.1cm}

In the present work, we study the weak existence theory for the compressible {\em Navier-Stokes-Fourier-Poisson} equations in  time-dependent domain where we allow  the nonhomogeneous Dirichet condition for the temperature and  {\em non-vanishing heat flux on the boundary}, which is certainly more physical in the viewpoint of the motion of nuclear fluids or gaseous stars, and to the best of our knowledge, this problem has not been considered in the literature yet.  As a matter of fact, we shall intensively  use the concept of {\em ballistic energy} to accommodate the nonhomogeneous heat flux on the spatial boundary. 

\vspace*{.1cm}

Before going to more details, we  note down the constitutive relations that will be used to proceed our work.

\subsection{Hypothesis}\label{Section-hypothesis}
Motivated by \cite{Eduard-Ducomet-2, Feireisl-NSF-1,Feireisl-NSF-2},   let us now make the following set of  assumptions. 

\vspace*{.1cm}
\noindent
$\bullet$ {\bf Viscosity coefficients.}	In agreement with several practical experiments, we consider the viscosity coefficients $\mu(\theta)$ and $\eta(\theta)$ to be continuously differentiable functions depending on the temperature $\theta$, namely $\mu(\theta), \eta(\theta)\in \C^1([0,+\infty))$ and satisfy 
\begin{equation}\label{hypo-mu}
	\begin{aligned}
		&0< \underline \mu (1+\theta) \leq \mu(\theta) \leq \overline \mu (1+\theta), \quad \sup_{\theta \in [0,+\infty)} |\mu^\prime(\theta)| \leq \overline m , \ \ \text{and}\\
		&  0\leq  \underline \eta (1+\theta) \leq \eta(\theta) \leq \overline \eta (1+\theta) . 
	\end{aligned}
\end{equation} 

\vspace*{.1cm}
\noindent 
$\bullet$ {\bf Heat conductive coefficient.} In accordance with the recent work \cite{Chauduri-Feireisl} (see also \cite{Danica}), we need a much stricter assumption on the heat conducting coefficient $\kappa(\theta)$ appearing in the Fourier's law \eqref{Fourier_law}. More precisely, we assume 
$\kappa(\theta) \in \C^1([0,+\infty))$ such that
\begin{equation}\label{hypo-kappa}
	0< \underline{\kappa} (1+\theta^\alpha) \leq \kappa(\theta) \leq  \overline{\kappa}(1+\theta^\alpha) , \quad \text{for } \alpha>6. 
\end{equation}
In above, all the quantities  $\underline \mu$, $\overline \mu$, $\overline m$, $\underline \eta$, $\overline \eta$, $\underline{\kappa}$, $\overline{\kappa}$ are positive.

\vspace*{.1cm}
\noindent 
$\bullet$ {\bf Constitutive relations for pressure, internal energy and entropy.}
The pressure $p(\rho, \theta)$ is supposed to be composed from the interaction between particles of the fluid, and the radiation term due to the temperature. More precisely,  we  consider   
\begin{align}\label{hypo-press}
	p(p, \theta) = p_M(\rho, \theta) + p_R(\theta), 
\end{align} 
where the properties of $p_{M}(\rho,\theta)$ are given below in \eqref{hypo-p_m}, \eqref{hypo-p_m_with_P},  and 
the radiative pressure $p_R(\theta)$ given by 
\begin{align*}
	p_R(\theta)=\frac{a}{3} \theta^4 ,
\end{align*}
with the so-called    Stefan-Boltzmann constant $a>0$. We refer  for instance \cite{buet2004asymptotic} by   C. Buet and  B. Despr{\'{e}}s and 
\cite{ducomet2001simplified} by B. Ducomet for more details about those assumptions.

\vspace*{.1cm}

Similarly, the decomposition of  internal energy $e(\rho, \theta)$ and specific entropy  $s(\rho, \theta)$ read as 
\begin{align}\label{hypo-energy}
	e(\rho, \theta) = e_M(\rho,\theta) + \frac{a}{\rho} \theta^4  
\end{align}
and 
\begin{align}\label{hypo-entropy}
	s(\rho, \theta) = s_M(\rho,\theta) + \frac{4a}{3\rho} \theta^3  
\end{align}
respectively.  We refer \cite{Eduard-Ducomet-2} B. Ducomet and E. Feireisl,  for the details about the hypothesis \eqref{hypo-energy}, \eqref{hypo-entropy}. 
According to the hypothesis of thermodynamic stability, the molecular components  satisfies 
\begin{align}\label{hypo-p_m}
	\frac{\partial p_M}{\partial \rho}  >  0 \quad \forall \, \rho, \, \theta >0  , 
\end{align}
and there exists some positive constant $c>0$ such that
\begin{align}\label{hypo-e_m}
	0 < \frac{\partial e_M}{\partial \theta} \leq c  \quad \forall \, \rho, \, \theta > 0  .
\end{align}
Moreover, it holds that 
\begin{align}\label{hypo-limit-e_m}
	\lim_{\theta \to 0^+} e_M(\rho, \theta) = \underline{e_M}(\rho) >0  \quad \text{for any fixed } \ \rho >0,
\end{align}
and 
\begin{align}\label{hypo-bound_deri_e_m}
	\left| \rho \frac{\partial e_M}{\partial \rho}(\rho, \theta)\right| \leq c {e_M}(\rho, \theta)  \quad \forall \, \rho, \, \theta >0.
\end{align}
We also suppose that there is a function 
\begin{align}\label{hypo_P}
	P\in \C^1[0,\infty), \ \ \ P(0) =0 , \ \ \ P^\prime(0) >0,
\end{align}
and two positive constants $\underline Z, \overline Z$ such that 
\begin{align}\label{hypo-p_m_with_P}
	p_M(\rho, \theta) = \theta^{\frac{5}{2}} P\left(\frac{\rho}{\theta^{\frac{3}{2}}}\right) \ \ \ \text{whenever }  0< \frac{\rho}{\theta^{\frac{3}{2}}} \leq \underline Z , \ \ \text{or} \ \ \frac{\rho}{\theta^{\frac{3}{2}}} > \overline Z,
\end{align}
and satisfying the relation 
\begin{align}\label{hypo-p_m_with_P-2}
	p_M (\rho , \theta) = \frac{2}{3} \rho e_M(\rho, \theta) , \quad \text{for } \frac{\rho}{\theta^{\frac{3}{2}}} > \overline Z .
\end{align}
Based on the above assumption, we derive that 
\begin{align}\label{lower_bound_rho_e}
	\rho e(\rho, \theta) \geq a \theta^4 + \frac{3}{2} p_\infty \rho^{\frac{5}{3}}	.
\end{align}
This estimate can be shown in the following explicit way. First, observe that $P^\prime(Z)>0$ for all $0<Z<\underline Z$ or $Z>\overline Z$. Now, we extend $P$ as a strictly increasing function in $[\underline Z, \overline Z]$ so that we have 
\begin{align}\label{fact-1}
	P^\prime(Z)>0 \quad \forall Z>0.
\end{align} 
Next, by \eqref{hypo-bound_deri_e_m}, \eqref{hypo-p_m_with_P} and \eqref{hypo-p_m_with_P-2}  we infer that
\begin{align}\label{fact-2-1}
	\lim_{Z\to \infty} \frac{P(Z)}{Z^{\frac{5}{3}}} = p_\infty > 0 . 
\end{align}
Using \eqref{hypo-p_m_with_P}, \eqref{hypo-p_m_with_P-2} and \eqref{fact-2-1} one can deduce that
\begin{align}\label{limit-e_M} 
	\lim_{\theta\to 0^+} e_M(\rho, \theta) = \frac{3}{2} \rho^{\frac{2}{3}} p_\infty .
\end{align}
Moreover, $e_M$ is a strictly increasing function of $\theta$  in $(0,\infty)$ (see \eqref{hypo-e_m})  for any fixed $\rho$, which  together with \eqref{hypo-energy} and \eqref{limit-e_M}, we obtain  the required estimate \eqref{lower_bound_rho_e}.

\vspace*{.1cm}
Further, in agreement with the Gibb's relation \eqref{equ-gibbs}, the molecular  component $s_M$ of the entropy $s$ satisfies
\begin{align}\label{Gibbs-1}
	\frac{\partial s_M}{\partial \theta} = \frac{1}{\theta} \frac{\partial e_M}{\partial \theta} \ \  \ \text{and} \ \ \  
	\frac{\partial s_M}{\partial \rho} = -\frac{1}{\rho^2} \frac{\partial p_M}{\partial \theta} .
\end{align}
We set 
\begin{align}\label{def-s-m}
	s_M(\rho, \theta) = S(Z), \ \ Z = \frac{\rho}{\theta^{3/2}}, \ \ S^\prime(Z) =  -\frac{3}{2}\frac{\frac{5}{3} P(Z) - ZP^\prime(Z)}{Z^2} < 0 
\end{align}
in the degenerate area $\rho > \overline Z \theta^{3/2}$. We also require that the  {\em Third law of thermodynamics} is satisfied,
\begin{align}\label{third-law}
	\lim_{Z\to \infty} S(Z) = 0  .
\end{align}
At this point, we refer the book  \cite[Chapter 2]{Feireisl-Novotny-book} by E. Feireisl and A. Novotn{\'{y}} 
for more details about the assumptions on the constitutive relations  \eqref{hypo-press} to \eqref{third-law}.

\section{Weak formulations}\label{section-prem-weak-for}

In this section, we shall prescribe the expected weak formulations for the system 
\eqref{continuity-eq}--\eqref{pois-eq}.
Throughout the section, we assume that the density $\rho$ remains ``zero" outside the   fluid domain $\Omega_t$ for each $t\in [0,T]$. Indeed, we will show that if the initial density $\rho_0=0$ outside $\Omega_0$, then $\rho$ will also vanish outside $\Omega_t$ for any $t\in (0,T]$ (see Section \ref{Section:Solid-part}).

\vspace*{.1cm}
\noindent 
{\bf I.} {\bf Continuity equation.}	 It is convenient to consider the continuity equation in the whole physical space $\mathbb R^3$ provided the density is supposed to be zero outside the fluid domain $\Omega_t$ for each $t\in [0,T]$. Specifically, the weak formulation of the continuity equation \eqref{continuity-eq} is supposed to be 
\begin{align}\label{weak-continuity}
	-	\int_0^T \int_{\Omega_t} \left(\rho \partial_t \vphi +    \rho \uvect \cdot \nabla_x \vphi \right)  =  \int_{\Omega_0} \rho_0(\cdot) \vphi(0, \cdot)  ,
\end{align}
for any  test function $\vphi \in \C^1( [0,T]\times \mathbb R^3; \mathbb R)$ with $\vphi(T,\cdot)=0$.  Of course, we assume that $\rho \geq 0$ a.e. in $\mathbb R^3$.

Moreover, the equation \eqref{continuity-eq}  will  also be satisfied in the sense of renormalized solutions introduced by DiPerna and Lions \cite{Lions-Diperna}: 
\begin{align}\label{weak-con-ren}
	-	\int_0^T \int_{\Omega_t} \rho B(\rho)\left( \partial_t \vphi +     \uvect \cdot \nabla_x \vphi \right) + 
	\int_0^T \int_{\Omega_t} b(\rho) \Div_x \uvect \vphi 
	=  \int_{\Omega_0} \rho_0 B(\rho_0)\vphi(0, \cdot),
\end{align}
for any   test function $\vphi \in \C^1( [0,T]\times \mathbb R^3; \mathbb R)$ with $\vphi(T,\cdot)=0$, $b \in L^\infty \cap \C([0,+\infty))$ such that $b(0)=0$ and $\displaystyle B(\rho) = B(1)+ \int_{1}^\rho \frac{b(z)}{z^2}$. 

\vspace*{.1cm}
\noindent
{\bf II.} {\bf Momentum equation.}   We  write the expected weak formulation for the momentum equation as 
\begin{align}\label{weak_form_momen}
	-\int_0^T \int_{\Omega_t} \left( \rho \uvect \cdot \partial_t  \boldvphi   +  \rho[\uvect \otimes \uvect] : \nabla_x  \boldvphi + p(\rho, \theta) \Div_x \boldvphi   \right)     - \int_{\Omega_0}   (\rho \uvect)_0\cdot  \boldvphi(0, \cdot) \notag   \\
	= \int_0^T \int_{\Omega_t} \rho \nabla_x \Psi \cdot \boldvphi - \int_0^T \int_{\Omega_t} \Ss : \nabla_x \boldvphi ,
\end{align}
for  any test function $\boldvphi\in \C^1(\overline Q_T; \mathbb R^3)$ satisfying 
\begin{align*}
	\boldvphi(T,\cdot)=0,  \ \ \text{in } \Omega_T, \text{ and } \ \  	(\boldvphi \cdot \mathbf n ) |_{\Gamma_t} = 0 , \quad \text{for  any $t \in [0,T]$.}
\end{align*}
The impermeability condition will be then satisfied in the sense of trace, 
\begin{align*}
	(\uvect  , \nabla_x \uvect)\in L^2(Q_T;  \mathbb R^3)   , \quad (	\uvect - \V)\cdot \mathbf n |_{\Gamma_t} = 0, \quad \text{for any } t\in [0,T].   
\end{align*}

\vspace*{.1cm}
\noindent 
{\bf III.} {\bf Entropy inequality.} Using the Gibbs' equation \eqref{equ-gibbs}, we deduce the entropy equation from \eqref{temp-eq}, given by 
\begin{align}\label{entropy_eq}
	\partial_t (\rho s) + \Div_x (\rho s \uvect ) + \Div_x \left(\frac{\qvect}{\theta} \right) = \frac{1}{\theta} \left( \Ss : \nabla_x \uvect - \frac{\qvect}{\theta} \cdot \nabla_x \theta    \right) .
\end{align}
Based on the fact that the $a$ $priori$ bounds only provide the $ L^1$ bound for the entropy production rate, the entropy equation is formulated by an inequality (see \cite{Feireisl-Novotny-book}).
\begin{align}\label{entropy_ineq}
	\partial_t (\rho s) + \Div_x (\rho s \uvect ) + \Div_x \left(\frac{\qvect}{\theta} \right) \geq \frac{1}{\theta} \left( \Ss : \nabla_x \uvect - \frac{\qvect}{\theta} \cdot \nabla_x \theta    \right). 
\end{align}
Then, the weak formulation  for \eqref{entropy_ineq} should be of the form:
\begin{multline}\label{weak-for-entropy}
	- \int_0^T \int_{\Omega_t} \left(\rho s\partial_t \vphi + \rho s\uvect \cdot \nabla_x \vphi + \frac{\qvect}{\theta} \cdot \nabla_x \vphi  \right)  - \int_{\Omega_0} (\rho s)_0 \vphi(0,\cdot)   \\
	\geq \int_0^T\int_{\Omega_t} \frac{\vphi}{\theta}  \left( \Ss : \nabla_x \uvect  - \frac{\qvect}{\theta} \cdot \nabla_x \theta    \right) , 
\end{multline} 
for any  test function $\vphi \in \C^1(\overline Q_T; \mathbb R)$ with $\vphi \geq 0$, $\vphi(T,\cdot)=0$  and $\vphi|_{\Omega_t}=0$ for all $t\in [0,T]$.

\vspace*{.1cm}
\noindent
{\bf IV.} {\bf Poisson equation.} The Poisson equation \eqref{pois-eq} will be  considered in the whole space  $\mathbb R^3$ provided $\rho=0$ outside the domain $\Omega_t$ for each $t\in [0,T]$.  Accordingly, the expected weak formulation for the Poisson equation will be
\begin{align}\label{weak-poission}
	\int_0^T \int_{\Omega_t} \nabla_x \Psi \cdot \nabla_x \vphi = \int_0^T \int_{\Omega_t} \rho \vphi 
\end{align}
for any test function $\vphi \in \C^1(\overline Q_T; \mathbb R)$.

\vspace*{.1cm}
\noindent 
{\bf V.} {\bf Ballistic energy inequality.}
We note that the weak formulation of energy equation  $\eqref{temp-eq}$ cannot be directly used to define the weak solution to system \eqref{continuity-eq}--\eqref{pois-eq}, due to the absence of information about heat flux on the boundary.  The method we adopted here is combining the weak formulation of entropy with the energy balance to get the ballistic energy.

Assuming all the quantities of concern are smooth and multiplying the momentum equation \eqref{momentum-eq} by $(\uvect -  \V)$, then integrating by parts w.r.t. space variable and integrating the energy equation \eqref{temp-eq}, one has by summing up (as well as using the continuity equation)
\begin{align}\label{energy-1}
	&	 \frac{\d}{\dt} \int_{\Omega_t}\left(\frac{1}{2} \rho |\uvect|^2  + \rho e   \right) + \int_{\Gamma_t} \qvect\cdot \mathbf n - \int_{\Omega_t} \rho \nabla_x \Psi \cdot (\uvect -\V)   \notag 
	\\
	= &
	-\int_{\Omega_t} \left(\rho[\uvect \otimes \uvect]    : \nabla_x \V - \Ss : \nabla_x \V + p\, \Div_x \V \right)
	+ \int_{\Omega_t} \partial_t(\rho \uvect) \cdot \V .
\end{align} 
We now observe that 
\begin{align}\label{esti-poisson}
	-\int_{\Omega_t}  \rho \nabla_x \Psi \cdot \uvect =  \int_{\Omega_t} \Psi \, \Div_x (\rho \uvect ) = -\int_{\Omega_t} \Psi \, \partial_t \rho = \frac{1}{4\pi g} \int_{\Omega_t} \Psi \partial_t(\Delta_x \Psi)   \notag \\
	= -\frac{1}{8\pi g} \frac{\d}{\dt}\int_{\Omega_t} |\nabla_x \Psi|^2  ,
\end{align}
where no boundary integral will appear after the first integration by parts since the density is supposed to be ``zero"  outside the fluid domain $\Omega_t$ for any  $t\in [0,T]$.  In the second integration by parts we simply use that $\disp \frac{\partial \Psi}{\partial {\mathbf n}}=0$ on $\Gamma_t$ for each $t\in [0,T]$.

Using \eqref{esti-poisson} in \eqref{energy-1}, we have
\begin{align}\label{energy-1-2}
	&	 \frac{\d}{\dt} \int_{\Omega_t}\left(\frac{1}{2} \rho |\uvect|^2  + \rho e  -\frac{1}{8\pi g} |\nabla_x \Psi|^2 \right) + \int_{\Gamma_t} \qvect\cdot \mathbf n 
	\notag \\
	= &
	-\int_{\Omega_t} \left(\rho[\uvect \otimes \uvect]    : \nabla_x \V - \Ss : \nabla_x \V + p\, \Div_x \V \right)
	+ \int_{\Omega_t} \partial_t(\rho \uvect) \cdot \V 
	- \int_{\Omega_t} \rho \nabla_x \Psi \cdot \V  .
\end{align}
The last term in the r.h.s. of the above equality can be computed as follows:   
\begin{align*}
	-	\int_{\Omega_t} \rho (\nabla_x \Psi \cdot \V) =\frac{1}{4\pi g} \int_{\Omega_t}  \Delta_x \Psi (\nabla_x \Psi \cdot \V) = -\frac{1}{4\pi g} \int_{\Omega_t}\nabla_x \Psi \cdot \nabla_x (\nabla_x \Psi \cdot \V) ,
\end{align*}
thanks to the homogeneous Neumann boundary condition of $\Psi$, and eventually one has,
\begin{align}\label{esti-poiison-V}
	-	\int_{\Omega_t} \rho (\nabla_x \Psi \cdot \V) 
	= - \frac{1}{8\pi g} \int_{\Omega_t} |\nabla_x \Psi|^2 \Div_x \V .
\end{align}
But from the Poisson equation \eqref{pois-eq} we  have 
\begin{align*}
	\int_{\Omega_t} |\nabla_x \Psi|^2 \leq C(g) \|\rho\|_{L^{\frac{6}{5}}(\Omega_t)} \|\Psi\|_{L^6(\Omega_t)} \leq \epsilon \|\Psi\|^2_{W^{1,2}(\Omega_t)} + 
	\frac{C(g)}{\epsilon}  \|\rho\|^2_{L^{\frac{6}{5}}(\Omega_t)},
\end{align*}
for some constant $C(g)>0$ and for any $\epsilon>0$, thanks to the fact that $W^{1,2}(\Omega_t)\hookrightarrow L^6(\Omega_t)$ (as we are in dimension $3$). Now, since $\disp \int_{\Omega_t} \Psi \dx =0$,  by using generalized Poincar\'e inequality (see Lemma \ref{Poincare}) we have 
\begin{align}\label{bound-poisson}
	\|\Psi\|_{W^{1,2}(\Omega_t)} \leq  C(g) \|\rho\|_{L^{\frac{6}{5}}(\Omega_t)} \leq C_g\|\rho\|^{\frac{7}{12}}_{L^1(\Omega_t)}\|\rho\|^{\frac{5}{12}}_{L^{\frac{5}{3}}(\Omega_t)} \leq C(\rho_0, g)
	\Big(\int_{\Omega_t} \rho^{\frac{5}{3}}\Big)^{\frac{1}{4}} 
	\notag  \\
	\leq \Big(\int_{\Omega_t}\rho^{\frac{5}{3}}\Big)^{\frac{1}{2}} + C(\rho_0, g),
\end{align}
where we also used the standard interpolation inequality.

\vspace*{.1cm}
This leads to the following two facts:
\begin{itemize}
	\item[(i)] We have  
	\begin{align}\label{bound-rho-psi-V}  
		\left|\int_{\Omega_t} \rho \nabla_x \Psi \cdot \V  \right| \leq C(\V)\int_{\Omega_t} \rho^{\frac{5}{3}} + C(g,\V,\rho_0), 
	\end{align}
	(recall that the fluid satisfies the mass conservation law, i.e., $\disp\int_{\Omega_t}\rho \dx = \int_{\Omega_0}\rho_0 \dx$.)
	
	\item[(ii)] The negative term $\disp -\frac{1}{8\pi g} \|\nabla_x \Psi\|^2_{L^2(\Omega_t)}$ can be absorbed by the term $\rho e(\rho, \theta)$ in the equation \eqref{energy-1-2}.  In fact, the estimate \eqref{bound-poisson} can be written in a more precise way, namely
	\begin{align*}
		\|\Psi\|^2_{W^{1,2}(\Omega_t)} \leq 
		6\pi g p_{\infty} \int_{\Omega_t} \rho^{\frac{5}{3}} + C(\rho_0, g, p_\infty) ,
	\end{align*}
	so that using the lower bound of $\rho e(\rho, \theta)$ from \eqref{lower_bound_rho_e}, we have 
	\begin{align}\label{bound-grad-psi}
		\frac{1}{8\pi g} \|\nabla_x \Psi\|^2_{L^2(\Omega_t)} \leq \frac{3p_\infty}{4}  \int_{\Omega_t} \rho^{\frac{5}{3}} + C(\rho_0, g, p_\infty) \leq \frac{1}{2} \rho e(\rho, \theta) + C(\rho_0,g, p_\infty) ,
	\end{align}
	and thus the required result follows. 
\end{itemize}

However, the relation \eqref{energy-1-2} cannot be used in the weak formulation as the boundary integral 
\begin{align*}
	\int_{\Gamma_t} \qvect \cdot \mathbf n 
\end{align*} 
cannot be controlled  in the formulation. To get rid of this,
we shall multiply the entropy inequality \eqref{entropy_ineq} by some $\widetilde \theta\in \C^1(\overline Q_T; \mathbb R)$ such that 
\begin{align}\label{def-theta-tilde}
	\inf_{\overline Q_T} \widetilde\theta>0 \ \text{ and } \ \widetilde \theta (t,x) = \theta_B(t,x)  \quad \text{for } (t,x) \in \cup_{t\in(0,T)}\big(\{t\} \times \Gamma_t\big) , 
\end{align}
and integrating by parts we have
\begin{align}\label{energy-2}
	-	\frac{\d}{\dt} \int_{\Omega_t}  \rho s \widetilde \theta 
	+ \int_{\Omega_t} \left[ \rho s \left(\partial_t \widetilde \theta +  \uvect \cdot \nabla_x \widetilde \theta\right) + \frac{\qvect}{\theta} \cdot \nabla_x \widetilde \theta \right] - \int_{\Gamma_t} \frac{\widetilde \theta}{\theta} \,\qvect \cdot \mathbf n 
	\notag \\
	\leq - \int_{\Omega_t} \frac{\widetilde \theta}{\theta} \left( \Ss : \nabla_x \uvect  - \frac{\qvect}{\theta} \cdot \nabla_x \theta    \right).
\end{align}
Let us first  add \eqref{energy-1-2} and \eqref{energy-2} and use   the bound \eqref{bound-grad-psi} to absorb the negative term $\disp - \frac{1}{8\pi g}\|\nabla_x \Psi\|^2_{L^2(\Omega_t)}$  in terms of $\rho e(\rho, \theta)$. Then, using  \eqref{boundary-theta-1}, \eqref{def-theta-tilde} and  replacing  $\qvect=-\kappa(\theta)\nabla_x \theta$,  we obtain
\begin{align}\label{energy-3} 
	& \frac{\d}{\dt} \int_{\Omega_t} \left(\frac{1}{2} \rho |\uvect |^2  + \rho e - \rho s \widetilde \theta \right) 
	+ \int_{\Omega_t} \frac{\widetilde \theta}{\theta} \left( \Ss : \nabla_x \uvect 
	+\frac{\kappa(\theta)}{\theta} |\nabla_x \theta|^2    \right)  
	\notag \\
	&	\leq
	-\int_{\Omega_t} \left(\rho[\uvect \otimes \uvect]    : \nabla_x \V - \Ss : \nabla_x \V + p\, \Div_x \V \right) + \int_{\Omega_t} \partial_t(\rho \uvect) \cdot \V 
	\notag \\
	& \ \ \   - \int_{\Omega_t} \rho \nabla_x \Psi \cdot \V - \int_{\Omega_t} \left[ \rho s \left(\partial_t \widetilde \theta +  \uvect \cdot \nabla_x \widetilde \theta\right) - \frac{\kappa(\theta)}{\theta}\nabla_x \theta  \cdot \nabla_x \widetilde \theta \right] ,
\end{align}
where we observe that \eqref{energy-3}  does not contain the boundary heat flux and therefore, it is suitable for the weak formulation.

Till here, we get the ballistic energy for the system \eqref{continuity-eq}--\eqref{pois-eq}, and rewrite it in an explicit way as below:
\begin{align}\label{energy-4} 
	& - \int_0^T \partial_t \psi \int_{\Omega_t} \left(\frac{1}{2} \rho |\uvect |^2  + \rho e - \rho s \widetilde \theta \right) 
	+ \int_0^T \psi \int_{\Omega_t} \frac{\widetilde \theta}{\theta} \left( \Ss : \nabla_x \uvect 
	+\frac{\kappa(\theta)}{\theta} |\nabla_x \theta|^2    \right)  
	\notag \\
	&	\leq \psi(0)  \int_{\Omega_0} \bigg( \frac{1}{2} \frac{|(\rho \uvect)_0|^2}{\rho_0} + \rho_0 e_0(\rho_0, \theta_0) - \rho_0 s(\rho_0, \theta_0) \widetilde \theta(0,\cdot) \bigg) \notag \\   
	& -\int_0^T \psi \int_{\Omega_t} \left(\rho[\uvect \otimes \uvect]    : \nabla_x \V - \Ss : \nabla_x \V + p\, \Div_x \V \right) + \int_0^T \psi \int_{\Omega_t} \partial_t(\rho \uvect) \cdot \V 
	\notag \\
	& \ \ \ -  C\int_0^T \psi \int_{\Omega_t} \rho \nabla_x \Psi \cdot \V 
	- \int_0^T \psi \int_{\Omega_t} \Big[ \rho s \left(\partial_t \widetilde \theta +  \uvect \cdot \nabla_x \widetilde \theta\right) - \frac{\kappa(\theta)}{\theta}\nabla_x \theta  \cdot \nabla_x \widetilde \theta \Big] ,
\end{align}
for any $\psi \in \C^1_c([0,T))$ with $\psi \geq 0$ and $\partial_t \psi \leq 0$.

\begin{defi}\label{def}
	We say that the trio $(\rho,\uvect,\theta )$ is a weak solution  of the problem 
	\eqref{continuity-eq}--\eqref{pois-eq} with boundary
	conditions \eqref{Navier-slip}--\eqref{boundary-Psi}
	and initial conditions \eqref{initial_conditions}--\eqref{initial-energy} if the following items hold:
	\begin{itemize}
		\item $\rho\in L^\infty(0,T;L^{\frac{5}{3}}(\mathbb{R}^3))$
		, $\rho\geq 0$, $\rho\in L^q(Q_T)$ with some certain $q>1$,
		
		\item $\uvect$, $\nabla_x\uvect$ $\in L^2(Q_T)$, $\rho \uvect \in L^\infty(0,T; L^1(\mathbb R^3))$,
		
		\item $\theta>0$ $a.e.$ on $Q_T$, $\theta\in L^\infty(0,T;L^4(\mathbb{R}^3))$,  $\theta, \nabla_x\theta, \log\theta, \nabla_x\log\theta \in L^2(Q_T)$,  and
		\item the relations \eqref{weak-continuity}, \eqref{weak-con-ren}, \eqref{weak_form_momen}, \eqref{weak-poission}, \eqref{energy-4}are satisfied.
	\end{itemize}
\end{defi}

\section{Main result}

Here we state the main theorem of this paper:
\begin{theorem}\label{Main-theorem}
	Assume that $\Omega_0 \subset \mathbb{R}^3$ is a bounded domain of the class $\C^{2+\nu_0}$ for some $\nu_0>0$ and, suppose that $\V\in \C^1([0,T]; \C^3_c(\mathbb{R}^3;\mathbb{R}^3))$ satisfying \eqref{trajec-eq} and the hypothesis in subsection \ref{Section-hypothesis} are satisfied. Then the Naiver-Stokes-Fourier-Poisson system \eqref{continuity-eq}--\eqref{pois-eq} with boundary
	conditions \eqref{Navier-slip}--\eqref{boundary-Psi}
	and initial conditions \eqref{initial_conditions}--\eqref{initial-energy} admits a weak solution in the sense of Definition \ref{def} on any finite time interval $(0,T)$.
\end{theorem}

\section{Penalized problem} \label{sec-penalized}

Let us choose $R>0$ such that 
\begin{align*}
	\V|_{[0,T]\times \{ |x|>R \}} = 0 , \quad \overline \Omega_0 \subset \{ x \in \mathbb R^3 : |x| \leq  R \},
\end{align*}
and then take the reference domain $$\B: = \{x\in \mathbb R^3 : |x| < 2R\}  .  $$

\subsection{Mollification of the coefficients and initial data}\label{Section-hypo-molified}

$\bullet$ The viscosity coefficients are taken as 
\begin{equation}\label{def-mu-omega}
	\mu_\omega (\theta) = f_\omega \mu(\theta) \in \C^\infty_c([0,T]\times \B) , 
\end{equation}
and
\begin{equation}\label{def-nu-omega}
	\eta_\omega (\theta) =  f_\omega\eta(\theta) \in \C^\infty_c([0,T]\times \B) ,
\end{equation}
where the function $f_\omega \in \C^\infty_c([0,T]\times \B)$ such that 
\begin{align}\label{def-f-omega}
	\begin{dcases}
		0< \omega  \leq f_\omega \leq 1 \ \ \ \text{in } \, [0,T] \times \B ,\text{ for } \omega >0,	\\
		f_\omega(t,\cdot) |_{\Omega_t} = 1 \ \text{ for any } t \in [0,T] , \ \ \text{and}\\
		\|f_\omega \|_{L^{p}( ((0,T) \times \B)\setminus Q_T)} \leq c \, \omega \ \, \text{for some } p\geq \frac{\alpha+1}{\alpha-1}  .
	\end{dcases}
\end{align}
From the above definitions, it is  obvious that   
\begin{align}
	\mu_\omega , \  \eta_\omega \to 0 \ \  \text{ a.e. in } \ \ ((0,T)\times \B )\setminus Q_T \ \ \text{ as } \ \omega \to 0.
\end{align}

\vspace*{.1cm}
\noindent 
$\bullet$ We also set the heat conductivity coefficient as follows:
\begin{align}\label{defi-kappa-nu}
	\kappa_\nu (\theta, t, x) = \chi_\nu \kappa(\theta) , 
\end{align}
where $\chi_\nu \in L^\infty((0,T)\times \B)$ such that
\begin{align}\label{def-chi-nu}
	\chi_\nu = 1 \ \ \text{ in } \, Q_T  \ \text{ and } \ \chi_\nu = \nu  \
	\text{ in } \, ((0,T)\times \B) \setminus Q_T \ \ \text{for } \nu >0.
\end{align}

\vspace*{.1cm}
\noindent 
$\bullet$
Similarly, we introduce a variable coefficient $a_\xi:= a_\xi(t,x)$ representing the radiative parts of the pressure, internal energy and entropy, given by 
\begin{align}\label{a-xi}
	a_\xi ( t, x) = \chi_\xi  a,
\end{align} 
with $\chi_\xi\in L^\infty((0,T)\times \B)$ such that 
\begin{align}\label{def-chi-xi}
	\chi_\xi = 1 \ \ \text{ in } \, Q_T  \ \text{ and } \ \chi_\xi = \xi  \
	\text{ in } \, ((0,T)\times \B) \setminus Q_T \ \ \text{for } \xi >0.
\end{align}
We now set 
\begin{align}\label{pressure-pena}
	p_{\xi, \delta} (\rho, \theta) = p_M(\rho, \theta) + \frac{a_\xi}{3} \theta^4  + \delta \rho^\beta , \quad \beta \geq 4, \ \delta>0 ,\\
	\label{energy-entropy-pena}
	e_\xi(\rho, \theta) = e_M(\rho,\theta) + \frac{a_\xi \theta^4}{\rho} , \quad s_\xi(\rho, \theta) = s_M(\rho, \theta) + \frac{4a_\xi \theta^3}{3\rho} . 
\end{align}

\vspace*{.1cm}
\noindent 
$\bullet$
Let us now define the modified initial data $\rho_{0,\delta}$, $(\rho \uvect)_{0,\delta}$ and $\theta_{0,\delta}$.  We consider $\rho_{0,\delta}$ such that
\begin{equation}\label{pena-rho-initial}
	\begin{aligned}
		\rho_{0,\delta} \geq 0, \ \ \rho_{0,\delta} \not\equiv 0 \  \text{ in } \Omega_0, \ \ \rho_{0,\delta} =0 \ \text{ in }  \mathbb R^3\setminus \Omega_0, \ \ \int_{\B}\left(\rho^{\frac{5}{3}}_{0,\delta} + \delta \rho^{\beta}_{0,\delta} \right) \leq c , \\
		\text{and } \ \rho_{0,\delta} \to \rho_0 \ \text{ in } L^{\frac{5}{3}}(\B) \ \text{ as } \delta \to 0 , \ \ \ |\{ \rho_{0,\delta} <\rho_0 \}| \to  0 \ \text{ as } \delta \to 0.
	\end{aligned}
\end{equation}
In above, the constant $c>0$ is independent of the parameter $\delta$.

Next, the initial data for the momentum part is taken in such a way 
\begin{align}
	(\rho \uvect)_{0,\delta} = \begin{dcases}
		(\rho \uvect)_0  \ &\text{if } \rho_{0,\delta} \geq \rho_0 , \\
		0                \       & \text{else}.
	\end{dcases}
\end{align} 

For the temperature part, we consider $0<\underline \theta \leq \theta_{0,\delta} \leq \overline \theta$ with $\theta_{0,\delta} \in L^\infty(\B) \cap \C^{2+\nu_0}(\B)$ for some exponent  $\nu_0\in (0,1)$ where $\underline \theta, \overline \theta$ are positive real numbers as introduced  in Section \ref{section-problm-state}.

Moreover, $\rho_{0,\delta}$ and $\theta_{0,\delta}$ are taken in such a way that 
\begin{align}
	\int_{\Omega_0} \rho_{0,\delta} e(\rho_{0,\delta}, \theta_{0,\delta}) \to  	\int_{\Omega_0} \rho_{0} e(\rho_{0}, \theta_{0}), \ \ \text{and}
\end{align}
\begin{align}\label{initial-entropy-pen}
	\rho_{0,\delta} s(\rho_{0,\delta}, \theta_{0,\delta}) \to   \rho_{0} s(\rho_{0}, \theta_{0}) \ \text{ weakly in } \ L^1(\Omega_0) . 
\end{align}

\subsection{Penalization in the fixed domain and weak formulations}\label{Section-penalized-wk}

We begin this subsection by   shortly describing the strategy of the proof for \Cref{Main-theorem}.
\begin{enumerate}
	\item
	In  the momentum equation, we add the penalized term 
	\begin{align}\label{pena-boundary-mom} 
		\frac{1}{\veps} \int_0^T \int_{\Gamma_t}  (\uvect - \V ) \cdot \mathbf n \ \vphi \cdot \mathbf n \, \ \text{for } \veps > 0 \ \mbox{small},
	\end{align} 
	which was originally proposed by Stokes and Carey in \cite{StoCar}.
	In principle, this allows to deal with the slip boundary conditions. Indeed, as $\veps \to 0$, this additional term yields the boundary condition $(\uvect - \V)\cdot \mathbf n = 0$ on $\Gamma_t$, after reaching some uniform estimates w.r.t. $\veps$.  Accordingly, the reference domain $(0,T) \times \B$ is separated by an impermeable interface $\cup_{t\in (0,T)} \{t\} \times \Gamma_t$ to a {\em fluid domain} $Q_T$ and a {\em solid domain} $((0,T)\times \B) \setminus Q_T$. 
	
	As a matter of fact,  we need to take care the behaviour of the solution in the solid domain. To do so,  we consider the variable coefficients $\mu_\omega, \eta_\omega, \kappa_\nu, a_\xi$  as presented in Section \ref{Section-hypo-molified}. Moreover, similar to the existence theory developed in \cite{Feireisl-Novotny-book}, we  introduce the artificial pressure $p_{\xi,\delta}$ with an extra term $\delta \rho^\beta$ (see \eqref{pressure-pena}), which gives  some  more  (regularity) information about the density.

	\item We  add a term $\lambda \theta^{\alpha+1}$ into the energy balance and $\lambda \theta^{\alpha}$ into the entropy balance, where $\lambda >0$ and $\alpha$ is appearing in  \eqref{hypo-kappa} in the hypothesis of heat conductivity coefficient $\kappa$. These terms yield a control over the temperature in the solid domain. More  precisely, these extra penalized terms help to get rid of some unusual terms in  solid domain while passing to the limit as $\xi, \nu\to 0$.

	\item Keeping $\veps, \omega, \nu, \lambda, \xi$ and $\delta>0$ fixed, 
	we use the existence theory for the compressible
	N-S-F   system with nonhomogeneous  boundary data in the fixed reference domain, developed in \cite{Chauduri-Feireisl} (the part of Poisson equation with the N-S-F can be easily handled in the fixed domain). 
	
	\item  Taking the initial density
	$\rho_0$ vanishing outside $\Omega_0$ and letting $\veps \to 0$
	for fixed $\omega, \nu, \lambda, \xi, \delta > 0$ we obtain a ``two-fluid''
	system where the density vanishes in the solid part $\left((0,T)
	\times \B \right) \setminus Q_T$. 
	Then, in order to get rid of the terms in $\left((0,T)
	\times \B \right) \setminus Q_T$, we 
	tend all other parameters to zero.  To this end, it is required to introduce a proper scaling to let the parameters 
	$\omega, \nu, \xi, \lambda$ to zero simultaneously. This has been rigorously prescribed in Section \ref{Section-Scaling}. Finally,  we let $\delta\to 0$ in a standard fashion, as already used in other related works.
	
\end{enumerate}

Now we are ready to state the weak formulation for the penalized problem.   We consider that the extended $\uvect$  vanishes on the boundary  $(0,T)\times \partial \B$, that is 
\begin{align}
	\uvect|_{\partial \B} = 0, \quad \text{ for all  } t\in (0,T) .
\end{align}

\vspace*{.1cm}

\noindent
{\bf I.} {\bf Continuity equation.}  The weak formulation for the continuity equation reads as 
\begin{align}\label{weak-conti}
	-\int_0^T \int_{\B} \rho B(\rho)\left( \partial_t \vphi +    \uvect \cdot \nabla_x \vphi \right) 
	+	\int_0^T \int_{\B} b(\rho) \Div_x \uvect \vphi 
	=  \int_{\B} \rho_{0,\delta} B(\rho_{0,\delta})\vphi(0, \cdot),
\end{align}
for any  test function $\vphi \in \C^1_c( [0,T)\times \B; \mathbb R)$  and any $b \in L^\infty \cap \C([0,+\infty))$ such that $b(0)=0$ and $\displaystyle B(\rho) = B(1)+ \int_{1}^\rho \frac{b(z)}{z^2}$.

\vspace*{.1cm}
\noindent 
{\bf II.} {\bf Momentum equation.}
The momentum equation is represented by the family of integral identities 
\begin{align}\label{weak-momen}
	-	\int_0^T \int_{\B} \left( \rho \uvect \cdot \partial_t  \boldvphi   +  \rho[\uvect \otimes \uvect] : \nabla_x  \boldvphi + p_{\xi, \delta}(\rho, \theta) \Div_x \boldvphi   \right)   
	+\int_0^T \int_{\B} \Ss_\omega : \nabla_x \boldvphi 
	\notag  \\
	- \int_0^T \int_{\B} \rho \nabla_x \Psi \cdot \boldvphi +  \frac{1}{\veps} \int_0^T \int_{\Gamma_t}  (\uvect -\V) \cdot \mathbf n \ \boldvphi \cdot \mathbf n 
	=
	\int_{\B}   (\rho \uvect)_{0,\delta}\cdot  \boldvphi(0, \cdot) ,
\end{align}
for  any  test function $\boldvphi\in \C^1_c([0,T) \times \B; \mathbb R^3)$  and 
\begin{align}\label{stress_tensor-omega}
	\Ss_\omega(\theta, \nabla_x \uvect) = \mu_\omega(\theta, t,x) \left( \nabla_x \uvect + \nabla^\top_x \uvect -\frac{2}{3} \Div_x \uvect \mathbb I  \right) +\eta_\omega(\theta,t,x) \Div_x \uvect \mathbb I,
\end{align}

\vspace*{.1cm}
\noindent 
{\bf III.} {\bf Poisson equation.} The  weak formulation for the Poisson equation is given by
\begin{align}\label{pena-weak-poisson}
	\int_0^T \int_{\B} \nabla_x \Psi \cdot \nabla_x \vphi = \int_0^T \int_{\B} \rho \vphi , 
\end{align}
for any test function $\vphi \in \C^1((0,T)\times \B; \mathbb R)$, under the assumption that $\rho=0$ outside $\B$.

\vspace*{.1cm}
\noindent 
{\bf IV.} {\bf Entropy inequality.}
Next, we write the penalized  entropy inequality, given by 
\begin{align}\label{penalized-entropy}
	-\int_0^T \int_{\B} \left(\rho s_\xi(\rho,\theta) \left(\partial_t \vphi + \uvect \cdot \nabla_x \vphi\right) - \frac{\kappa_\nu(\theta, t,x)}{\theta} \nabla_x \theta \cdot \nabla_x \vphi  \right) 
	\notag \\
	- \int_{\B} \rho_{0,\delta} s(\rho_{0,\delta}, \theta_{0,\delta}) \vphi(0,\cdot)   
	+ \int_0^T \int_\B \lambda \theta^{\alpha} \vphi 
	\notag \\
	\geq \int_0^T \int_{\B} \frac{\vphi}{\theta}  \left( \Ss_\omega : \nabla_x \uvect +
	\frac{\kappa_\nu(\theta, t,x)}{\theta} |\nabla_x \theta|^2   \right) , 
\end{align}
for any test function $\vphi \in \C^1_c([0,T) \times \B; \mathbb R)$ with $\vphi \geq 0$.

\vspace*{.1cm}
\noindent
{\bf V.} {\bf Ballistic energy inequality.} 
We now write the ballistic energy to the penalized problem as follows:
\begin{align}\label{energy-pena-1-2}
	& -\int_0^T \partial_t \psi \int_{\B} \bigg(\frac{1}{2} \rho |\uvect |^2  + \rho e_\xi(\rho, \theta) - \rho s_\xi(\rho, \theta)  \widetilde \theta +  \frac{\delta}{\beta-1} \rho^\beta \bigg) + \int_0^T \psi \int_\B \lambda \theta^{\alpha+1} 
	\notag \\ 
	& \ + \int_0^T \psi \int_{\B} \frac{\widetilde \theta}{\theta} \bigg( \Ss_\omega : \nabla_x \uvect 
	+\frac{\kappa_\nu(\theta,t,x)}{\theta} |\nabla_x \theta|^2    \bigg) + \frac{1}{\veps} \int_0^T \psi \int_{\Gamma_t} |(\uvect - \V)\cdot \mathbf n|^2   
	\notag \\
	&	\leq  \psi(0) \int_{\B} \bigg(\frac{1}{2}  \frac{|(\rho\uvect)_{0,\delta}|^2}{\rho_{0,\delta}}  
	+ \rho_{0,\delta}e_\xi(\rho_{0,\delta}, \theta_{0,\delta}) - \rho_{0,\delta} s_\xi(\rho_{0,\delta}, \theta_{0,\delta})\widetilde \theta(0,\cdot)  + \frac{\delta}{\beta-1} \rho^\beta_{0,\delta}    \bigg)  
	\notag \\
	& \ + \int_0^T \psi \int_{\B} \lambda \theta^{\alpha} \widetilde \theta +  \int_0^T \psi \int_{\B} \partial_t(\rho \uvect) \cdot \V -\int_0^T \psi \int_{\B} \rho \nabla_x \Psi \cdot \V 
	\notag \\
	& \	- \int_0^T \psi \int_{\B} \Big(\rho[\uvect \otimes \uvect]    : \nabla_x \V - \Ss_\omega : \nabla_x \V + p_{\xi, \delta}(\rho, \theta) \Div_x \V \Big) 
	\notag \\
	&   \  - \int_0^T \psi \int_{\B} \Big[ \rho s_{\xi} \left(\partial_t  \widetilde \theta +  \uvect \cdot \nabla_x  \widetilde \theta \right) - \frac{\kappa_\nu(\theta,t,x)}{\theta}\nabla_x \theta  \cdot \nabla_x \widetilde \theta \Big] ,
\end{align}
for any $\psi \in \C^1_c([0,T))$ with $\psi \geq 0$, $\partial_t \psi \leq 0$.

In above, the function $\widetilde \theta\in  \C^1([0,T]\times \overline \B);\mathbb R)$ satisfies 
\begin{align}\label{choice-tilde-theta-B}
	\inf_{\overline{(0,T)\times \B}} \widetilde \theta >0 ,
\end{align} 
and it can be chosen in the following way:
first we consider $\widetilde \theta$ 
as the  solution to
\begin{align}\label{test-function-theta_B}
	-\Delta_x \widetilde \theta (t, \cdot) = 0 \ \ \text{in } \Omega_t, \quad \widetilde \theta(t, \cdot)|_{\Gamma_t} = \theta_B(t, \cdot), \quad  \text{for each } \, t\in [0,T] ,
\end{align}
where $\theta_B$ is the nonhomogeneous boundary data for the temperature as prescribed in \eqref{boundary-theta-1}.  
Then, we smoothly extend this $\widetilde \theta$  throughout the reference domain $\overline{\B}$  for each $t\in [0,T]$ and  
this particular $\widetilde \theta$ is used  in the ballistic energy inequality \eqref{energy-pena-1-2}.

\vspace*{.1cm}

Moreover, to have such formulation \eqref{energy-pena-1-2} of the energy balance,  we also need to consider that the temperature $\theta$ in the reference domain $(0,T)\times \B$ must satisfy the boundary condition 
\begin{align}\label{boundary-new-theta-B}
	\theta |_{\partial \B} = \widetilde \theta |_{\partial \B} , \quad \text{for all } t \in [0,T]. 
\end{align}

\begin{defi}\label{weak-solution-penalization}
	We say that the trio $(\rho,\uvect,\theta )$ is a weak solution to the penalized problem with initial data \eqref{pena-rho-initial}--\eqref{initial-entropy-pen} if the following items hold:
	\begin{itemize}
		\item $\rho\in L^\infty(0,T;L^{\frac{5}{3}}(\mathbb{R}^3))\cap L^\infty(0,T;L^{\beta}(\mathbb{R}^3))$, $\rho\geq 0$, $\rho\in L^q((0,T)\times \B)$ with some certain $q>1$,
		
		\item $\uvect$, $\nabla_x\uvect$ $\in L^2((0,T)\times \B)$,
		$\rho \uvect \in L^\infty(0,T; L^1(\B))$,
		
		\item $\theta>0$ $a.e.$ on $Q_T$, $\theta\in L^\infty(0,T;L^4(\B))$, $\theta, \nabla_x\theta, \log\theta, \nabla_x\log\theta \in L^2((0,T)\times \B)$, and  
		
		\item the relations \eqref{weak-conti}, \eqref{weak-momen}, \eqref{pena-weak-poisson}, \eqref{penalized-entropy},
		\eqref{energy-pena-1-2} are satisfied.
	\end{itemize}
\end{defi}

\begin{theorem}\label{weak-solution-fixed-domain}
	Assume that $\V\in \C^1([0,T]; \C_c^3(\mathbb{R}^3;\mathbb{R}^3 ))$ and, suppose that  the hypotheses in subsections \ref{Section-hypothesis}, \ref{Section-hypo-molified},  and the equations of states are satisfied. Moreover, the initial data satisfy \eqref{pena-rho-initial}--\eqref{initial-entropy-pen}.
	Then there exists a weak solution to the penalized problem on any time interval $(0, T)$ in the sense of Definition \ref{weak-solution-penalization}.  
\end{theorem}

\begin{proof}
	Here we just give a short explanation on the proof. The existence of weak solution with the non-homogeneous Dirichlet condition for the temperature in the fixed domain is similar to \cite{Chauduri-Feireisl}. It is necessary to  consider the continuity equation with a viscous term $\Delta_x\rho$, solving the momentum equations via Faedo-Galerkin approximations and the Poisson equation. Instead of pursuing the solution to the entropy equation, we look for the solution to the internal energy equation. Note that \eqref{energy-pena-1-2} is adopted by dealing with the particular non-homogeneous boundary conditions.
	As pointed out in \cite[Theorem 3.1]{Sarka-et-al-ZAMP}, here also we face the following difficulties. 
	\begin{itemize}
		\item The penalized terms  $\disp \frac{1}{\veps} \int_0^T \int_{\Gamma_t}  (\uvect -\V) \cdot \mathbf n \ \boldvphi \cdot \mathbf n $ in \eqref{weak-momen} and $\disp \frac{1}{\veps} \int_0^T \psi \int_{\Gamma_t} |(\uvect - \V)\cdot \mathbf n|^2 $ in   \eqref{energy-pena-1-2}. 
		
		\item The jumps in functions $\disp \kappa_\nu (\theta, t, x)$ in \eqref{defi-kappa-nu}, and $\disp a_\xi( t, x)$ in \eqref{a-xi}.

	\end{itemize}
	
	The strategy to overcome these difficulties  has already been discussed in the beginning of the proof of Theorem 3.1 in \cite{Sarka-et-al-ZAMP}. In the present work, we employ 
	similar methodology for the proof.  We  emphasize that the term $\lambda\theta^{\alpha+1}$ ($\alpha>6$)  is necessary for our modified internal energy equation to provide uniform bounds of high power of the temperature on $\B$. Moreover, additional difficulty will arise to get a proper  bound of the term $\rho s_\xi(\rho, \theta) \uvect$ appearing in right hand side of the energy inequality \eqref{energy-pena-1-2}.
	
	A detailed study of the existence of such  approximated solutions can be
	found in \cite{KMNPW-L}.
\end{proof}

The rest of the paper is devoted to prove the main result of this paper, that is, \Cref{Main-theorem}.

\subsection{Uniform bounds}

This  subsection is devoted to  establish the uniform bounds for the weak solution which is constructed by virtue of Theorem \eqref{weak-solution-fixed-domain}.

Let us define the following Helmholtz-type function (see \cite[Chapter 2.2.3]{Feireisl-Novotny-book}):
\begin{align}\label{H-B}
	\mathcal	H_{\widetilde \theta, \xi}(\rho, \theta) : =
	\rho e_\xi(\rho, \theta) - \rho s_\xi(\rho, \theta) \widetilde \theta ,
\end{align}
where $\widetilde \theta \in \C^1([0,T]\times \overline \B ; \mathbb R)$ has been introduced in the previous section (see 
\eqref{choice-tilde-theta-B}--\eqref{boundary-new-theta-B}). 

Accordingly, we denote 
\begin{align}\label{H-B-0}
	\mathcal	H_{\widetilde \theta, \xi}(\rho_{0,\delta}, \theta_{0,\delta}) = \rho_{0,\delta} e_\xi(\rho_{0,\delta}, \theta_{0,\delta}) - \rho_{0,\delta} s_\xi(\rho_{0,\delta}, \theta_{0,\delta}) \widetilde \theta(0,\cdot) . 
\end{align}
We now consider $\psi_\zeta\in \C^1_c([0,T))$ with
\begin{align*}
	\psi_\zeta(t) = 
	\begin{dcases} 1 \quad &\text{for } t < \tau - \zeta , \\
		0   \quad & \text{for } t \geq \tau ,     
	\end{dcases}  \ \ \text{ for any given } \tau \in (0,T) , \ 0< \zeta < \tau   ,
\end{align*}
and using it as a test function in \eqref{energy-pena-1-2}, we derive after passing to the limit $\zeta \to 0$, that
\begin{align}\label{energy-pena-3}
	&	\int_{\B} \left(\frac{1}{2} \rho |\uvect |^2  + \mathcal H_{\widetilde \theta, \xi}(\rho, \theta)  
	+ \frac{\delta}{\beta-1} \rho^\beta\right)(\tau, \cdot)  	+ \frac{1}{\veps} \int_0^\tau \int_{\Gamma_t} |(\uvect -\V) \cdot \mathbf n |^2
	\notag \\
	&
	+ \int_0^\tau \int_\B \lambda \theta^{\alpha+1}  
	+ \int_0^\tau \int_{\B} \frac{\widetilde \theta}{\theta} \left( \Ss_\omega : \nabla_x \uvect + \frac{\kappa_\nu(\theta,t,x)}{\theta} |\nabla_x \theta|^2     \right) 
	\notag \\	
	\leq  &
	\int_{\B} \bigg(\frac{1}{2}  \frac{|(\rho\uvect)_{0,\delta}|^2}{\rho_{0,\delta}}  	+ \mathcal 	H_{\widetilde \theta, \xi}(\rho_{0,\delta}, \theta_{0,\delta})  + \frac{\delta}{\beta-1} \rho^\beta_{0,\delta}  \bigg) + \int_0^\tau \int_\B \lambda \theta^{\alpha} \, \widetilde \theta 
	\notag \\
	&  + \int_\B (\rho \uvect \cdot \V)(\tau, \cdot)  - \int_\B (\rho \uvect)_{0,\delta} \V(0,\cdot) - \int_0^\tau \int_\B\rho \nabla_x \Psi \cdot \V
	\notag \\
	& 	-\int_0^\tau \int_{\B} \Big(\rho[\uvect \otimes \uvect]    : \nabla_x \V - \Ss_\omega : \nabla_x \V + p_{\xi,\delta}(\rho,\theta)\, \Div_x \V   + \rho \uvect \cdot \partial_t \V    \Big) 
	\notag \\
	&    -	\int_0^\tau \int_{\B} \Big[\rho s_{\xi}(\rho,\theta) \Big(\partial_t \widetilde \theta  + \uvect \cdot \nabla_x \widetilde \theta  \Big) - \frac{\kappa_\nu(\theta, t, x)}{\theta} \nabla_x \theta \cdot \nabla_x \widetilde{\theta}   \Big]  ,
\end{align}
for almost all  $\tau \in (0,T)$.

\vspace*{.1cm}
Let us now find the uniform bounds of the right hand side of the modified energy inequality \eqref{energy-pena-3}. 
First, we recall that the fluid system  satisfies the mass conservation law, that is 
\begin{align*}
	\int_{\B} \rho(\tau, \cdot ) = \int_\B \rho_{0,\delta} (\cdot) = \int_{\Omega_0} \rho_0(\cdot) = C(\rho_0) >0.
\end{align*}
Keeping that in mind, we proceed to find the estimates. 

\vspace*{.1cm}
\noindent 
{$\bullet$ \bf Step 1.}  
(i)  For any $\epsilon>0$ small, we have 
\begin{align}\label{esti-1}
	\int_\B  (\rho \uvect \cdot \V)(\tau,\cdot)  \leq C(\V) 	\left| \int_\B  \sqrt{\rho} \sqrt{\rho} \uvect(\tau,\cdot) \right| \leq C(\V, \rho_0)  + \epsilon \int_\B \rho |\uvect|^2 .
\end{align} 

\vspace*{.1cm}
\noindent 
(ii) 
Next, recall how we obtain \eqref{bound-rho-psi-V}, and we reach to the following:
\begin{align}\label{esti-nabla-Psi-V}
	\left|\int_0^\tau \int_\B \rho \nabla_x \Psi \cdot \V \right| &\leq C(\V)\int_0^\tau \int_\B \rho^{\frac{5}{3}} + C(\V,\rho_0, g) \notag \\
	&\leq C(\V, p_\infty, g)\bigg( \int_0^\tau \int_\B \rho e_\xi(\rho, \theta) + 1 \bigg).
\end{align}

\vspace*{.1cm}
\noindent 
(iii) 
Without loss of generality we assume $0<\lambda\leq 1$ from now onwards.  Then, by using H\"older's and Cauchy-Schwarz inequality, we obtain 
\begin{align}\label{esti-2} 
	\int_0^\tau \int_\B \Ss_\omega : \nabla_x \V 
	&\leq   \frac{1}{2} \int_0^\tau \int_\B  \frac{\widetilde{\theta}}{\theta} \Ss_\omega : \nabla_x \uvect + C(\V, \widetilde{\theta}) \int_0^\tau \int_\B \theta \notag  \\
	& \leq \frac{1}{2} \int_0^\tau \int_\B  \frac{\widetilde{\theta}}{\theta} \Ss_\omega : \nabla_x \uvect + 
	\epsilon \int_0^\tau \int_\B \lambda \theta^{\alpha+1} +  \frac{C(\V, \widetilde{\theta}, \epsilon)}{\lambda^{1/{\alpha}}} .
\end{align}
We also have that 
\begin{align}\label{esti-3}
	\left| \int_0^\tau \int_\B\rho[\uvect \otimes \uvect]  : \nabla_x \V  \right| \leq  C(\V)\int_0^\tau \int_\B \rho |\uvect|^2  , \ \ \text{and}
\end{align}
\begin{align}\label{esti-4}
	\left| \int_0^\tau \int_\B\rho \uvect \cdot \partial_t \V  \right| \leq  C(\V, \rho_0) + C \int_0^\tau \int_\B \rho |\uvect|^2  .
\end{align}

\vspace*{.1cm}
\noindent 
(iv) 
Next, since $0<\lambda\leq 1$,  it is easy to observe that
\begin{align}\label{esti-5}
	\int_0^\tau \int_\B \lambda \theta^{\alpha} \, \widetilde{\theta} \leq \frac{C (\widetilde{\theta})}{\epsilon} + \epsilon \int_0^\tau \int_\B \lambda \theta^{\alpha+1} .
\end{align}

\vspace*{.1cm}
\noindent 
(v) 
The pressure term  $p_{\xi, \delta}(\rho, \theta)$ in \eqref{pressure-pena} can be estimated as follows. First, we recall the point \eqref{fact-1} which indeed tells that $P^\prime(Z)>0$ for all $Z>0$.
Further, we  recall the fact \eqref{fact-2-1} which gives
$\disp \lim_{Z\to \infty} \frac{P(Z)}{Z^{\frac{5}{3}}} = p_\infty > 0$.
Therefore, we obtain the following bounds on the molecular pressure $p_M$, 
\begin{equation}\label{mole_bound_p_M}
	\begin{aligned} 
		\underline c \rho \theta \leq & p_M \leq \overline c \rho \theta      \quad  \text{if } \rho < \overline Z \theta^{\frac{3}{2}}, \\
		\underline c \rho^{\frac{5}{3}} \leq & p_M \leq 
		\begin{cases}
			\overline c \theta^{\frac{5}{2}}  \quad  \text{if } \rho < \overline Z \theta^{\frac{3}{2}} , \\
			\overline c	\rho^{\frac{5}{3}} \quad  \text{if } \rho > \overline Z \theta^{\frac{3}{2}} ,
		\end{cases}
	\end{aligned}
\end{equation}
and $p_M$ is monotone in $\underline Z \theta^{\frac{3}{2}} \leq \rho \leq \overline Z \theta^{\frac{3}{2}}$. 

With the above information, we deduce that 
\begin{align}\label{estimate-pressue-term}
	\left|\int_0^\tau \int_\B p_{\xi, \delta} (\rho, \theta) \Div_x \V \right| 
	\leq C(\V) \int_0^\tau \int_\B \frac{\delta}{\beta-1} \rho^\beta +  C(\V) \int_0^\tau \int_\B a_{\xi}  \theta^4 \notag \\ 
	+ C(\V) \int_0^\tau \int_\B \rho^{\frac{5}{3}} 
	+ \epsilon \int_0^\tau \int_\B \lambda \theta^{\alpha+1} + \frac{C(\V, \epsilon)}{\lambda^{5/(2\alpha-3)}} .
\end{align}
In fact, we have 
that 
\begin{align}\label{lower_bound_rho_e_xi}
	\rho e_\xi \geq a_\xi \theta^4 + \frac{3}{2} p_\infty \rho^{\frac{5}{3}}	,
\end{align}
which can be shown in the same way as we have obtained \eqref{lower_bound_rho_e}, and 
therefore, 
\begin{align*}
	\int_0^\tau \int_\B \left(a_\xi \theta^4 + \rho^{\frac{5}{3}} \right)  \leq C(p_\infty) \int_0^\tau \int_\B \rho e_\xi .
\end{align*}

Using the above inequality in \eqref{estimate-pressue-term} and  together with all other estimates above, we have from \eqref{energy-pena-3}  (by fixing $\epsilon>0$ small enough),
\begin{align}\label{energy-pena-4}
	&	\int_{\B} \left(\frac{1}{2} \rho |\uvect |^2 + 	\mathcal H_{\widetilde{\theta}, \xi}(\rho, \theta) + \frac{\delta}{\beta-1} \rho^\beta\right)(\tau, \cdot) + \frac{1}{\veps} \int_0^\tau \int_{\Gamma_t} |(\uvect -\V) \cdot \mathbf n |^2   \notag  \\
	&	
	\ \ 	+ \int_0^T \int_\B \lambda \theta^{\alpha+1}	  
	+ \int_0^\tau \int_{\B} \frac{\widetilde \theta}{\theta} \left( \Ss_\omega : \nabla_x \uvect  + \frac{\kappa_\nu(\theta,t,x)}{\theta} |\nabla_x \theta|^2    \right) 
	\notag \\
	\leq  &
	\int_{\B} \bigg(\frac{1}{2}  \frac{|(\rho\uvect)_{0,\delta}|^2}{\rho_{0,\delta}} + \mathcal  	H_{\widetilde{\theta}, \xi}(\rho_{0,\delta}, \theta_{0,\delta})  + \frac{\delta}{\beta-1} \rho^\beta_{0,\delta}  - (\rho \uvect)_{0,\delta} \V(0,\cdot) \bigg) 
	\notag \\
	&	+ C \int_0^\tau \int_\B \bigg(\frac{1}{2}\rho |\uvect|^2   + \rho e_\xi (\rho, \theta) 
	+  \frac{\delta}{\beta-1} \rho^\beta \bigg) 
	\notag \\
	& + \bigg|	\int_0^\tau \int_{\B} \Big[\rho s_{\xi}(\rho,\theta) \left(\partial_t \widetilde{\theta}  + \uvect \cdot \nabla_x \widetilde{\theta}  \right) - \frac{\kappa_\nu(\theta, t, x)}{\theta} \nabla_x \theta \cdot \nabla_x \widetilde{\theta}   \Big]  \bigg|  \notag \\
	&  + C \Big(1+\frac{1}{\lambda^{5/(2\alpha-3)}  } \Big) ,
\end{align}
for almost all $\tau \in (0,T)$, where $\mathcal H_{\widetilde{\theta}, \xi}(\rho, \theta)$ and $\mathcal H_{\widetilde{\theta}, \xi}(\rho_{0,\delta}, \theta_{0,\delta})$ are defined by \eqref{H-B} and \eqref{H-B-0} respectively and $C>0$ is some constant that may depend on the quantities  $\V$, $\rho_0$, $p_\infty$, $g$ and $\widetilde{\theta}$ but not on the parameters $\lambda$, $\omega$, $\xi$, $\nu$, $\veps$ or $\delta$.

\vspace*{.2cm}
\noindent 
{$\bullet$ \bf Step 2.}   (i)
To the next, we recall the expression of $\Ss_\omega$ from \eqref{stress_tensor-omega} and using \eqref{def-mu-omega}, \eqref{def-nu-omega} and \eqref{hypo-mu},
we obtain 
\begin{align}\label{lower-bound-S-w-u}
	\int_0^\tau \int_\B \frac{\widetilde{\theta}}{\theta} \Ss_\omega : \nabla_x \uvect \geq c_1 (\omega) \inf_{\overline{(0,T)\times \B}} |\widetilde{\theta}| 
	\int_0^\tau \int_\B \Big|\nabla_x \uvect + \nabla^t_x \uvect -\frac{2}{3}\Div_x \uvect \mathbb I   \Big|^2,
\end{align}
for some constant $c_1(\omega)>0$. 

On the other hand, by the Korn-Poincar\'e inequality (see Lemma \ref{Korn-Poincare}) we have 
\begin{align*}
	\|\uvect\|^2_{W^{1,2}(\B; \mathbb R^3)} &\leq C \Big\|\nabla_x \uvect + \nabla^t_x \uvect -\frac{2}{3}\Div_x \uvect \mathbb I   \Big\|^2_{L^2(\B; \mathbb R^3)} + C \Big(\int_\B \rho |\uvect| \Big)^2 \\
	& \leq  C \Big\|\nabla_x \uvect + \nabla^t_x \uvect -\frac{2}{3}\Div_x \uvect \mathbb I   \Big\|^2_{L^2(\B; \mathbb R^3)} + C(\rho_0)\int_\B \rho |\uvect|^2 .
\end{align*} 
Therefore, 
\begin{align}\label{u_W_12}
	c_1(\omega) c(\widetilde\theta) \int_0^\tau	\|\uvect\|^2_{W^{1,2}(\B; \mathbb R^3) } \leq   \int_0^\tau \int_\B \frac{\widetilde{\theta}}{\theta} \Ss_\omega : \nabla_x \uvect + C(\rho_0, \widetilde{\theta}) \int_0^\tau \int_\B \rho |\uvect|^2,
\end{align}
for some constant $c(\widetilde \theta), C(\rho_0, \widetilde\theta)>0$. 

\begin{remark}\label{remark-c_1-c_2-omega}
	From the definition of $\mu_\omega$ and $\eta_\omega$ 
	it is clear that the constant $c_1(\omega)$ 
	behaves like $``c  \omega"$ in $\B\setminus \Omega_t$ for some constant $c>0$ which is independent in $\omega$.
\end{remark}

\noindent 
(ii) We still need to estimate of the last couple of terms containing $\widetilde{\theta}$ in the right hand side of \eqref{energy-pena-4}. Recall the definition of $\kappa_\nu$ from \eqref{defi-kappa-nu}, 
we get 
\begin{equation}\label{esti-aux-new}
	\begin{aligned}
		\bigg|\int_0^\tau \int_\B\frac{\kappa_\nu(\theta)}{\theta} \nabla_x \theta \cdot \nabla_x \widetilde{\theta}\bigg| 
		\leq 
		\frac{1}{2}\int_0^\tau \int_{\B} \widetilde{\theta} \frac{\kappa_\nu(\theta)}{\theta^2}|\nabla_x \theta|^2 + C(\widetilde{\theta}) \int_0^\tau \int_\B \kappa_\nu(\theta)  ,
	\end{aligned}
\end{equation} 
where the first integral can be absorbed by the associated leading term in the l.h.s. of \eqref{energy-pena-4}. 

Now, using the hypothesis on $\kappa_\nu(\theta)$, the second integral in \eqref{esti-aux-new}  can be estimated as follows:  
\begin{align}\label{bound-kappa-theta-new}
	C(\widetilde{\theta}) \int_0^\tau \int_\B \kappa_\nu(\theta) &\leq  C(\widetilde{\theta}) + C(\widetilde{\theta}) c(\nu)  \int_0^\tau \int_\B \theta^\alpha \notag \\
	& \leq C(\widetilde{\theta}) + \epsilon \int_0^\tau \int_\B \lambda \theta^{\alpha+1} + C(\widetilde{\theta}, \epsilon) \frac{(c(\nu))^{\alpha+1}}{\lambda^\alpha}, 
\end{align} 
for any $\eps>0$ small enough, where $c(\nu)\simeq c \nu$ on $\B\setminus \Omega_t$ and $c(\nu)\simeq c$ for some constant $c>0$ independent in $\nu$ (this can be seen from the choice of $\kappa_\nu(\theta)$).

\vspace*{.1cm}
\noindent (iii)
We now need to find a proper bound of $\disp \rho s_{\xi}(\rho, \theta) |\uvect|$.
Recall \eqref{third-law}, there exists some $c>0$ such that 
\begin{align}\label{Fact-1}  
	s_M(\rho,\theta)	= S\left(\frac{\rho}{\theta^{\frac{3}{2}}}\right)  \leq c  \quad \text{when } \frac{\rho}{\theta^{3/2}} > 1 ,
\end{align}
and therefore, 
\begin{align}\label{Fact-2} 
	\rho s_\xi (\rho, \theta) = \rho s_M(\rho, \theta) + \frac{4a_\xi}{3}\theta^3 \leq c\rho + \frac{4a_\xi}{3}\theta^3, \quad \text{ for } \frac{\rho}{\theta^{3/2}} >1 .
\end{align}

On the other hand, when $\disp \frac{\rho}{\theta^{3/2}} \leq 1$, we use the strategy developed in \cite[Section 4, formula (4.6)]{Feireisl2012weak} and  according to that, one has  (using the Gibb's relation \eqref{equ-gibbs}, the hypothesis \eqref{hypo-press}--\eqref{hypo-p_m_with_P-2} and \eqref{Gibbs-1})
\begin{align*}
	s_M(\rho,\theta) \leq 	C (1+ |\log \rho| + [\log \theta]^+) .
\end{align*}
This yields
\begin{equation} 
	\begin{aligned}\label{Fact-3-1}
		\rho s_{\xi}(\rho, \theta) = \rho s_M(\rho, \theta) + \frac{4a_\xi}{3}\theta^3
		\leq 	C	\left(\rho + |\rho \log \rho|  + |\rho| [\log \theta]^+ \right) + \frac{4a_\xi}{3}\theta^3  .  
	\end{aligned}
\end{equation}
Now, observe that 
\begin{align}\label{bound-rho-log-rho}
	|\rho \log \rho| \leq 
	\begin{dcases}
		C \rho^{\frac{1}{2}}, \ \ \ \text{when } 0< \rho \leq 1, \\
		\frac{3}{2} \rho [\log \theta]^+, \ \ \ \text{when } \rho >1 \ \ (\text{consequently $\theta>1$ since $\frac{\rho}{\theta^{3/2}} \leq 1$}),
	\end{dcases}
\end{align} 
where we have used the fact that $|\rho^{\frac{1}{2}} \log \rho|$ is bounded for $0<\rho \leq 1$.

Using \eqref{bound-rho-log-rho} in \eqref{Fact-3-1}, we get
\begin{align}\label{Fact-3}
	\rho s_\xi(\rho, \theta) \leq C \Big(  \rho + \rho^{\frac{1}{2}} + \rho [\log \theta]^+ \Big) + \frac{4a_\xi}{3}\theta^3
	, \quad  \text{ for } \frac{\rho}{\theta^{3/2}} \leq 1 .
\end{align}
Thus, we get  
\begin{align}\label{fact-rho-s}
	&\int_0^\tau 	\int_\B \rho s_{\xi}(\rho, \theta) |\uvect| \notag \\
	& \leq  C \int_0^\tau \int_\B \rho |\uvect| + C \int_0^\tau \int_\B \rho^{\frac{1}{2}} |\uvect| +  C \int_0^\tau \int_\B \rho |\uvect|[\log \theta]^+  + C \int_0^\tau \int_\B a_\xi \theta^3 |\uvect|   \notag \\
	&\leq  C(\rho_0)  + C\int_0^\tau \int_\B \rho |\uvect|^2 +  C \int_0^\tau \int_\B \rho ([\log \theta]^+)^2 +  \frac{c_1(\omega) c(\widetilde\theta) }{2} \int_0^\tau \int_\B |\uvect|^2  \notag \\
	& \ \ \ +  \frac{1}{2c_1(\omega)c(\widetilde\theta)}  \int_0^\tau \int_\B  a_\xi^2 \theta^6  \notag \\
	&\leq  C(\rho_0)  + C(\rho_0, \widetilde \theta)\int_0^\tau \int_\B \rho |\uvect|^2 +  C \int_0^\tau \int_\B \rho ([\log \theta]^+)^2 +  \frac{1}{2} \int_0^\tau \int_\B \frac{\widetilde \theta}{\theta} \Ss_\omega : \nabla_x \uvect   \notag \\
	& \ \ \ +  \frac{1}{2c_1(\omega)c(\widetilde\theta)}  \int_0^\tau \int_\B  a_\xi^2 \theta^6,
\end{align}
where we have used \eqref{Fact-3} and \eqref{u_W_12}, and the constants   $c_1(\omega)$, $c(\widetilde \theta)$ are appearing in \eqref{u_W_12}.

Here,  we observe that the third term in the last inclusion is arising due to the case when $\disp\frac{\rho}{\theta^{3/2}}\leq 1$. Keeping in mind this point, we have  
\begin{align}\label{fact-rho-s-1}
	\int_0^\tau \int_\B \rho ([\log \theta]^+)^2 
	\leq 
	\int_0^\tau \int_\B \theta^{\frac{3}{2}} ([\log \theta]^+)^2 
	&\leq  	
	\int_0^\tau \int_\B \theta^{\frac{5}{2}} \notag \\ 
	&	\leq \epsilon \int_0^\tau \int_\B \lambda \theta^{\alpha+1}  + \frac{C(\epsilon)}{\lambda^{5/(2\alpha-3)}} 
\end{align} 
for any chosen $\epsilon>0$  (since $[\log \theta]^+\leq \theta^{1/2}$).

Moreover, we find (since $\alpha>6$)
\begin{equation} 
	\begin{aligned}\label{fact-rho-s-3} 
		\frac{1}{2c_1(\omega)c(\widetilde \theta)}  \int_0^\tau \int_\B a_\xi^2 \theta^6 
		\leq
		\epsilon \int_0^\tau  \int_\B \lambda \theta^{\alpha+1} + C(\epsilon, \widetilde \theta)\int_0^\tau \int_\B \bigg(\frac{a^2_\xi}{c_1(\omega)}  \lambda^{-\frac{6}{\alpha+1}} \bigg)^{\frac{\alpha+1}{\alpha-5}} 
	\end{aligned}
\end{equation}
for any given $\epsilon>0$.

Therefore, by using    \eqref{fact-rho-s-1} and \eqref{fact-rho-s-3} in \eqref{fact-rho-s},  we obtain 
\begin{align}\label{Fact-4}
	&\int_0^\tau 	\int_\B \rho s_{\xi}(\rho, \theta) |\uvect| \notag \\
	& \leq 
	C(\rho_0) + C(\rho_), \widetilde \theta)\int_0^\tau \int_\B \rho |\uvect|^2 + 
	\frac{1}{2} \int_0^\tau \int_\B \frac{\widetilde \theta}{\theta} \Ss_\omega : \nabla_x \uvect 
	+  	
	C\epsilon \int_0^\tau \int_\B \lambda \theta^{\alpha+1} \notag \\ 
	& \  \  +   \frac{C(\epsilon)}{\lambda^{5/(2\alpha-3)}} 
	+ C(\epsilon, \widetilde \theta)
	\int_0^\tau \int_\B \bigg(\frac{a^2_\xi}{c_1(\omega)  \lambda^{\frac{6}{\alpha+1}}} \bigg)^{\frac{\alpha+1}{\alpha-5}} .
\end{align}

\vspace*{.1cm}
\noindent 
(iv) We finally observe that 
\begin{align}\label{norm-kappa-theta-1}
	\int_\B \widetilde{\theta} \, \kappa_\nu (\theta) \frac{|\nabla_x \theta|^2}{|\theta|^2} 
	\geq \underline{\kappa} c_2(\nu) \inf_{\overline{(0,T)\times \B}} |\widetilde{\theta}|  \int_\B   \big(\theta^{-2} + \theta^{\alpha-2} \big) |\nabla_x \theta|^2 ,
\end{align}  
in the l.h.s. of \eqref{energy-pena-4},	for some  constant $c_2(\nu)>0$ which behaves like  $``c\nu"$ in $\B\setminus \Omega_t$ for some constant $c>0$ that is independent in $\nu$ and in $\Omega_t$, $c_2(\nu)$ does not depend on $\nu$ since $\chi_\nu=1$ in $\Omega_t$ for any $t\in [0,T]$.

\vspace*{.2cm}
\noindent 
$\bullet$ Collecting the bounds \eqref{esti-aux-new}--\eqref{bound-kappa-theta-new} and \eqref{norm-kappa-theta-1}, the bound \eqref{Fact-4} along with \eqref{lower-bound-S-w-u}--\eqref{u_W_12},
we have from \eqref{energy-pena-4} (by fixing  $\epsilon>0$  small enough) that
\begin{align}\label{energy-pena-4-2}
	&	\int_{\B} \left(\frac{1}{2} \rho |\uvect |^2 +  \mathcal H_{\widetilde{\theta}, \xi}(\rho, \theta) + \frac{\delta}{\beta-1} \rho^\beta\right)(\tau, \cdot)   
	+ \frac{1}{\veps} \int_0^\tau \int_{\Gamma_t} |(\uvect -\V) \cdot \mathbf n |^2 \notag 
	\\
	&+ \int_0^T \int_\B \lambda \theta^{\alpha+1}  
	+ \int_0^\tau \int_\B \frac{\widetilde \theta}{\theta} \Ss_\omega : \nabla_x \uvect  + c_2(\nu) \int_0^\tau \int_\B \big( \theta^{-2} + \theta^{\alpha-2} \big)|\nabla_x \theta|^2
	\notag \\
	&\leq   
	\int_{\B} \bigg(\frac{1}{2}  \frac{|(\rho\uvect)_{0,\delta}|^2}{\rho_{0,\delta}}  + \mathcal H_{\widetilde{\theta}, \xi}(\rho_{0,\delta}, \theta_{0,\delta})  + \frac{\delta}{\beta-1} \rho^\beta_{0,\delta}  - (\rho \uvect)_{0,\delta} \V(0,\cdot) \bigg) 
	\notag \\
	& \	\ + C \int_0^\tau \int_\B \bigg( \frac{1}{2}\rho |\uvect|^2    + \rho e_\xi (\rho, \theta) 
	+  \frac{\delta}{\beta-1} \rho^\beta\bigg) \notag \\
	& \ \ + C \bigg(1+\frac{1}{\lambda^{5/(2\alpha-3)}} + \frac{\nu^{\alpha+1}}{\lambda^\alpha}+\bigg(\frac{\xi^{2}}{\omega \lambda^{6/(\alpha+1)} } \bigg)^{\frac{\alpha+1}{\alpha-5}}  \bigg) , 
\end{align}
for almost all  $\tau \in (0,T)$, and the constant $C>0$  may depend on the quantities  $\V$, $\rho_0$, $p_\infty$, $g$ and $\widetilde{\theta}$ but not on the parameters $\lambda$, $\omega$, $\xi$, $\nu$, $\veps$ or $\delta$.

Applying the
Gr\"onwall's inequality in \eqref{energy-pena-4-2}, we  deduce that
\begin{align}\label{energy-pena-5}
	&	\int_{\B} \left(\frac{1}{2} \rho |\uvect |^2 +  \mathcal H_{ \widetilde{\theta}, \xi}(\rho, \theta) + \frac{\delta}{\beta-1} \rho^\beta\right)(\tau, \cdot)   
	+ \frac{1}{\veps} \int_0^\tau \int_{\Gamma_t} |(\uvect -\V) \cdot \mathbf n |^2   	\notag \\  
	&\ 
	+	\int_0^\tau \int_\B \lambda \theta^{\alpha+1} 
	+  \int_0^\tau \int_\B \frac{\widetilde \theta}{\theta} \Ss_\omega : \nabla_x \uvect + c_2(\nu) \int_0^\tau \int_\B \big( \theta^{-2} + \theta^{\alpha-2} \big)|\nabla_x \theta|^2
	\notag \\
	&\leq 
	C	\int_{\B} \bigg(\frac{1}{2}  \frac{|(\rho\uvect)_{0,\delta}|^2}{\rho_{0,\delta}}  + \mathcal H_{\widetilde{\theta}, \xi}(\rho_{0,\delta}, \theta_{0,\delta})  + \frac{\delta}{\beta-1} \rho^\beta_{0,\delta}  - (\rho \uvect)_{0,\delta} \V(0,\cdot) \bigg) \notag \\
	& \ \ +C \bigg(1+\frac{1}{\lambda^{5/9}} + \frac{\nu^{\alpha+1}}{\lambda^\alpha} + \bigg(\frac{\xi^{2}}{\omega \lambda^{6/(\alpha+1)} } \bigg)^{\frac{\alpha+1}{\alpha-5}}  \bigg) ,
\end{align}
for almost all $\tau \in (0,T)$ and $C>0$ constant which has been specified in  \eqref{energy-pena-4-2}.  In above, we have used the following facts:
since $0<\lambda \leq 1$ and $\alpha>6$, one has
$$ \frac{1}{\lambda^{5/(2\alpha-3)}} < \frac{1}{\lambda^{5/9}} \quad \text{as } \ \frac{5}{(2\alpha-3)} < \frac{5}{9} .$$

To ensure that the left hand side of \eqref{energy-pena-5} is positive,  we proceed as follows. Setting a constant $\overline \rho$ such that $\disp \int_\B (\rho-\overline \rho)=0$ for almost all $\tau \in [0,T)$ and we rewrite the estimate \eqref{energy-pena-5} as the total dissipation inequality:
\begin{align}\label{energy-pena-6}
	&	\int_{\B} \left(\frac{1}{2} \rho |\uvect |^2 +	\mathcal H_{\widetilde{\theta}, \xi}(\rho, \theta) - 
	(\rho-\overline \rho) \frac{\partial\mathcal  H_{\widetilde{\theta} , \xi}(\overline \rho, \widetilde{\theta}) }{\partial \rho} - \mathcal H_{\widetilde{\theta}  , \xi}(\overline \rho, \widetilde{\theta}) + \frac{\delta}{\beta-1} \rho^\beta\right)(\tau, \cdot)  \notag  \\
	& \ 	+ \frac{1}{\veps} \int_0^\tau \int_{\Gamma_t} |(\uvect -\V) \cdot \mathbf n |^2 
	+	\int_0^\tau \int_\B \lambda \theta^{\alpha+1}	  
	\notag \\
	& 
	\ +  \int_0^\tau \int_\B \frac{\widetilde \theta}{\theta} \Ss_\omega : \nabla_x \uvect+c_2(\nu) \int_0^\tau \int_\B \big( \theta^{-2} + \theta^{\alpha-2} \big)|\nabla_x \theta|^2  \notag
	\\
	&\leq 
	C	\int_{\B} \bigg(\frac{1}{2}  \frac{|(\rho\uvect)_{0,\delta}|^2}{\rho_{0,\delta}}  +  \mathcal H_{\widetilde{\theta}, \xi}(\rho_{0,\delta}, \theta_{0,\delta}) + \frac{\delta}{\beta-1} \rho^\beta_{0,\delta}  - (\rho \uvect)_{0,\delta} \V(0,\cdot)  \bigg) \notag  \\
	& \	- \int_\B \bigg( (\rho_{0,\delta}-\overline \rho) \frac{\partial \mathcal H_{\widetilde{\theta} , \xi}(\overline \rho, 1) }{\partial \rho} + \mathcal H_{\widetilde{\theta}, \xi}(\overline \rho, 1)\bigg) \notag \\
	& \  +  C \bigg(1+\frac{1}{\lambda^{5/9}} + \frac{\nu^{\alpha+1}}{\lambda^\alpha}+\bigg(\frac{\xi^{2}}{\omega \lambda^{6/(\alpha+1)}}\bigg)^{\frac{\alpha+1}{\alpha-5}}  \bigg) ,
\end{align}
for almost all $\tau \in (0,T)$. In \eqref{energy-pena-6}, the left hand side is positive due to the hypothesis of thermodynamic stability \eqref{hypo-p_m}, \eqref{hypo-e_m}.

\vspace*{.2cm}
\noindent 
$\bullet$ {\bf The uniform bounds.}
(i) From \eqref{energy-pena-6}, we directly have 
\begin{align}\label{uniform-bound-1}
	\int_0^T \int_{\Gamma_t} |(\uvect - \V)\cdot \mathbf n|^2 & \leq \veps \, C \bigg(1+\frac{1}{\lambda^{5/9}} + \frac{\nu^{\alpha+1}}{\lambda^\alpha}+\bigg(\frac{\xi^{2}}{\omega \lambda^{6/(\alpha+1)}}\bigg)^{\frac{\alpha+1}{\alpha-5}}  \bigg)  ,  \\
	\esssup_{\tau \in [0,T]} \|\delta \rho^\beta(\tau, \cdot)\|_{L^1(\B)}    &\leq C \bigg(1+\frac{1}{\lambda^{5/9}} + \frac{\nu^{\alpha+1}}{\lambda^\alpha}+\bigg(\frac{\xi^{2}}{\omega \lambda^{6/(\alpha+1)}}\bigg)^{\frac{\alpha+1}{\alpha-5}}  \bigg) ,             \label{uniform-bound-2} \\	  
	\esssup_{\tau \in [0,T]} \|\sqrt{\rho} \uvect(\tau, \cdot)\|^2_{L^2(\B)}    &\leq C \bigg(1+\frac{1}{\lambda^{5/9}} + \frac{\nu^{\alpha+1}}{\lambda^\alpha}+\bigg(\frac{\xi^{2}}{\omega \lambda^{6/(\alpha+1)}}\bigg)^{\frac{\alpha+1}{\alpha-5}}  \bigg)  ,             \label{uniform-bound-3} \\	
	\text{and } \ \  \| \lambda \theta^{\alpha+1}\|_{L^1((0,T)\times \B)} &\leq C \bigg(1+\frac{1}{\lambda^{5/9}} + \frac{\nu^{\alpha+1}}{\lambda^{\alpha}}+\bigg(\frac{\xi^{2}}{\omega \lambda^{6/(\alpha+1)}}\bigg)^{\frac{\alpha+1}{\alpha-5}}  \bigg)   \label{uniform-bound-4}.   
\end{align}

\vspace*{.1cm}
\noindent 
(ii) We also have 
\begin{align}\label{uniformbound-stress} 
	\int_0^T \int_\B \frac{\widetilde \theta}{\theta}| \Ss_\omega : \nabla_x \uvect | \leq
	C \bigg(1+\frac{1}{\lambda^{5/9}} + \frac{\nu^{\alpha+1}}{\lambda^\alpha}+\bigg(\frac{\xi^{2}}{\omega \lambda^{6/(\alpha+1)}}\bigg)^{\frac{\alpha+1}{\alpha-5}}  \bigg) ,
\end{align}
and, by applying Korn-Poincar\'{e} inequality, one can further deduce that
\begin{align}
	c_1(\omega)\|\uvect\|^2_{L^2(0,T; W^{1,2}(\B; \mathbb R^3))} \leq  C \bigg(1+\frac{1}{\lambda^{5/9}} + \frac{\nu^{\alpha+1}}{\lambda^\alpha}+\bigg(\frac{\xi^{2}}{\omega \lambda^{6/(\alpha+1)}}\bigg)^{\frac{\alpha+1}{\alpha-5}}  \bigg) . 
	\label{uniform-bound-5}
\end{align}

\vspace*{.1cm}
\noindent
(iii) From \eqref{energy-pena-6}, we get
\begin{align}\label{esti-temp-sub}
	c_2(\nu)\int_0^T \int_\B \big(\theta^{-2} + \theta^{\alpha-2} \big) |\nabla_x \theta|^2 \leq C \bigg(1+\frac{1}{\lambda^{5/9}} + \frac{\nu^{\alpha+1}}{\lambda^\alpha}+\bigg(\frac{\xi^{2}}{\omega \lambda^{6/(\alpha+1)}}\bigg)^{\frac{\alpha+1}{\alpha-5}}  \bigg)  .
\end{align}
In other words, 
\begin{align}\label{bound-theta-derivative}
	&c_2(\nu)\int_0^T \int_\B  \Big( \big|\nabla_x \log(\theta)\big|^2 + \big|\nabla_x \theta^{\frac{\alpha}{2}}\big|^2  \Big) \notag \\ &\leq C \bigg(1+\frac{1}{\lambda^{5/9}} + \frac{\nu^{\alpha+1}}{\lambda^\alpha}+\bigg(\frac{\xi^{2}}{\omega \lambda^{6/(\alpha+1)}}\bigg)^{\frac{\alpha+1}{\alpha-5}}  \bigg)  .
\end{align}

\vspace*{.1cm}
\noindent
(iv)
Now, since $\mathcal H_{\widetilde{\theta} , \xi}$ is coercive (this can be proved in accordance with  \cite[Proposition 3.2]{Feireisl-Novotny-book}) and bounded from below, we get 
\begin{align}\label{bound-rho-e-xi}
	\esssup_{\tau\in [0,T]} \|\rho e_\xi(\tau, \cdot)\|_{L^1(\B)} \leq C \bigg(1+\frac{1}{\lambda^{5/9}} + \frac{\nu^{\alpha+1}}{\lambda^\alpha}+\bigg(\frac{\xi^{2}}{\omega \lambda^{6/(\alpha+1)}}\bigg)^{\frac{\alpha+1}{\alpha-5}}  \bigg)  ,
\end{align}
and consequently we have 
\begin{align}\label{bound-theta-L4} 
	\esssup_{\tau\in (0,T)} \|a_\xi \theta^4(\tau, \cdot)\|_{L^1(\B)} \leq C \bigg(1+\frac{1}{\lambda^{5/9}} + \frac{\nu^{\alpha+1}}{\lambda^\alpha}+\bigg(\frac{\xi^{2}}{\omega \lambda^{6/(\alpha+1)}}\bigg)^{\frac{\alpha+1}{\alpha-5}}  \bigg) , \\
	\label{rho_5/3}  \esssup_{\tau\in (0,T)} \|\rho(\tau, \cdot)\|^{\frac{5}{3}}_{L^{\frac{5}{3}}(\B)}  \leq C \bigg(1+\frac{1}{\lambda^{5/9}} + \frac{\nu^{\alpha+1}}{\lambda^\alpha}+\bigg(\frac{\xi^{2}}{\omega \lambda^{6/(\alpha+1)}}\bigg)^{\frac{\alpha+1}{\alpha-5}}  \bigg) .
\end{align}

\vspace*{.1cm}
\noindent
(v)  The bound \eqref{rho_5/3}  gives the uniform  bound for the gravitational potential $\Psi$ w.r.t. ``$\veps$". In fact, by following the steps as we obtained \eqref{bound-poisson}, one could get
\begin{align}\label{bound-poisson-eps}
	\|\Psi\|^2_{L^2(0,T; W^{1,2}(\B))} &\leq  \|\rho\|^{\frac{5}{3}}_{L^\infty(0,T; L^\frac{5}{3}(\B))} + C(\rho_0,g) \notag \\
	&\leq C \bigg(1+\frac{1}{\lambda^{5/9}} + \frac{\nu^{\alpha+1}}{\lambda^\alpha}+\bigg(\frac{\xi^{2}}{\omega \lambda^{6/(\alpha+1)}}\bigg)^{\frac{\alpha+1}{\alpha-5}}  \bigg) .
\end{align}

\vspace*{.1cm}
\noindent
(vi)
Then by \eqref{bound-theta-derivative}, \eqref{uniform-bound-4} and generalized  Poincar\'e inequality from \Cref{Poincare} (since the condition \eqref{measurable-condition} satisfies), we deduce that 
\begin{align}\label{bound-theta-gamma}
	&\|\theta^\gamma \|^2_{L^2(0,T; W^{1,2}(\B))} \leq \widehat C_1, 
	\quad \text{for any } 1\leq \gamma \leq \frac{\alpha}{2} , \ \ \text{where } \alpha >6 ,
\end{align}
where the constant $\widehat C_1>0$ may depend on the parameters $\xi, \nu, \omega, \lambda$ but not on $\veps$.

The estimate \eqref{bound-theta-derivative}  also provides us
\begin{align} 
	\label{bound-log-theta}
	&	\|\nabla_x \log \theta \|^2_{L^2(0,T; L^{2}(\B))} \leq \widehat C_2,  
\end{align}
for some constant $\widehat C_2>0$ that is independent in $\veps$.

\vspace*{.1cm}  
\noindent
(vii)  Further,  we have 
\begin{align}\label{bound-kappa-theta-2}
	&\int_0^T \int_\B	\frac{\kappa_\nu(\theta,t,x)}{\theta} |\nabla_x \theta|   \notag \\
	& 
	\leq \frac{1}{2}	\int_0^T \int_\B \frac{\kappa_\nu(\theta, t,x)}{\theta^2} |\nabla_x \theta|^2 + \int_0^\tau \int_\B \kappa_\nu(\theta,t,x)  \notag \\
	& \leq C(\nu) \int_0^T \int_\B  \Big( \big|\nabla_x \log(\theta)\big|^2 + \big|\nabla_x \theta^{\frac{\alpha}{2}}\big|^2  \Big) + C(\nu) \int_0^T \int_\B (1+ \theta^\alpha) 
	\leq \widehat C_3 , 
\end{align} 
where we have used the bounds \eqref{bound-theta-derivative} and \eqref{uniform-bound-4}. The constant $\widehat C_3>0$  does depend on the parameters $\xi, \nu, \omega, \lambda$ but not on $\veps$.

\vspace*{.1cm}
\noindent
(viii) Using the technique based on the {\em Bogovskii operator}, 
one can get more information about the modified pressure, namely 
$$p_{\xi, \delta}(\rho, \theta) =  p_M(\rho, \theta) + \frac{a_\xi}{3}\theta^4 + \delta \rho^\beta \quad (\beta\geq 4). $$
The idea is to use the multipliers of the form 
\begin{align*}
	\psi\, \mathcal L \big[ \rho -\frac{1}{|\B|} \int_\B \rho    \big], \quad \psi \in \mathcal D(0,T) , \ \ 0\leq \psi\leq 1
\end{align*}
in the momentum equation\eqref{momentum-eq},
where $\mathcal L$ is defined as follows:
the function $\mathbf v = \mathcal L[f]$ such that it solves the problem
$$  \Div ( \mathbf v) = f \quad \text{in } \B, \quad \mathbf v |_{\partial \B} = 0  .  $$
In what follows, one can ensure that there exists some $\upsilon>0$ such that
\begin{align}\label{Bogovskii}
	\iint_K \left( p_{\xi,\delta} (p,\theta) \rho^\upsilon + \delta \rho^{\beta+\upsilon} \right) \leq C(K),
\end{align}
for any compact set $K\subset (0,T)\times \B$ such that 
\begin{align}\label{Set-K}
	K \cap \left(\cup_{\tau\in [0,T]} \big(\{\tau\} \times \Gamma_\tau\big)  \right) = \emptyset .
\end{align}
Moreover, the constant $\upsilon$ can be chosen independently of $\veps$, $\omega$, $\lambda$, $\xi$, $\nu$ and $\delta$. For more details, we refer \cite[Section 4.2]{Feireisl-NSF-1} or \cite{Feireisl-hana}.

\vspace*{.1cm}
\noindent 
(ix)  To find a suitable estimate for the term $\rho s_\xi(\rho, \theta)$, we recall the estimates \eqref{Fact-2} and \eqref{Fact-3}, which gives 
\begin{align}\label{bound-rho_s} 
	\rho s_\xi(\rho, \theta) \leq 
	\begin{dcases}
		c \rho + \frac{4a_\xi}{3} \theta^3 , \qquad \qquad \qquad \qquad \qquad \ \ \text{for } \frac{\rho}{\theta^{3/2}} >1, \\
		C \Big(\rho^{\frac{1}{2}} + \rho +  \rho [\log \theta]^+   \Big)   + \frac{4a_\xi}{3}\theta^3  \\
		\qquad \leq C \Big(\theta^{\frac{3}{4}} + \theta^{\frac{3}{2}} +  \theta^2     \Big)   + \frac{4a_\xi}{3}\theta^3 , \ \ \, \text{for } \frac{\rho}{\theta^{3/2}} \leq 1 .
	\end{dcases}
\end{align}
Now, using the  estimates \eqref{uniform-bound-4}, \eqref{rho_5/3}, one can deduce that 
\begin{align}\label{rho-s-xi}    
	\|\rho s_\xi(\rho, \theta)\|_{L^p((0,T)\times \B)} \leq \widehat C_4  , \quad \text{for certain  } p\geq 1,
\end{align}
where the constant $\widehat C_4 >0$ that may depend on  the parameters $\xi, \nu, \lambda, \omega$  
but not on $\veps$.

One can  further deduce that 
\begin{align}\label{rho-s-xi-u}
	\|\rho s_{\xi}(\rho, \theta) \uvect\|_{L^1((0,T)\times \B)} \leq C(\lambda)  ,
\end{align}
and again the constant $C(\lambda)>0$ is independent   in  
$\veps$. This can be determined  as follows: the crucial term to estimate is $a_\xi\theta^3\uvect$ due to the bound  \eqref{bound-rho_s}.   Indeed, we find (since $\alpha>6$)
\begin{align*}
	\left|\int_0^T \int_\B a_\xi  \theta^3 \uvect \right| 
	&\leq C \int_0^T \int_\B \theta^6 + C \int_0^T \int_\B |\uvect|^2 \\
	& \leq C(\lambda) \|\lambda \theta^{\alpha+1}\|_{L^1((0,T)\times \B)} + C \left\|\uvect \right\|^2_{L^2(0,T; W^{1,2}(\B; \mathbb R^3))},
	%
\end{align*} 
which is uniformly  bounded in $\veps>0$  by means of \eqref{uniform-bound-4} and \eqref{uniform-bound-5}.

\section{Passing to the limit} 

In this section, we first perform the limit $\veps\to 0$ and then together we pass to the limit $\lambda$, $\xi$, $\nu$, $\omega$ and finally $\delta$ to $0$.

\subsection{Penalization limit: passing with $\veps\to 0$}  

In this subsection, we fix all the parameters $\delta$, $\lambda$, $\nu$, $\xi$ and  $\omega$.    Then,
passing to the limit $\veps\to 0$, we directly obtain
\begin{align}\label{retriving boundary}
	(\uvect - \V )\cdot \mathbf n \big|_{\Gamma_\tau} = 0 ,\quad \text{for a.a. } \ \tau \in [0,T],  
\end{align}
so we retrieve the impermeability boundary condition \eqref{imperm}.

\vspace*{.1cm}
\noindent 
$\bullet$ 
By \eqref{bound-theta-L4}, \eqref{rho_5/3} and  \eqref{bound-theta-gamma}, we respectively have (up to a suitable subsequence)
\begin{align}
	&	\theta_\veps \to \theta \quad \text{weakly$^*$ in }   \ L^\infty(0,T; L^4(\B))  \ \text{ as $\veps \to 0$}   , \label{conv-theta-1}   \\
	&   \rho_\veps \to \rho \quad \text{weakly$^*$ in }   \ L^\infty(0,T; L^{\frac{5}{3}}(\B)) \ \text{ as $\veps \to 0$} , \label{conv-rho-1}   \\   
	& \theta_\veps \to \theta \quad \text{weakly in } \  L^2(0,T; W^{1,2}(\B)) \ \text{ as $\veps \to 0$}   ,  \label{conv-theta-2} 
\end{align}
Due to \eqref{uniform-bound-4} and \eqref{bound-theta-L4}, we also have 
\begin{align}\label{conv-theta-8}
	&\theta^{\alpha+1}_\veps  \to \overline{\theta^{\alpha+1}} \quad \text{weakly in } \ L^1((0,T)\times \B) \ \text{ as $\veps \to 0$} , \\
	\label{conv-theta-4}
	&	\theta^4_\veps  \to \overline{\theta^4} \quad \text{weakly in } \ L^1((0,T)\times \B) \ \text{ as $\veps \to 0$} .
\end{align}
Here and in the sequel, the ``bar" denotes a weak limit of a composed or nonlinear function. 

\vspace*{.1cm}
\noindent 
$\bullet$ Thanks to  \eqref{uniform-bound-5} and \eqref{bound-poisson-eps}, we respectively have 
\begin{align}
	\label{conv-u}
	&\uvect_\veps  \to \uvect \quad \text{weakly in } \ L^2(0,T; W^{1,2}(\B; \mathbb R^3)) \ \ \text{as $\veps \to 0$} , \ \,
	\text{and} \\
	\label{conv-poisson}
	&	\Psi_\veps  \to \Psi \quad \text{weakly in } \ L^2(0,T; W^{1,2}(\B)) \ \ \text{as $\veps \to 0$} .
\end{align}

\vspace*{.1cm}
\noindent 
$\bullet$
We have also better convergence result of $\{\rho_\veps\}_\veps$ than \eqref{conv-rho-1}: using the continuity equation \eqref{weak-conti} one indeed get 
\begin{align}\label{conv-rho-2} 
	\rho_\veps \to \rho \quad \text{in }  \ \C_{\text{weak}}([0,T]; L^{\frac{5}{3}}(\B)) \ \text{ as $\veps \to 0$} .   
\end{align}   
The above fact, together with \eqref{conv-u} and the fact $L^{\frac{5}{3}}(\B) \hookrightarrow W^{-1,2}(\B)$ is compact, one has 
\begin{align}\label{conv-rho-u}
	\rho_\veps \uvect_\veps \to \rho \uvect \quad \text{weakly$^*$ in } \ L^{\infty}(0,T; L^{\frac{5}{4}}(\B, \mathbb R^3))  \ \text{ as $\veps \to 0$}
\end{align}

\vspace*{.1cm}
\noindent 
$\bullet$ 
Now, by using   \eqref{conv-theta-4} and \eqref{Bogovskii},  the asymptotic behavior \eqref{mole_bound_p_M} of $p_M$, and then by utilizing  \eqref{conv-rho-2}, \eqref{conv-theta-2} we have
\begin{align}\label{limit-pressure}
	p_{\xi, \delta}(\rho_\veps, \theta_\veps) = p_M(\rho_\veps, \theta_\veps) + \frac{a_\xi}{3} \theta^4_\veps + \delta \rho^\beta_\veps \to    \overline{p_{M}(\rho, \theta)} + \frac{a_\xi}{3} \overline{\theta^4} + \delta \overline{\rho^\beta} \quad \text{weakly in } \ L^1(K) ,  
\end{align}
where $K$ is as given by \eqref{Bogovskii}--\eqref{Set-K}.

\vspace*{.1cm}
\noindent 
$\bullet$  
Further, since $W^{1,2}_0(\B, \mathbb R^3)\hookrightarrow L^6(\B, \mathbb R^3)$ is compact, using \eqref{conv-u}, the convective term satisfies (by following the steps of \cite[Section 4]{sarka-aneta-al-JMPA} or \cite[Chapter 3.6.4]{Feireisl-Novotny-book})
\begin{align}\label{weak-limit-rho_u-cross_u}
	\rho_\veps \uvect_\veps \otimes \uvect_\veps \to  \overline{\rho \uvect \otimes \uvect} \quad \text{weakly in } L^2(0,T; L^{\frac{30}{29}}(\B, \mathbb R^3)) ,
\end{align}
and indeed, 
\begin{align}
	\overline{\rho \uvect \otimes \uvect} = \rho \uvect \otimes \uvect \quad \text{a.a. in } (0,T)\times \B,
\end{align}
since $L^{\frac{5}{4}}(\B)\hookrightarrow W^{-1,2}(\B)$ is compact.

\subsection{Pointwise convergence of the temperature and the density} 

$\bullet$ In order to show a.e. convergence of
the temperature, we follow the technique  based on the Div–Curl Lemma   and
Young measures methods   (see for instance \cite[Section 3.6.2]{Feireisl-Novotny-book}).       
To this end, we  set 
\begin{align}
	&	\mathbf U_\veps = \left[\rho_\veps s_\xi (\rho_\veps , \theta_\veps ), \  \rho_\veps s_\xi (\rho_\veps, \theta_\veps) \uvect_\veps + \frac{\kappa_\nu(\theta_\veps) \nabla_x \theta_\veps}{\theta_\veps}   \right], \\
	& \mathbf W_\veps = \left[G(\theta_\veps), 0, 0, 0  \right] , 
\end{align}
where $G$ is bounded and globally Lipschitz function in 
$[0, \infty)$.  Then due to the estimates obtained in previous section, $\Div_{t,x} \mathbf U_\veps$ is precompact in $W^{-1,s}((0,T)\times \B)$ and $\text{Curl}_{t,x} \mathbf W_\veps$ is precompact in $W^{-1,s}((0,T)\times \B)^{4\times 4}$ with certain $s>1$. Therefore using the Div-Curl lemma 
for $\mathbf U_\veps$ and $\mathbf W_\veps$, we may derive that 
\begin{align}\label{ineq-1}
	\overline{\rho s_{\xi}(\rho_\veps, \theta_\veps) G(\theta_\veps)} = \overline{\rho s_{\xi}(\rho_\veps, \theta_\veps) }\ \overline{G(\theta_\veps)} 
\end{align}
In fact, by applying the theory of parameterized (Young) measures (see \cite[Section 3.6.2]{Feireisl-Novotny-book}), one can show that
\begin{align}\label{ineq-2} 
	\overline{\rho s_M(\rho, \theta) G(\theta)} \geq 	\overline{\rho s_M(\rho, \theta)} \ \overline{ G(\theta)}, \qquad \overline{\theta^3 G(\theta)} \geq \overline{\theta^3} \ \overline{G(\theta)} .
\end{align}
Combining \eqref{ineq-1}--\eqref{ineq-2} and taking $G(\theta)= \theta$, we deduce 
\begin{align*}
	\overline{\theta^4} = \overline{\theta^3} \, \theta  ,
\end{align*}
which  yields 
\begin{align}\label{limit-theta}
	\theta_\veps \to \theta  \quad \text{a.a. in } \ (0,T)\times \B. 
\end{align}
Moreover,  thanks to \eqref{bound-log-theta} and  using the generalized Poincar\'{e} inequality in \Cref{Poincare}, one can prove that $\log \theta \in L^2((0,T) \times \B)$ which ensures that the  limit temperature is positive a.e. on the set $(0,T)\times \B$.

\vspace*{.1cm}
\noindent 
$\bullet$
Next proceeding as \cite[Section 4.1.2]{Sarka-et-al-ZAMP},   one can further obtain 
\begin{align}\label{limit-rho}
	\rho_\veps \to \rho \quad \text{a.a. in } \ (0,T)\times \B .
\end{align}

\vspace*{.1cm}
\noindent 
$\bullet$
Then, using \eqref{limit-theta}, \eqref{limit-rho}, \eqref{conv-u} and the bounds \eqref{rho-s-xi} and \eqref{rho-s-xi-u},  
we identify the following limits: 
\begin{align}\label{limits-stress-entropy}
	\begin{dcases}
		\Ss_\omega(\theta_\veps, \nabla_x \uvect_\veps) \to \Ss_\omega(\theta, \nabla_x \uvect) \  \ \ &\text{ weakly in } \ L^1((0,T)\times \B) , \\
		\rho_\veps  s_\xi(\rho_\veps, \theta_\veps) \to \rho s_\xi (\rho, \theta) \ \ \ &\text{ weakly in } \ L^1((0,T)\times \B), \\
		\rho_\veps s_\xi(\rho_\veps, \theta_\veps) \uvect_\veps \to \rho s_\xi(\rho, \theta) \uvect \ \ \ &\text{ weakly in } \ L^1((0,T)\times \B) ,
	\end{dcases}
\end{align}
up to a suitable subsequence.

\subsection{The limiting system as $\veps\to 0$} In this subsection, we summarize  the limiting  behaviors of all the quantities from the previous two subsections, and write weak formulation for the limiting system (as $\veps \to 0$). 

\vspace*{.1cm}
\noindent 
$\bullet$ Passing to the limit as $\veps\to 0$, we can obtain the continuity equation exactly as \eqref{weak-conti}.

\vspace*{.1cm}
\noindent 
$\bullet$
Next, we proceed to pass to the limit in the momentum equation \eqref{weak-momen}. Having in hand the local estimates (and limits) of the pressure term (see \eqref{Bogovskii} and \eqref{limit-pressure}), we consider the test functions 
\begin{equation} 
	\begin{aligned}\label{test-func}
		\boldvphi \in \C_c^1( [0,T) ; W^{1, \infty}(\B; \mathbb R^3) ) , \ \  \text{Supp}\, [\Div_x \boldvphi(\tau, \cdot)] \cap \Gamma_\tau = \emptyset, \\
		\boldvphi \cdot \mathbf n |_{\Gamma_\tau} = 0 , \ \ \  \forall \tau \in [0,T] .
	\end{aligned}  
\end{equation}
Then in accordance with the limits in the previous two subsections, we have upon $\veps \to 0$,
\begin{align}\label{weak-momen-limit}
	-	\int_0^T \int_{\B} \left( \rho \uvect \cdot \partial_t  \boldvphi   +  \rho[\uvect \otimes \uvect] : \nabla_x  \boldvphi + p_{\xi, \delta}(\rho, \theta) \Div_x \boldvphi   \right)   
	+\int_0^T \int_{\B} \Ss_\omega : \nabla_x \boldvphi  \notag \\
	\int_0^T \int_\B \rho \nabla_x \Psi \cdot \nabla_x \vphi		  =
	\int_{\B}   (\rho \uvect)_{0,\delta}\cdot  \boldvphi(0, \cdot) ,
\end{align}
for any test function as in \eqref{test-func}.

\vspace*{.1cm}
\noindent 
$\bullet$
The weak formulations for $\Psi$  has same expressions as \eqref{pena-weak-poisson}  after passing to the limit as $\veps \to 0$.

\vspace*{.1cm}
\noindent 
$\bullet$
Further, by using \eqref{bound-kappa-theta-2}, we have  
\begin{align}\label{weak-limit-kappa-theta}
	\frac{\kappa_\nu(\theta_\veps)}{\theta_\veps} |\nabla_x \theta_\veps| \to  \frac{\kappa_\nu(\theta)}{\theta} |\nabla_x \theta| \quad \text{weakly in } \ L^1((0,T)\times \B) . 
\end{align}
We also have that the terms $\disp \frac{1}{\theta_\veps} \Ss_\omega(\theta_\veps, \nabla_x \uvect_\veps): \nabla_x \uvect_\veps$ and  $\disp \frac{\kappa_\nu(\theta_\veps)|\nabla_x \theta_\veps|^2}{\theta_\veps^2}$ are weakly lower semicontinuous.
These, together with \eqref{conv-theta-4},  \eqref{limit-theta} and \eqref{limit-rho}, the entropy inequality follows (as $\veps\to 0$) 
\begin{align}\label{limit-enropy-ineq}
	-\int_0^T \int_{\B} \left(\rho s_\xi(\rho,\theta) \left( \partial_t \vphi +  \uvect \cdot \nabla_x \vphi\right) - \frac{\kappa_\nu(\theta, t,x)}{\theta} \nabla_x \theta \cdot \nabla_x \vphi  \right) 
	\notag \\
	- \int_{\B} \rho_{0,\delta} s(\rho_{0,\delta}, \theta_{0,\delta}) \vphi(0,\cdot)   
	+ \int_0^T \int_\B \lambda \theta^{\alpha} \vphi 
	\notag \\
	\geq \int_0^T \int_{\B} \frac{\vphi}{\theta}  \left( \Ss_\omega : \nabla_x \uvect 
	+ \frac{\kappa_\nu(\theta, t,x)}{\theta} |\nabla_x \theta|^2   \right) ,
\end{align}
for any test function $\vphi \in \C^1_c([0,T)\times \B; \mathbb R)$ with $\vphi \geq 0$. 

\vspace*{.1cm}
\noindent 
$\bullet$
Let us pass to the limit $\veps\to 0$ in the modified energy inequality \eqref{energy-pena-1-2}. Thanks to the almost everywhere convergence results \eqref{limit-theta}, \eqref{limit-rho}, the bound \eqref{bound-rho-e-xi} and the fact that $\{\rho_\veps e_{\xi}(\rho_\veps, \theta_\veps )\}_{\veps}$ is nonnegative, we have, using the Fatou's lemma, 
\begin{align*}
	\limsup_{\veps \to 0} \int_0^T \int_{\B} \rho_\veps e_{\xi}(\rho_\veps, \theta_\veps ) \partial_t \psi \leq \int_0^T \int_{\B} \rho e_{\xi}(\rho, \theta ) \partial_t \psi
\end{align*} 
by choosing $\psi \in \C^1_c([0,T))$ such that $\partial_t \psi \leq 0$.
Similarly, one has 
\begin{align*}
	\limsup_{\veps \to 0} \int_0^T \int_{\B} \rho_\veps |\uvect_\veps|^2 \partial_t \psi \leq 	 \int_0^T \int_{\B} \rho |\uvect|^2 \partial_t \psi . 
\end{align*}
Using the above information
and gathering all the limits \eqref{limit-theta}, \eqref{limit-rho}, \eqref{limit-pressure},  \eqref{weak-limit-kappa-theta}, \eqref{weak-limit-rho_u-cross_u}, the limiting ballistic energy inequality reads as (passing to the limit as $\veps\to 0$ in \eqref{energy-pena-1-2}) 
\begin{align}\label{limit-energy-pena}
	&	-\int_0^T \partial_t \psi \int_\B \bigg(\frac{1}{2} \rho |\uvect|^2  + \mathcal H_{\widetilde{\theta}, \xi}(\rho, \theta)   + \frac{\delta}{\beta-1} \rho^\beta\bigg)  + \int_0^T \psi \int_\B \lambda \theta^{\alpha+1} \notag \\
	& 	  
	+ \int_0^T \psi \int_{\B} \frac{\widetilde{\theta}}{\theta} \bigg( \Ss_\omega : \nabla_x \uvect +  \frac{\kappa_\nu(\theta,t,x)}{\theta} |\nabla_x \theta|^2     \bigg) 
	\notag \\	
	&\leq  
	\psi(0)\int_{\B} \bigg(\frac{1}{2}  \frac{|(\rho\uvect)_{0,\delta}|^2}{\rho_{0,\delta}}  + 	\mathcal H_{\widetilde{\theta}, \xi}(\rho_{0,\delta}, \theta_{0,\delta})  + \frac{\delta}{\beta-1} \rho^\beta_{0,\delta} - (\rho \uvect)_0 \V(0)  \bigg) 
	\notag \\
	& \ 	+ \int_0^T \psi \int_\B \lambda \theta^{\alpha}  \widetilde{\theta} 
	- \int_0^T \int_\B \rho \uvect \cdot \partial_t (\psi \V)  -\int_0^T \psi \int_\B \rho \nabla_x \Psi \cdot \V  \notag  \\ 
	& \ 	- \int_0^T  \psi\int_{\B} \Big(\rho[\uvect \otimes \uvect]    : \nabla_x \V  - \Ss_\omega : \nabla_x \V  + p_{\xi, \delta}(\rho,\theta)\, \Div_x \V  \Big) 
	\notag \\ 
	& \
	-	\int_0^T \psi \int_{\B} \Big[\rho s_{\xi}(\rho,\theta) \big(\partial_t \widetilde{\theta}  + \uvect \cdot \nabla_x \widetilde{\theta} \big) - \frac{\kappa_\nu(\theta, t, x)}{\theta} \nabla_x \theta \cdot \nabla_x \widetilde{\theta}   \Big] ,
\end{align}
for all $\psi \in \C^1_c([0,T))$ with $\psi\geq 0$ and $\partial_t \psi \leq 0$.

\subsection{Get rid of the density dependent solid part}\label{Section:Solid-part} 

Let us take care the density-dependent terms in the solid part $((0,T)\times \B)\setminus Q_T$. We use \cite[Lemma 4.1]{Sarka-Kreml-Neustupa-Feireisl-Stebel} (see also \cite[Section 4.1.4]{sarka-aneta-al-JMPA}) to conclude that 
the density $\rho$ remains ``zero" on the solid part if it was so initially, thanks to the continuity equation and the fact that density is square-integrable. Here, we must mention that the square-integrability of  density is identified from the estimate of $\delta \rho^\beta$ given by \eqref{uniform-bound-2}.

\vspace*{.1cm}

We also recall that, the parameters $\omega$, $\xi$ and $\nu$ are not  involved in $\Omega_t$ ($t\in [0,T]$) as per our extension strategy given in the beginning of Section \ref{sec-penalized}. 

\vspace*{.1cm} 
\noindent 
$\bullet$ Since, we have set the initial data $\rho_{0,\delta}$ to be zero in $\B\setminus \Omega_0$, it remains zero in $\B\setminus \Omega_t$ due to \cite[Lemma 4.1]{Sarka-Kreml-Neustupa-Feireisl-Stebel}.  This leads to the following weak formulation for the continuity equation upon passing to the limit as $\veps \to 0$, 
\begin{align}\label{weak-conti-recover}
	-\int_0^T \int_{\Omega_t} \rho B(\rho)\left( \partial_t \vphi +    \uvect \cdot \nabla_x \vphi \right) 
	+	\int_0^T \int_{\Omega_t} b(\rho) \Div_x \uvect \vphi 
	=  \int_{\Omega_0} \rho_{0,\delta} B(\rho_{0,\delta})\vphi(0, \cdot),
\end{align}
for any  test function $\vphi \in \C^1_c( [0,T)\times \B; \mathbb R)$  and any $b \in L^\infty \cap \C([0,+\infty))$ such that $b(0)=0$ and $\displaystyle B(\rho) = B(1)+ \int_{1}^\rho \frac{b(z)}{z^2}$. 

\vspace*{.1cm}
\noindent 
$\bullet$ Using the fact $\rho=0$ in $\B\setminus \Omega_t$ for any $t\in [0,T]$, the momentum equation now  reads 
\begin{align}\label{weak-momen-recover}
	&-	\int_0^T \int_{\Omega_t} \left( \rho \uvect \cdot \partial_t  \boldvphi   +  \rho[\uvect \otimes \uvect] : \nabla_x  \boldvphi + p_{\delta}(\rho, \theta) \Div_x \boldvphi   \right)   
	+\int_0^T \int_{\Omega_t} \Ss : \nabla_x \boldvphi \notag \\
	& \qquad  - \int_0^T \int_{\Omega_t} \rho \nabla_x \Psi \cdot \boldvphi 
	=
	\int_{\Omega_0}   (\rho \uvect)_{0,\delta}\cdot  \boldvphi(0, \cdot) -\int_0^T \int_{\B\setminus \Omega_t} \Ss_\omega : \nabla_x \boldvphi  \notag \\ 
	& \hspace{4cm}+ \int_0^T \int_{\B\setminus \Omega_t} \frac{a_\xi}{3}\theta^4 \Div_x \boldvphi ,
\end{align}
for any test function $\boldvphi$ satisfying 
\begin{align*}
	\boldvphi \in \C^1_c([0,T)\times \B;\mathbb R^3) \quad \text{with} \quad \boldvphi(\tau, \cdot) \cdot \mathbf n \big|_{\Gamma_\tau} =0 \ \ \text{for any } \tau\in [0,T].
\end{align*}

\vspace*{.1cm} 
\noindent
$\bullet$  Since $\rho=0$ in $\B\setminus \Omega_t$ for each $t\in [0,T]$, the weak formulation for $\Psi$ from \eqref{pena-weak-poisson} simply reduces to
\begin{align}\label{pena-weak-poisson-recover}
	\int_0^T \int_{\Omega_t} \nabla_x \Psi \cdot \nabla_x \vphi = \int_0^T \int_{\Omega_t} \rho \vphi , 
\end{align}
for any test function $\vphi \in \C^1((0,T)\times \B; \mathbb R)$.

\vspace*{.1cm}
\noindent 
$\bullet$ Next, we write the limiting entropy inequality (recall the fact that $s_\xi(\rho, \theta)=s_M(\rho, \theta)+ \frac{4a_\xi}{3\rho}\theta^3$)
\begin{align}\label{enropy-ineq-recover}
	-\int_0^T \int_{\Omega_t} \rho s(\rho,\theta) \left(\partial_t \vphi + \
	\uvect \cdot \nabla_x \vphi \right) -	  \int_0^T \int_{\B\setminus \Omega_t} \frac{4a_\xi}{3}\theta^3  \left(\partial_t \vphi + 
	\uvect \cdot \nabla_x \vphi \right) 
	\notag\\
	+ \int_0^T \int_{\Omega_t} \frac{\kappa(\theta, t,x)}{\theta} \nabla_x \theta \cdot \nabla_x \vphi 
	+	 \int_0^T \int_{\B \setminus \Omega_t} \frac{\kappa_\nu(\theta, t,x)}{\theta} \nabla_x \theta \cdot \nabla_x \vphi 
	\notag \\
	- \int_{\Omega_0} \rho_{0,\delta} s(\rho_{0,\delta}, \theta_{0,\delta}) \vphi(0,\cdot) 	- \int_{\B\setminus \Omega_0} \frac{4a_\xi}{3}\theta_{0,\delta}^3 \vphi(0,\cdot) 
	+	 \int_0^T \int_\B \lambda \theta^{\alpha} \vphi 
	\notag \\
	\geq \int_0^T \int_{\Omega_t} \frac{\vphi}{\theta}  \left( \Ss: \nabla_x \uvect   + \frac{\kappa(\theta, t,x)}{\theta} |\nabla_x \theta|^2   \right)
	\notag \\
	+ \int_0^T \int_{\B\setminus \Omega_t} \frac{\vphi}{\theta}  \left( \Ss_\omega : \nabla_x \uvect    + \frac{\kappa_\nu(\theta, t,x)}{\theta} |\nabla_x \theta|^2   \right), 
\end{align}
for any test function $\vphi \in \C^1_c([0,T)\times \B)$ with $\vphi\geq 0$.

\noindent 
$\bullet$  We shall now look at the ballistic energy inequality \eqref{limit-energy-pena}.  
It has the following form now. 
\begin{align}\label{energy-pena-recover}
	&	- \int_0^T\int_{\Omega_t} \partial_t \psi\bigg(\frac{1}{2} \rho |\uvect |^2  + \rho e(\rho, \theta) - \rho s(\rho, \theta) \widetilde{\theta} 
	+ \frac{\delta \rho^\beta}{\beta-1} \bigg)
	-  \int_0^T \int_{\B\setminus\Omega_t} a_\xi \theta^4 \partial_t \psi   \notag  \\
	&  +  \int_0^T \int_{\B\setminus \Omega_t}\frac{4a_\xi}{3} \theta^3 \widetilde{\theta} \partial_t \psi  
	+ \int_0^T \psi \int_{\Omega_t} \frac{\widetilde{\theta}}{\theta} \bigg( \Ss : \nabla_x \uvect  + \frac{\kappa(\theta,t,x)}{\theta} |\nabla_x \theta|^2     \bigg)     \notag \\
	& +  \int_0^T \psi \int_{\B\setminus \Omega_t} \frac{\widetilde{\theta}}{\theta} \bigg( \Ss_\omega : \nabla_x \uvect + \frac{\kappa_\nu(\theta,t,x)}{\theta} |\nabla_x \theta|^2     \bigg)   + \int_0^T \psi \int_\B \lambda \theta^{\alpha+1}  	  
	\notag \\	
	&\leq   
	\psi(0)\int_{\Omega_0}\bigg(\frac{1}{2}  \frac{|(\rho\uvect)_{0,\delta}|^2}{\rho_{0,\delta}}  + 	\rho_{0,\delta} e(\rho_{0,\delta}, \theta_{0,\delta}) - \rho_{0,\delta} s(\rho_{0,\delta}, \theta_{0,\delta})\widetilde{\theta}(0)  + \frac{\delta \rho_{0,\delta}^\beta}{\beta-1}    \bigg) \notag \\ 
	& - \psi(0) \int_{\Omega_0}  (\rho \uvect)_{0,\delta} \V(0)  + \psi(0) \int_{\B\setminus \Omega_0}  a_\xi \Big(\theta^4_{0,\delta} -  \frac{4}{3}\theta^3_{0,\delta} \widetilde{\theta}(0)  \Big)  \notag \\
	&+   \int_0^T \psi \int_\B \lambda \theta^{\alpha} \, \widetilde{\theta}  - \int_0^T \int_{\Omega_t} \rho \uvect \cdot \partial_t (\psi \V)   - \int_0^T \psi \int_{\Omega_t} \rho \nabla_x \Psi \cdot   \V \notag \\ 
	&	- \int_0^T \psi \int_{\Omega_t} \Big(\rho[\uvect \otimes \uvect]    : \nabla_x \V  - \Ss : \nabla_x \V  + p_{ \delta}(\rho,\theta)\, \Div_x \V  \Big)   
	\notag \\
	& + \int_0^T \psi \int_{\B\setminus \Omega_t} \Ss_\omega : \nabla_x \V  -  \int_0^T \psi \int_{\B\setminus \Omega_t} \frac{a_\xi}{3} \theta^4 \Div_x \V  \notag \\
	&
	-	\int_0^T \psi \int_{\Omega_t} \bigg[\rho s(\rho,\theta) \big(\partial_t \widetilde{\theta}  + \uvect \cdot \nabla_x \widetilde{\theta}  \big) - \frac{\kappa(\theta, t, x)}{\theta} \nabla_x \theta \cdot \nabla_x \widetilde{\theta}  \bigg]  \notag \\
	& - 	\int_0^T \psi \int_{\B \setminus \Omega_t} \frac{4a_\xi}{3}\theta^3  \Big(\partial_t \widetilde{\theta} + \uvect \cdot \nabla_x \widetilde{\theta}  \Big)  
	+ \int_0^T \psi \int_{\B \setminus \Omega_t} \frac{\kappa_\nu(\theta, t, x)}{\theta} \nabla_x \theta \cdot \nabla_x \widetilde{\theta}  ,
\end{align}
for any $\psi \in \C^1_c([0,T))$ with $\psi\geq 0$ and $\partial_t \psi \leq 0$.  



\subsection{Passing to the limit of other parameters}\label{Section-Scaling}
In this section, we first  consider the following scaling for all the parameters $\omega, \xi, \nu, \lambda$. Let  
\begin{align}\label{scaling} 
	\lambda = \nu^{\frac{1}{3}}   = \omega^{\frac{1}{3}} =  \xi^{\frac{1}{6}} = h \quad \text{for } h>0.
\end{align}

\subsubsection{Step 1. Bounds of the integrals in $((0,T)\times \B) \setminus Q_T$}

In this step, we shall find suitable  bounds of the integrals in $((0,T)\times \B) \setminus Q_T$.

\vspace*{.1cm}
\noindent 
$\bullet$ We start with the following. Recall the weak formulation \eqref{weak-momen-recover} for the momentum equation and focus on the integrals on $\B\setminus \Omega_t$. First,  we compute
\begin{align}\label{limit-1}  
	\left|	\int_0^T \int_{\B\setminus \Omega_t} \frac{a_\xi}{3} \theta^4 \Div_x \boldvphi \right| 
	&\leq C\|\Div_x \boldvphi \|_{L^\infty((0,T)\times \B)}  \, \frac{\xi}{\lambda^{\frac{4}{\alpha+1}}} \bigg( \int_0^T \int_{\B} \lambda \theta^{\alpha+1} \bigg)^{\frac{4}{\alpha+1}} 
	\notag \\
	& \leq  \frac{C\xi}{\lambda^{\frac{4}{\alpha+1}}} \bigg[  1 + \frac{1}{\lambda^{\frac{5}{9}}}  +  \frac{\nu^{\alpha+1}}{\lambda^\alpha} + \bigg(\frac{\xi^{2}}{\omega \lambda^{\frac{6}{\alpha+1}}}\bigg)^{\frac{\alpha+1}{\alpha-5}}  \bigg]^{\frac{4}{\alpha+1}} 
	\notag \\
	& \leq  \frac{C\xi}{\lambda^{\frac{4}{\alpha+1}}} \bigg[  1 + \frac{1}{\lambda^{\frac{20}{9(\alpha+1)}}}  +  \frac{\nu^{4}}{\lambda^{\frac{4\alpha}{(\alpha+1)}}} + \bigg(\frac{\xi^{2}}{\omega \lambda^{\frac{6}{\alpha+1}}}\bigg)^{\frac{4}{\alpha-5}}  \bigg] 
	\notag \\
	& \leq \frac{C\xi}{\lambda^{\frac{4}{7}}} \bigg[1+ \frac{1}{\lambda^{\frac{20}{63}}} + \frac{\nu^4}{\lambda^4}  + \bigg(\frac{\xi^{2}}{\omega \lambda^{\frac{6}{7}}}\bigg)^{\frac{4}{(\alpha-5)}} \bigg] 
	\notag \\
	&= :A_1(\xi, \omega, \lambda). 
\end{align}
thanks to the bound \eqref{uniform-bound-4} and using the fact that 
\begin{align}\label{rel-lambda-alpha6}
	\frac{1}{\lambda^{1/(\alpha+1)}} < \frac{1}{\lambda^{1/7}} \ \text{ since $\alpha>6$ and $0<\lambda \leq 1$} . 
\end{align}

Utilizing the scaling \eqref{scaling}, we conclude
\begin{align}\label{limit-h-1}
	A_1(\xi,\omega, \lambda) =   \frac{C h^{6}}{h^{\frac{4}{7}  }} \bigg[1 + \frac{1}{h^{\frac{20}{63}}} + \frac{h^{12}}{h^4} + 
	\bigg(\frac{h^{12}}{h^3 h^{\frac{6}{7}} }\bigg)^{\frac{4}{\alpha-5}} \bigg]  \leq \frac{Ch^6}{h^{\frac{8}{9}}} .
\end{align}

\vspace*{.1cm}
\noindent 
$\bullet$ 
Secondly, we  compute the following: 
\begin{align}\label{limit-2}
	&\left|\int_0^T \int_{\B \setminus \Omega_t} \Ss_\omega : \nabla_x \boldvphi  \right|   \notag  \\
	\leq& 
	\|\nabla_x\boldvphi\|_{L^\infty((0,T)\times \B)} \left(\int_0^T \int_{\B\setminus \Omega_t} \Big(\frac{1}{\sqrt{\theta}} |\Ss_\omega|\Big)^{\frac{2\alpha+2}{2\alpha+1}}\right)^{\frac{2\alpha+1}{2\alpha+2}} \left( \int_0^T \int_{\B\setminus \Omega_t} \theta^{\alpha +1}\right)^{\frac{1}{2\alpha+2}}    \notag \\
	\leq
	&\frac{C}{\lambda^{\frac{1}{2\alpha+2}}} \bigg( \int_0^T \int_{\B\setminus \Omega_t}\lambda \theta^{\alpha+1} \bigg)^{\frac{1}{2\alpha+2}} 
	\bigg(\int_0^T \int_{\B\setminus \Omega_t} \bigg(\frac{1}{\sqrt{\theta}} \sqrt{|\Ss_\omega: \nabla_x \uvect|}\sqrt{|f_\omega|(1+\theta)}\bigg)^{\frac{2\alpha+2}{2\alpha+1}}\bigg)^{\frac{2\alpha+1}{2\alpha+2}}  \notag \\
	\leq 
	&\frac{C}{\lambda^{\frac{1}{2\alpha+2}}} \bigg( \int_0^T \int_{\B\setminus \Omega_t}\lambda \theta^{\alpha+1} \bigg)^{\frac{1}{2\alpha+2}} \notag \\ 
	& \qquad \qquad \times \bigg(\int_0^T \int_{\B\setminus \Omega_t} \frac{1}{\theta} |\Ss_\omega : \nabla_x \uvect|  \bigg)^{\frac{1}{2}} \bigg( \int_0^T \int_{\B \setminus \Omega_t} \big(f_\omega (1+\theta)\big)^{\frac{\alpha+1}{\alpha}} \bigg)^{\frac{\alpha}{2\alpha+2}}  \notag \\
	\leq 
	&\frac{C}{\lambda^{\frac{1}{2\alpha+2}}} \bigg( \int_0^T \int_{\B\setminus \Omega_t}\lambda \theta^{\alpha+1} \bigg)^{\frac{1}{2\alpha+2}} \times  \bigg(\int_0^T \int_{\B\setminus \Omega_t} \frac{1}{\theta} |\Ss_\omega : \nabla_x \uvect|  \bigg)^{\frac{1}{2}}  \notag \\
	& \quad  \times \|f_\omega\|_{L^{\frac{\alpha+1}{\alpha-1}}(((0,T)\times \B)\setminus Q_T) }^{\frac{1}{2}}
	\Bigg[ 1+ \frac{1}{\lambda^{\frac{1}{2\alpha+2}}} \bigg(\int_0^T \int_{\B\setminus \Omega_t} \lambda \theta^{\alpha+1} \bigg)^{\frac{1}{2\alpha+2}}   \Bigg] 
	\notag \\
	\leq & 
	\frac{C\sqrt{\omega}}{\lambda^{\frac{1}{2\alpha+2}}} \bigg[1+\frac{1}{\lambda^{\frac{5}{9}}} + \frac{\nu^{\alpha+1}}{\lambda^\alpha} + \bigg(\frac{\xi^{2}}{\omega \lambda^{\frac{6}{\alpha+1}}} \bigg)^{\frac{\alpha+1}{\alpha-5}}   \bigg]^{\frac{1}{2\alpha+2}}
	\bigg[1+\frac{1}{\lambda^{\frac{5}{9}}} + \frac{\nu^{\alpha+1}}{\lambda^{\alpha}} +\bigg(\frac{\xi^{2}}{\omega \lambda^{\frac{6}{\alpha+1}}}\bigg)^{\frac{\alpha+1}{\alpha-5}}   \bigg]^{\frac{1}{2}}  \notag \\ 
	& \qquad \qquad \qquad \qquad \qquad \ \ \times  \bigg[  1 +  \frac{1}{\lambda^{\frac{1}{2\alpha+2}}}  \bigg(1+\frac{1}{\lambda^{\frac{5}{9}}} + \frac{\nu^{\alpha+1}}{\lambda^\alpha}+\bigg(\frac{\xi^{2}}{\omega \lambda^{\frac{6}{\alpha+1}}}\bigg)^{\frac{\alpha+1}{(\alpha-5)}}   \bigg)^{\frac{1}{2\alpha+2}} \bigg] 
	\notag \\
	\leq & 
	\frac{C\sqrt{\omega}}{\lambda^{\frac{1}{2\alpha+2}}} 
	\bigg[1+\frac{1}{\lambda^{\frac{5}{18(\alpha+1)}}} + 
	\bigg(\frac{\nu^{\alpha+1}}{\lambda^{\alpha}}\bigg)^{\frac{1}{2(\alpha+1)}}
	+
	\bigg(\frac{\xi^{2}}{\omega \lambda^{\frac{6}{\alpha+1}}} \bigg)^{\frac{1}{2(\alpha-5)}}   \bigg] \notag \\
	& \qquad \qquad \qquad \qquad \qquad \ \ \times \bigg[ 1+ \frac{1}{\lambda^{\frac{5}{18}}} + \bigg(\frac{\nu^{\alpha+1}}{\lambda^\alpha}\bigg)^{\frac{1}{2}}  +\bigg(\frac{\xi^{2}}{\omega \lambda^{\frac{6}{\alpha+1}}  }\bigg)^{\frac{\alpha+1}{2(\alpha-5)}} \bigg] 
	\notag \\
	& + \frac{C\sqrt{\omega}}{\lambda^{\frac{1}{\alpha+1}}} 
	\bigg[1+\frac{1}{\lambda^{\frac{5}{9(\alpha+1)}}} + 
	\bigg(\frac{\nu^{\alpha+1}}{\lambda^{\alpha}}\bigg)^{\frac{1}{\alpha+1}}
	+
	\bigg(\frac{\xi^{2}}{\omega \lambda^{\frac{6}{\alpha+1}}} \bigg)^{\frac{1}{(\alpha-5)}}   \bigg] \notag \\
	& \qquad \qquad \qquad \qquad \qquad \ \ \times \bigg[ 1+ \frac{1}{\lambda^{\frac{5}{18}}} + \bigg(\frac{\nu^{\alpha+1}}{\lambda^\alpha}\bigg)^{\frac{1}{2}}  +\bigg(\frac{\xi^{2}}{\omega \lambda^{\frac{6}{\alpha+1}}  }\bigg)^{\frac{\alpha+1}{2(\alpha-5)}} \bigg]  \notag \\
	=: &  A_2(\xi, \omega, \lambda).
\end{align}
Here, we have utilized that
\begin{align*}
	|\Ss_\omega| \leq  \sqrt{|\Ss_\omega : \nabla_x \uvect| (\mu_\omega(\theta) + \eta_\omega(\theta) )} \leq C \sqrt{|\Ss_\omega : \nabla_x \uvect| |f_\omega| (1+\theta)}, 
\end{align*}
by means of the definitions of $\mu_\omega(\theta), \eta_\omega(\theta)$ given by \eqref{def-mu-omega}--\eqref{def-nu-omega} and  the hypothesis \eqref{hypo-mu}). Then, we have used the bounds \eqref{uniform-bound-4}, \eqref{uniformbound-stress} and the fact $\|f_\omega\|_{L^p(((0,T)\times \B) \setminus Q_T) }\leq c\, \omega$  for $p\geq \frac{\alpha+1}{\alpha-1}$ (see \eqref{def-f-omega}) to deduce the estimate \eqref{limit-2}.

%


Using the scaling \eqref{scaling} in \eqref{limit-2},   we further deduce that
\begin{align}\label{limit-h-2}
	A_2(\xi, \omega, \lambda) 
	&=  
	\frac{Ch^{\frac{3}{2}}}{h^{\frac{1}{2\alpha+2}}} 
	\bigg[1+\frac{1}{h^{\frac{5}{18(\alpha+1)}}} + 
	\bigg(\frac{h^{3(\alpha+1)}}{h^{\alpha}}\bigg)^{\frac{1}{2(\alpha+1)}}
	+
	\bigg(\frac{h^{12}}{ h^3 h^{\frac{6}{\alpha+1}}} \bigg)^{\frac{1}{2(\alpha-5)}}   \bigg] 
	\notag \\
	& \qquad \quad \times \bigg[ 1+ \frac{1}{h^{\frac{5}{18}}} + \bigg(\frac{h^{3(\alpha+1)}}{h^\alpha}\bigg)^{\frac{1}{2}}  +\bigg(\frac{h^{12}}{h^3 h^{\frac{6}{\alpha+1}}  }\bigg)^{\frac{\alpha+1}{2(\alpha-5)}} \bigg] 
	\notag \\
	& \ + \frac{C h^{\frac{3}{2}}}{h^{\frac{1}{\alpha+1}}} 
	\bigg[1+\frac{1}{h^{\frac{5}{9(\alpha+1)}}} + 
	\bigg(\frac{h^{3(\alpha+1}}{h^{\alpha}}\bigg)^{\frac{1}{\alpha+1}}
	+
	\bigg(\frac{h^{12}}{h^3 h^{\frac{6}{\alpha+1}}} \bigg)^{\frac{1}{(\alpha-5)}}   \bigg] 
	\notag \\
	&\qquad \qquad \times \bigg[ 1+ \frac{1}{h^{\frac{5}{18}}} + \bigg(\frac{h^{3(\alpha+1)}}{h^\alpha}\bigg)^{\frac{1}{2}}  +\bigg(\frac{h^{12}}{h^3 h^{\frac{6}{\alpha+1}}  }\bigg)^{\frac{\alpha+1}{2(\alpha-5)}} \bigg] 
	\notag \\
	&\leq \frac{C h^{\frac{3}{2}}}{h^{\frac{14}{9(\alpha+1)} + \frac{5}{18} }}  = C h^{\frac{22\alpha - 6}{18(\alpha+1)}}.
\end{align}

\vspace*{.1cm}
\noindent 
$\bullet$ Next, we look to the entropy inequality \eqref{enropy-ineq-recover}. We observe that 
\begin{align}\label{limit-3}
	&\left|	\int_0^T \int_{\B \setminus \Omega_t} \frac{4a_\xi}{3}\theta^3(\partial_t \vphi +  \uvect\cdot \nabla_x \vphi ) \right| \notag \\
	\leq & \frac{C\xi}{\lambda^{\frac{3}{\alpha+1}}} \bigg(\int_0^T \int_{\B \setminus \Omega_t} \lambda \theta^{\alpha+1} \bigg)^{\frac{3}{\alpha+1}} + C\xi \int_0^T \int_{\B\setminus \Omega_t}|\uvect|^2 + C \xi \int_0^T \int_{\B \setminus \Omega_t} \theta^6 
	\notag \\
	\leq &    \frac{C\xi}{\lambda^{\frac{3}{\alpha+1}}} \bigg(\int_0^T \int_{\B \setminus \Omega_t} \lambda \theta^{\alpha+1} \bigg)^{\frac{3}{\alpha+1}} + C\xi \int_0^T \int_{\B\setminus \Omega_t}|\uvect|^2 + \frac{C \xi}{\lambda^{\frac{6}{\alpha+1}}} \bigg(\int_0^T \int_{\B \setminus \Omega_t} \lambda \theta^{\alpha+1}\bigg)^{\frac{6}{\alpha+1}}          
	\notag  \\
	\leq 
	& 
	\frac{C\xi}{\lambda^{\frac{3}{\alpha+1}}} \bigg[1+\frac{1}{\lambda^{\frac{5}{9}}} +  \frac{\nu^{\alpha+1}}{\lambda^\alpha} + \bigg(\frac{\xi^{2}}{\omega \lambda^{\frac{6}{\alpha+1}}}  \bigg)^{\frac{\alpha+1}{\alpha-5} }  \bigg]^{\frac{3}{\alpha+1}} 
	+\frac{C\xi}{\omega} \bigg[1+\frac{1}{\lambda^{\frac{5}{9}}} +  \frac{\nu^{\alpha+1}}{\lambda^\alpha} + \bigg(\frac{\xi^{2}}{\omega \lambda^{\frac{6}{\alpha+1}}}  \bigg)^{\frac{\alpha+1}{\alpha-5} }  \bigg]  \notag \\
	& \qquad \qquad + \frac{C\xi}{\lambda^{\frac{6}{\alpha+1} } } \bigg[1+\frac{1}{\lambda^{\frac{5}{9}}} +  \frac{\nu^{\alpha+1}}{\lambda^\alpha} + \bigg(\frac{\xi^{2}}{\omega \lambda^{\frac{6}{\alpha+1}}}  \bigg)^{\frac{\alpha+1}{\alpha-5} }  \bigg]^{\frac{6}{\alpha+1}} 
	\notag \\
	\leq & 
	\frac{C\xi}{\lambda^{\frac{3}{\alpha+1}}} \bigg[1+\frac{1}{\lambda^{\frac{15}{9(\alpha+1)}}} +  \frac{\nu^3}{\lambda^{\frac{3\alpha}{\alpha+1}}} + \bigg(\frac{\xi^{2}}{\omega \lambda^{\frac{6}{\alpha+1}}}  \bigg)^{\frac{3}{\alpha-5} }  \bigg] 
	+\frac{C\xi}{\omega} \bigg[1+\frac{1}{\lambda^{\frac{5}{9}}} +  \frac{\nu^{\alpha+1}}{\lambda^\alpha} + \bigg(\frac{\xi^{2}}{\omega \lambda^{\frac{6}{\alpha+1}}}  \bigg)^{\frac{\alpha+1}{\alpha-5} }  \bigg]  
	\notag \\
	&    \qquad \qquad   + \frac{C\xi}{\lambda^{\frac{6}{\alpha+1} } } \bigg[1+\frac{1}{\lambda^{\frac{10}{3(\alpha+1)}}} +  \frac{\nu^{6}}{\lambda^\frac{6\alpha}{\alpha+1}  } + \bigg(\frac{\xi^{2}}{\omega \lambda^{\frac{6}{\alpha+1}}}  \bigg)^{\frac{6}{\alpha-5} }  \bigg]
	\notag \\
	=: & A_3(\xi,\omega, \lambda) ,
\end{align}
thanks to the estimates \eqref{uniform-bound-4}, \eqref{uniform-bound-5} and the fact that $\alpha>6$.

By means of  \eqref{scaling}, we deduce 
\begin{align}\label{limit-h-3}
	A_3(\xi, \omega, \lambda) & = \frac{C h^{6}}{h^{\frac{3}{7}}}  \bigg[1+ \frac{1}{h^{\frac{5}{21}}} + \frac{h^9}{h^{\frac{3\alpha}{\alpha+1}}} + \bigg(\frac{h^{12}}{h^3 h^{\frac{6}{7}}} 
	\bigg)^{\frac{3}{\alpha-5}}  \bigg]   \notag \\   
	& \qquad \qquad + Ch^3 \bigg[1+ \frac{1}{h^{\frac{5}{9}}}  + \frac{h^{3(\alpha+1)}}{h^{\alpha}} +  \bigg(\frac{h^{12}}{h^3 h^{\frac{6}{7}}} 
	\bigg)^{\frac{\alpha+1}{\alpha-5}} \bigg] 
	\notag \\
	& \qquad \qquad + \frac{C h^{6}}{h^{\frac{6}{7}}} \bigg[1 + \frac{1}{h^{\frac{10}{21}}} + \frac{h^{18}}{h^{\frac{6\alpha}{\alpha+1}}} + 
	\bigg(\frac{h^{12}}{h^3 h^{\frac{6}{7}}} 
	\bigg)^{\frac{6}{\alpha-5}}
	\bigg] \notag \\
	& \leq \frac{Ch^{6}}{h^{\frac{14}{21}}} + \frac{Ch^{3}}{h^{\frac{5}{9}}}  + \frac{Ch^{6}}{h^{\frac{28}{21}}} .
\end{align}

\vspace*{.1cm} 
\noindent 
$\bullet$ The next term in the entropy inequality \eqref{enropy-ineq-recover} in $\B\setminus \Omega_t$ satisfies  the following:
\begin{align}\label{limit-4}
	&	\left|\int_0^T \int_{\B \setminus \Omega_t} \frac{\kappa_\nu(\theta,t,x)}{\theta} \nabla_x \theta \cdot \nabla_x \vphi \right| 
	\notag \\
	&	\leq \bigg(\int_0^T \int_{\B \setminus \Omega_t} \frac{\kappa_\nu(\theta,t,x)}{\theta^2} |\nabla_x \theta|^2\bigg)^{\frac{1}{2}} \bigg(\int_0^T \int_{\B \setminus \Omega_t} \kappa_\nu(\theta,t,x) |\nabla_x \vphi|^2 \bigg)^{\frac{1}{2}}
	\notag\\
	& \leq C  \|\nabla_x \vphi\|_{L^\infty((0,T)\times \B)} \left( 1 + \frac{1}{\lambda^{\frac{5}{14}}  } + \Big(\frac{\xi^{2}}{\omega \nu}\Big)^{\frac{\alpha}{2(\alpha-6)}}  \right) \sqrt{\nu} \bigg(\int_0^T \int_{\B\setminus \Omega_t}  (1+ \theta^\alpha)   \bigg)^{\frac{1}{2}} 
	\notag\\
	& \leq C \sqrt{\nu} \bigg[ 1 + \frac{1}{\lambda^{\frac{5}{9}}  } + \frac{\nu^{\alpha+1}}{\lambda^\alpha} 
	+ \bigg(\frac{\xi^{2}}{\omega \lambda^{\frac{6}{\alpha+1}}}\bigg)^{\frac{\alpha+1}{\alpha-5}}  \bigg]^{\frac{1}{2}} \left(1+ \frac{1}{\lambda^{\frac{\alpha}{2(\alpha+1)}}} \Big(\int_0^T \int_{\B} \lambda \theta^{\alpha+1}\Big)^{\frac{\alpha}{2(\alpha+1)} }  \right) 
	\notag\\
	& \leq C \sqrt{\nu} \bigg[ 1 + \frac{1}{\lambda^{\frac{5}{9}}  } + \frac{\nu^{\alpha+1}}{\lambda^\alpha} 
	+ \bigg(\frac{\xi^{2}}{\omega \lambda^{\frac{6}{\alpha+1}}}\bigg)^{\frac{\alpha+1}{\alpha-5}}  \bigg]^{\frac{1}{2}}
	\notag\\
	& \ \ \ 
	+ \frac{C \sqrt{\nu}}{\sqrt{\lambda}}  \bigg[ 1 + \frac{1}{\lambda^{\frac{5}{9}}  } + \frac{\nu^{\alpha+1}}{\lambda^\alpha} 
	+ \bigg(\frac{\xi^{2}}{\omega \lambda^{\frac{6}{\alpha+1}}}\bigg)^{\frac{\alpha+1}{\alpha-5}}  \bigg]^{ \frac{2\alpha+1}{2(\alpha+1)}  }  
	\notag\\
	&  = : A_4(\nu, \xi, \omega, \lambda) . 
\end{align}
We further compute (thanks to the choice \eqref{scaling}) 
\begin{align}\label{limit-h-4}
	A_4(\nu, \xi, \omega, \lambda) &= C h^{\frac{3}{2}} \bigg[1 + \frac{1}{h^{\frac{5}{18}}} + \frac{h^{\frac{3}{2}(\alpha+1)}}{h^{\frac{\alpha}{2}}} + \bigg(\frac{h^{12}}{ h^{3} h^{\frac{6}{\alpha+1}} }\bigg)^{\frac{\alpha+1}{2(\alpha-5)}}   \bigg]  \notag \\ 
	& \ + C h \bigg[ 1 + \frac{1}{h^{\frac{5 (2\alpha+1)}{ 18(\alpha+1)}} } +  \frac{ h^{\frac{3}{2}(2\alpha+1)} }{h^{\frac{1}{2}(2\alpha+1)}}   +   \bigg(\frac{h^{12}}{ h^{3} h^{\frac{6}{\alpha+1}} }\bigg)^{\frac{2\alpha+1}{2(\alpha-5)}}     \bigg]  \notag \\
	&\leq \frac{C h^{\frac{3}{2}}}{h^{\frac{5}{18}}} + \frac{Ch}{h^{\frac{5 (2\alpha+1)}{ 18(\alpha+1)}}}.
\end{align}

\vspace*{.1cm} 
\noindent 
$\bullet$ In a similar way, we can bound the following terms appearing in the energy inequality \eqref{energy-pena-recover}. Indeed,  the estimates \eqref{limit-1} and \eqref{limit-3} yield
\begin{align} \label{limit-5}
	\left|	\int_0^T \int_{\B\setminus\Omega_t} a_\xi\Big(\theta^4 - \frac{4}{3}\theta^3 \widetilde{\theta}\Big) \partial_t \psi \right| 
	+ \left| 	\int_0^T \int_{\B \setminus \Omega_t} \frac{4a_\xi}{3}\theta^3  \Big(\partial_t \widetilde{\theta} + \uvect \cdot \nabla_x \widetilde{\theta}  \Big) \psi  \right| 
	\notag \\
	+ \left| 	\int_0^T \int_{\B \setminus \Omega_t} \frac{a_\xi}{3} \theta^4 \Div_x \V \psi \right| \leq C \left(A_1(\xi, \omega, \lambda) + A_3(\xi, \omega, \lambda)\right)   .
\end{align}
On the other hand, the information \eqref{limit-2} and \eqref{limit-4} give
\begin{align}\label{limit-6}
	\left| \int_0^T \int_{\B \setminus \Omega_t} \Ss_\omega : \nabla_x \V \psi    \right| + 	\left| \int_0^T \int_{\B \setminus \Omega_t} \frac{\kappa_\nu(\theta, t, x)}{\theta} \nabla_x \theta \cdot \nabla_x \widetilde{\theta} \psi   \right| \notag \\
	\leq C \left(A_2(\xi, \omega, \lambda) + A_4(\nu, \xi, \omega, \lambda)\right)  .
\end{align}

\subsubsection{Step 2. Finding suitable estimate for the term $\lambda \theta^{\alpha}$ in the entropy balance}

In this step, we shall find a suitable bound of the term $\disp \int_0^T \int_{\B \setminus \Omega_t} \lambda \theta^{\alpha} \vphi$ appearing in the entropy inequality \eqref{enropy-ineq-recover}. To do that, we need a more precise bound of the term $\|\lambda \theta^\alpha \|_{L^1((0,T)\times \B)}$. 

%

In fact,  we have already computed that the integrals in $\B\setminus \Omega_t$ are ``small"  w.r.t.  all the parameters,   and as per construction there is no parameter involved in the integrals in $\Omega_t$. We shall make use of these information to get a sharper estimate  of $\|\lambda \theta^\alpha \|_{L^1((0,T)\times \B)}$.


Let us first rewrite the  energy inequality \eqref{energy-pena-recover} in the following form  (which can be written by using the similar strategy as we have used to obtain \eqref{energy-pena-3}) :
\begin{align}\label{energy-pena-recover-2}
	&\int_{\Omega_t} \left(\frac{1}{2} \rho |\uvect |^2  + \rho e(\rho, \theta) - \rho s(\rho, \theta)\widetilde{\theta}  
	+ \frac{\delta}{\beta-1} \rho^\beta\right)(\tau, \cdot)
	+ \int_0^\tau \int_\B \lambda \theta^{\alpha+1}   
	\notag \\
	&	+ \int_0^\tau \int_{\Omega_t} \frac{\widetilde{\theta}}{\theta} \left( \Ss : \nabla_x \uvect + \frac{\kappa(\theta,t,x)}{\theta} |\nabla_x \theta|^2     \right) 
	\notag \\	
	\leq &
	\int_{\Omega_0} \left(\frac{1}{2}  \frac{|(\rho\uvect)_{0,\delta}|^2}{\rho_{0,\delta}}  + 	\rho_{0,\delta} e(\rho_{0,\delta}, \theta_{0,\delta}) - \rho_{0,\delta} s(\rho_{0,\delta}, \theta_{0,\delta}) \widetilde{\theta}(0)  + \frac{\delta}{\beta-1} \rho^\beta_{0,\delta}   \right) \notag \\ 
	&  - \int_{\Omega_0} (\rho \uvect)_{0,\delta} \V(0)     +\int_{\B\setminus \Omega_0}  a_\xi \Big(\theta^4_{0,\delta} -  \frac{4}{3}\theta^3_{0,\delta} \widetilde{\theta}(0)  \Big)
	\notag \\	
	& + \int_0^\tau \int_\B \lambda \theta^{\alpha} \widetilde{\theta} + \int_{\Omega_t} (\rho\uvect \cdot \V)(\tau, \cdot) \notag \\ 
	&	+\left| \int_0^\tau  \int_{\Omega_t} \left(\rho[\uvect \otimes \uvect]    : \nabla_x \V  - \Ss : \nabla_x \V  + p_{\delta}(\rho,\theta)\, \Div_x \V   +  \rho \uvect\cdot \partial_t \V \right) \right| 
	\notag \\
	& + \left| 		\int_0^\tau  \int_{\Omega_t} \left[\rho s(\rho,\theta) \Big(\partial_t \widetilde{\theta}  + \uvect \cdot \nabla_x \widetilde{\theta}  \Big) - \frac{\kappa(\theta, t, x)}{\theta} \nabla_x \theta \cdot \nabla_x \widetilde{\theta}  \right] \right|
	\notag\\
	&  + \left|\int_0^\tau \int_{\Omega_t} \rho \nabla_x \Psi \cdot \V   \right|
	+ \left|\int_0^\tau \int_{\B\setminus \Omega_t} \Ss_\omega : \nabla_x \V \right| +  \left|\int_0^\tau  \int_{\B\setminus \Omega_t} \frac{a_\xi}{3} \theta^4 \Div_x \V \right| 
	\notag \\
	& +	\left|\int_0^\tau  \int_{\B \setminus \Omega_t} \frac{4a_\xi}{3}\theta^3  \Big(\partial_t \widetilde{\theta}  + \uvect \cdot \nabla_x \widetilde{\theta}  \Big) \right|
	+ \left|\int_0^\tau  \int_{\B \setminus \Omega_t} \frac{\kappa_\nu(\theta, t, x)}{\theta} \nabla_x \theta \cdot \nabla_x \widetilde{\theta}   \right|   ,
\end{align}
for almost all  $\tau\in (0,T)$.

\vspace*{.15cm}
\noindent 
$\bullet$ Now, the estimations of the terms  $\disp \int_{\Omega_t} (\rho \uvect \cdot \V)(\tau, \cdot)$, $\disp\int_0^\tau \int_{\Omega_t} \rho \nabla_x \Psi \cdot \V$,
$\disp \int_0^\tau \int_{\Omega_t} \rho [\uvect \otimes \uvect]: \nabla_x \V$, $\disp \int_0^\tau \int_{\Omega_t} \rho \uvect\cdot \partial_t \V$ and   $\disp \int_0^\tau \int_\B \lambda \theta^{\alpha} \widetilde{\theta}$ can be done in a similar fashion  as previous; see \eqref{esti-1}, \eqref{esti-nabla-Psi-V},  \eqref{esti-3}, \eqref{esti-4} and \eqref{esti-5} respectively (since all of those estimates are uniform w.r.t.  $\lambda$).

\vspace*{.15cm}
\noindent 
$\bullet$  We also recall that all the terms in fluid domain $Q_T$ is independent of the parameters $\omega, \xi, \nu$.  In what follows, let us estimate 
\begin{align}\label{esti-final-1} 
	\int_0^\tau \int_{\Omega_t} \Ss : \nabla_x \V 
	&\leq   \frac{1}{2} \int_0^\tau \int_{\Omega_t} \frac{\widetilde{\theta}}{\theta} \Ss : \nabla_x \uvect + C(\V, \widetilde{\theta}) \int_0^\tau \int_{\Omega_t} \theta \notag  \\
	& \leq \frac{1}{2} \int_0^\tau \int_{\Omega_t}  \frac{\widetilde{\theta}}{\theta} \Ss : \nabla_x \uvect + 
	C(\V,a,\widetilde{\theta}) \bigg(1+\int_0^\tau \int_{\Omega_t} a \theta^4 \bigg)  \notag \\   
	& \leq \frac{1}{2} \int_0^\tau \int_{\Omega_t}  \frac{\widetilde{\theta}}{\theta} \Ss : \nabla_x \uvect + 
	C(\V,a, \widetilde{\theta})\bigg(1+ \int_0^\tau \int_{\Omega_t} \rho e (\rho, \theta) \bigg),
\end{align}
since $a \theta^4  \leq \rho e(\rho, \theta)$. Note that, the term $\disp \frac{1}{2} \int_0^\tau \int_{\Omega_t}  \frac{\widetilde{\theta}}{\theta} \Ss : \nabla_x \uvect$ 
can be absorbed by the associated term in the l.h.s. of \eqref{energy-pena-recover-2}.

\vspace*{.15cm}
\noindent
$\bullet$   At this step, we first recall the choice of test function $\widetilde{\theta}\in \C^1([0,T]\times \overline{\B})$ 
from \eqref{choice-tilde-theta-B}--\eqref{boundary-new-theta-B}, Section \ref{Section-penalized-wk}.

In what follows, we have 
\begin{equation*} 
	\begin{aligned}
		\int_0^\tau \int_{\Omega_t} \frac{\kappa(\theta)}{\theta} \nabla_x \theta \cdot \nabla_x \widetilde \theta 
		&= \int_0^\tau \int_{\Omega_t}  \nabla_x K(\theta) \cdot \nabla_x \widetilde \theta  \\
		& = -\int_0^\tau \int_{\Omega_t}  K(\theta) \cdot \Delta_x \widetilde \theta  + \int_0^\tau \int_{\Gamma_t} K(\theta_B) \nabla_x \theta_B \cdot \mathbf n ,
	\end{aligned}
\end{equation*} 
where $\frac{\partial}{\partial \theta}K(\theta) = \frac{\kappa(\theta)}{\theta}$.   But thanks to \eqref{test-function-theta_B}, we have $\Delta_x \widetilde \theta =0$ in $\Omega_t$ and $\widetilde \theta |_{\Gamma_t}=\theta_B$ for each $t\in [0,T]$,  and as a consequence, we get 
\begin{equation}
	\bigg| \int_0^\tau \int_{\Omega_t} \frac{\kappa(\theta)}{\theta} \nabla_x \theta \cdot \nabla_x \widetilde \theta   \bigg| \leq C(\theta_B).
\end{equation}

\vspace*{.15cm}
\noindent 
$\bullet$ Let us now recall the pressure term 
$\disp p_\delta(\rho, \theta) = p_{M}(\rho, \theta) + \frac{a}{3}\theta^4 + \delta \rho^\beta$ and the bound of $p_M$ given by \eqref{mole_bound_p_M}. To this end,  we find  
\begin{align} \label{estimate-pressue-final}
	\left|\int_0^\tau \int_{\Omega_t}  p_{\xi, \delta} (\rho, \theta) \Div_x \V \right| 
	&	\leq C(\V) \int_0^\tau \int_{\Omega_t} \bigg(\frac{\delta}{\beta-1} \rho^\beta + a \theta^4 + \rho^{\frac{5}{3}} + \theta^{\frac{5}{2}} \bigg) \notag \\
	& \leq C(\V, p_\infty, a) \int_0^\tau\int_{\Omega_t} \bigg(\frac{\delta}{\beta-1} \rho^\beta + \rho e(\rho, \theta) + 1 \bigg). 
\end{align}

\vspace*{.15cm}
\noindent 
$\bullet$ Next, using  the bounds \eqref{Fact-2}--\eqref{Fact-3},  we deduce that
\begin{align}\label{estimate-entrpy-term-final}
	&\int_0^\tau \int_{\Omega_t} \rho s(\rho, \theta)|\uvect| \notag \\
	& \leq   \int_0^\tau \int_{\Omega_t} \rho |\uvect| + C \int_0^\tau \int_{\Omega_t} \rho |\uvect| [\log \theta]^{+} + C \int_0^\tau a \theta^3 |\uvect|  \notag  \\
	&\leq  C(\rho_0) + C\int_0^\tau \int_{\Omega_t} \rho |\uvect|^2 + \frac{C}{\epsilon} \int_0^\tau \int_{\Omega_t} \rho^2 ([\log \theta]^+)^2   + 2\epsilon \int_0^\tau \int_{\Omega_t} |\uvect|^2 
	\notag \\ 
	& \hspace{7cm} +  \frac{C(a)}{\epsilon} \int_0^\tau \int_{\Omega_t} \theta^6 
	\notag \\
	&\leq C(\rho_0) + C\int_0^\tau \int_{\Omega_t} \rho |\uvect|^2 +  \frac{C(a)}{ \epsilon} \int_0^\tau \int_{\Omega_t}  a \theta^4 
	+ 2\epsilon \int_0^\tau  \|\uvect\|^2_{W^{1,2}(\Omega_t;\mathbb R^3)}  \notag \\
	& \hspace{7cm} + \frac{C(a)}{\epsilon} \int_0^\tau \|\theta^3 \|^2_{L^2(\Omega_t)}  ,
\end{align}
using the fact that $\rho\leq \theta^{3/2}$ and $[\log \theta]^+ \leq \theta^{1/2}$  in the integral  $\int_0^\tau \int_{\Omega_t} \rho^2 ([\log \theta]^+)^2$. Also, we shall  use that $a\theta^4 \leq \rho e(\rho, \theta)$ in the above estimate.

Furthermore, 
\begin{align}\label{new-term-1}
	\| \theta^3 \|^2_{L^2(\Omega_t)} = \int_{\theta \leq \widehat{K} } \theta^6 + \int_{\theta > \widehat{K}} \theta^6 
	& \leq |\Omega_t| \widehat{K}^6 +   \widehat{K}^{6-\alpha} \int_{\Omega_t} \theta^\alpha \notag \\
	& \leq |\B| \widehat{K}^6 +   \widehat{K}^{6-\alpha} \int_{\Omega_t} \theta^\alpha. 
\end{align}
Then, by using H\"{o}lder and Poincar\'{e} inequalities, we have 
\begin{align}\label{new-term-2}
	\int_{\Omega_t} \theta^\alpha \lesssim \|\theta^\alpha\|_{L^3(\Omega_t)} = 
	\|\theta^{\frac{\alpha}{2}}\|^2_{L^6(\Omega_t)} \lesssim \|\theta^{\frac{\alpha}{2}}\|^2_{W^{1,2}(\Omega_t)} &\lesssim \|\nabla_x \theta^{\frac{\alpha}{2}}\|^2_{L^{2}(\Omega_t)} + C(\theta_B) \notag \\
	& \lesssim \int_{\Omega_t} \theta^{\alpha-2} |\nabla_x \theta|^2 + C(\theta_B)  .
\end{align}

Now, as we described in \eqref{lower-bound-S-w-u}--\eqref{u_W_12}, one has 
\begin{align}\label{esti-uvect-final}
	\int_0^\tau \| \uvect \|^2_{W^{1,2}(\Omega_t;\mathbb R^3)} \leq C(\widetilde \theta) \int_0^\tau \int_{\Omega_t} \frac{\widetilde \theta}{\theta} \Ss : \nabla_x \uvect + C(\rho_0, \widetilde \theta) \int_0^\tau \int_{\Omega_t} \rho |\uvect|^2 .
\end{align}
Moreover, 
\begin{align}\label{new-term-4}
	\int_0^\tau \int_{\Omega_t} \widetilde{\theta} \,  \kappa(\theta, t, x) \frac{|\nabla_x \theta|^2}{|\theta|^2} \gtrsim   \inf_{\overline{(0,T)\times \B}} |\widetilde{\theta}|   \int_0^\tau \int_{\Omega_t} \left(\frac{1}{\theta^2} + \theta^{\alpha-2}  \right)|\nabla_x \theta|^2 .
\end{align} 

To this end, we first fix $\epsilon>0$ small enough  in \eqref{estimate-entrpy-term-final} and since $\alpha>6$, we  choose $\widehat K >0$ large enough in \eqref{new-term-1}, so that, by means of \eqref{new-term-2} and \eqref{esti-uvect-final}--\eqref{new-term-4}, the terms  $\disp \int_0^\tau \|\uvect\|^2_{W^{1,2}(\Omega_t;\mathbb R^3)}$ and $\disp \int_0^\tau \|\theta^3\|^2_{L^2(\Omega_t)}$ can be dominated by the left hand side of \eqref{energy-pena-recover-2}.

\vspace*{.15cm}
\noindent
$\bullet$ Further,  one can  deduce that 
\begin{align}\label{esti-initial}
	\left| \int_{\B\setminus \Omega_0} a_\xi \Big(\theta^4_{0,\delta} - \frac{4}{3} \theta^3_{0,\delta}\widetilde \theta(0)  \Big) \right| \leq C(\theta_0, \widetilde \theta) \xi ,
\end{align}
for fixed $0<\delta<1$.

\vspace*{.15cm}
\noindent
$\bullet$
Using all the above estimates and the bounds of the terms in $\B\setminus \Omega_t$ from \eqref{limit-5}--\eqref{limit-6}, we have from \eqref{energy-pena-recover-2} (by applying Gr\"onwall's lemma)
\begin{align}\label{energy-pena-recover-3}
	&\int_{\Omega_t} \left(\frac{1}{2} \rho |\uvect |^2  + \rho e(\rho, \theta) - \rho s(\rho, \theta) \widetilde \theta  
	+ \frac{\delta}{\beta-1} \rho^\beta\right)(\tau, \cdot)
	+ \int_0^\tau \int_\B \lambda \theta^{\alpha+1}   \notag \\
	&	+ \int_0^\tau \|\uvect \|^2_{W^{1,2}(\Omega_t)} + \int_0^\tau \int_{\Omega_t} \widetilde \theta \, \frac{\kappa(\theta,t,x)}{\theta^2}  |\nabla_x \theta|^2 
	\notag \\	
	\leq &
	C\int_{\Omega_0} \bigg(\frac{1}{2}  \frac{|(\rho\uvect)_{0,\delta}|^2}{\rho_{0,\delta}}  + 	\rho_{0,\delta} e(\rho_{0,\delta}, \theta_{0,\delta}) - \rho_{0,\delta} s(\rho_{0,\delta}, \theta_{0,\delta}) \widetilde \theta(0)  + \frac{\delta}{\beta-1} \rho^\beta_{0,\delta} \bigg) \notag \\ 
	& - C\int_{\Omega_0} (\rho \uvect)_{0,\delta} \V(0)  + C
	+ C(\widetilde \theta, \theta_0) \xi \notag \\ 
	& + C \left( A_1(\xi, \omega, \lambda)  +  A_2(\xi, \omega, \lambda) +  A_3(\xi, \omega, \lambda) +  A_4(\nu, \xi, \omega, \lambda)  \right) ,
\end{align}
for almost all  $\tau\in (0,T)$,  where the constant $C>0$ does not depend on any of the parameters $\lambda, \omega, \nu, \xi$.

\paragraph{Bounds of the terms $\lambda \theta^{\alpha}$.} 
Let us recall the entropy balance \eqref{enropy-ineq-recover}. The only term left to estimate is the integral concerning $\lambda \theta^{\alpha}$. Indeed, we have 
\begin{align}\label{bound-lambda-theta-4}
	&\left|\int_0^T \int_\B \lambda \theta^{\alpha} \vphi \right| \notag \\
	&\leq  \lambda^{\frac{1}{\alpha+1}} \|\vphi\|_{L^\infty((0,T)\times \B)}\bigg(\int_0^T \int_\B \lambda \theta^{\alpha+1} \bigg)^{\frac{\alpha}{\alpha+1}} \notag \\
	&\leq C \lambda^{\frac{1}{\alpha+1}} \Big(1+ \xi +  A_1(\xi, \omega, \lambda)  +  A_2(\xi, \omega, \lambda) +  A_3(\xi, \omega, \lambda) +  A_4(\nu, \xi, \omega, \lambda) \Big)^{\frac{\alpha}{\alpha+1}},
\end{align} 
where we have used  the estimate \eqref{energy-pena-recover-3}, and it is clear that  constant $C>0$ is independent on  the parameters $\lambda, \xi, \omega, \nu, \delta$.

\subsubsection{Passing to the limits of $\omega,\xi, \nu, \lambda$}
Now, we are in position to pass to the limits of all the parameters $\omega$, $\xi$, $\nu$, $\lambda$ together.  
We keep in mind the scaling introduced in \eqref{scaling} w.r.t. $h$. Then,  recall the points \eqref{limit-1}--\eqref{limit-h-1},  \eqref{limit-2}--\eqref{limit-h-2}, \eqref{limit-3}--\eqref{limit-h-3}, \eqref{limit-4}--\eqref{limit-h-4}, from which it is not difficult to observe that 
\begin{align}\label{limits-all}
	\begin{dcases}
		A_1(\xi,\omega, \lambda) \leq  \frac{Ch^6}{h^{\frac{8}{9}}}
		\to 0 \text{ as } h\to 0, \\
		A_2(\xi, \omega, \lambda) \leq Ch^{\frac{22\alpha-6}{18(\alpha+1)}} \to 0 \ \ \text{ as } h\to 0, \\
		A_3(\xi, \omega, \lambda) \leq \frac{Ch^6}{h^{\frac{14}{21}}} + \frac{Ch^3}{h^{\frac{5}{9}}} + \frac{Ch^6}{h^{\frac{28}{21}}} 
		\to 0  \ \ \text{ as } h\to 0, 
		\\
		A_4(\nu, \xi, \omega, \lambda) \leq  \frac{C h^{\frac{3}{2}}}{h^{\frac{5}{18}}} + \frac{Ch}{h^{\frac{5 (2\alpha+1)}{ 18(\alpha+1)}}}
		\to 0 \ \  \text{ as } h \to 0  .
	\end{dcases}
\end{align} 

Next, from the estimate \eqref{limit-5}, one has 
\begin{align*}
	\begin{dcases} 
		\int_0^T \int_{\B\setminus\Omega_t} a_\xi\Big(\theta^4 - \frac{4}{3}\theta^3 \widetilde \theta\Big) \partial_t \psi \to 0 \ \ \text{ as } h \to 0, \\ 
		\int_0^T \int_{\B \setminus \Omega_t} \frac{4a_\xi}{3}\theta^3  \Big(\partial_t \widetilde \theta  + \uvect \cdot \nabla_x \widetilde \theta  \Big) \psi  \to 0 \ \ \text{ as } h\to 0, \\
		\int_0^T \int_{\B \setminus \Omega_t} \frac{a_\xi}{3} \theta^4 \Div_x \V \psi \to 0 \ \ \text{ as } h\to 0,
	\end{dcases} 
\end{align*}
and analogously, from \eqref{limit-6} we have 
\begin{align*}
	\begin{dcases} 
		\int_0^T \int_{\B \setminus \Omega_t} \Ss_\omega : \nabla_x \V \psi  \to 0  \ \ \text{ as } h \to 0, \\
		\int_0^T \int_{\B \setminus \Omega_t} \frac{\kappa_\nu(\theta, t, x)}{\theta} \nabla_x \theta \cdot \nabla_x \widetilde \theta \, \psi \to 0 \ \ \text{ as } h \to 0,
	\end{dcases} 
\end{align*}
in the energy inequality  \eqref{energy-pena-recover}. 

Also, from \eqref{esti-initial}, one has (in the energy inequality \eqref{energy-pena-recover})
\begin{align*}
	\int_{\B\setminus \Omega_0} a_\xi \Big(\theta^4_{0,\delta} - \frac{4}{3} \theta^3_{0,\delta} \widetilde \theta (0)  \Big) \to 0 \ \ \text{ as } h \to 0.
\end{align*}
We can also show that
\begin{align*}
	\int_{\B\setminus \Omega_0} \frac{4a_\xi}{3} \theta^3_{0,\delta} \vphi(0, \cdot) \to 0 \ \ \text{ as } h \to 0,
\end{align*}
in the entropy inequality \eqref{enropy-ineq-recover}.

Now, thanks to the estimate \eqref{bound-lambda-theta-4}, and the  limiting behaviors of $A_1$, $A_2$, $A_3$ and $A_4$ from \eqref{limits-all},
we have 
\begin{align}\label{limit-h-5}
	\int_0^T \int_\B \lambda \theta^{\alpha} \vphi \to 0 \quad \text{as } h\to 0 ,
\end{align}
in the entropy inequality \eqref{enropy-ineq-recover}.
Similarly, the term  $\displaystyle \int_0^T \int_{\B} \lambda\theta^{\alpha} \widetilde \theta \psi$ in the energy estimate \eqref{energy-pena-recover} vanishes as $h\to 0$.

\subsubsection{Resultant weak formulations.}
Taking into account all the above limits  of $\xi$, $\omega$, $\nu$, $\lambda$ to $0$ (equivalently $h\to 0$) from the previous subsection,   we now  write the weak formulations and the ballistic energy inequality for our system.

\vspace*{.1cm}
\noindent 
$\bullet$  The  weak formulation of the continuity equation will be the same as \eqref{weak-conti-recover} after passing to the limits.

\vspace*{.1cm}
\noindent 
$\bullet$
The weak formulation of the momentum equation becomes (from \eqref{weak-momen-recover}) 
\begin{align}\label{weak-momen-recover-2}
	-	\int_0^T \int_{\Omega_t} \left( \rho \uvect \cdot \partial_t  \boldvphi   +  \rho[\uvect \otimes \uvect] : \nabla_x  \boldvphi + p_{\delta}(\rho, \theta) \Div_x \boldvphi   \right)   
	+\int_0^T \int_{\Omega_t} \Ss : \nabla_x \boldvphi \notag  \\
	= \int_0^T \int_{\Omega_t} \rho \nabla_x \Psi \cdot \boldvphi +
	\int_{\Omega_0}   (\rho \uvect)_{0,\delta}\cdot  \boldvphi(0, \cdot) 
\end{align}
for any test function $\boldvphi$ satisfying 
\begin{align*}
	\boldvphi \in \C^\infty_c([0,T]\times \B;\mathbb R^3) \quad \text{with} \quad \boldvphi(\tau, \cdot) \cdot \mathbf n \big|_{\Gamma_\tau} =0 \ \ \text{for any } \tau\in [0,T].
\end{align*}

\vspace*{.1cm}
\noindent 
$\bullet$ The weak formulation for the Poisson equation has  the same expression as  \eqref{pena-weak-poisson-recover} after the limiting process. So, we do not write it again.

\vspace*{.1cm}
\noindent 
$\bullet$ The entropy inequality can be written as (from \eqref{enropy-ineq-recover}) 
\begin{align}\label{entropy-ineq-recover-2}
	&	-\int_0^T \int_{\Omega_t} \rho s(\rho,\theta) \left(\partial_t \vphi + \
	\uvect \cdot \nabla_x \vphi \right) 	
	+ \int_0^T \int_{\Omega_t} \frac{\kappa(\theta, t,x)}{\theta} \nabla_x \theta \cdot \nabla_x \vphi 
	\notag \\
	& \ 	- \int_{\Omega_0} \rho_{0,\delta} s(\rho_{0,\delta}, \theta_{0,\delta}) \vphi(0,\cdot) 
	\geq \int_0^T \int_{\Omega_t} \frac{\vphi}{\theta}  \left( \Ss : \nabla_x \uvect 
	+ \frac{\kappa(\theta, t,x)}{\theta} |\nabla_x \theta|^2   \right),
\end{align}
for any test function $\vphi \in \C^1_c([0,T)\times \B)$ with $\vphi\geq 0$.

\vspace*{.1cm}
\noindent 
$\bullet$  
The ballistic energy inequality \eqref{energy-pena-recover} now reduces to the following:
\begin{align}\label{energy-pena-recover-4}
	& -	\int_0^T\int_{\Omega_t} \partial_t \psi\left(\frac{1}{2} \rho |\uvect |^2  + \rho e(\rho, \theta) - \rho s(\rho, \theta)\widetilde \theta  
	+ \frac{\delta}{\beta-1} \rho^\beta\right)
	\notag \\
	&   
	+ \int_0^T \int_{\Omega_t} \frac{\widetilde \theta}{\theta} \left( \Ss : \nabla_x \uvect  + \frac{\kappa(\theta,t,x)}{\theta} |\nabla_x \theta|^2     \right) \psi 
	\notag \\	
	\leq  &
	\int_{\Omega_0} \psi(0)\left(\frac{1}{2}  \frac{|(\rho\uvect)_{0,\delta}|^2}{\rho_{0,\delta}}  + 	\rho_{0,\delta} e(\rho_{0,\delta}, \theta_{0,\delta}) - \rho_{0,\delta} s(\rho_{0,\delta}, \theta_{0,\delta})\widetilde \theta(0)  + \frac{\delta}{\beta-1} \rho^\beta_{0,\delta}  \right)
	\notag \\  
	&  - \psi(0) \int_{\Omega_0} (\rho \uvect)_{0,\delta} \V(0)	
	-\int_0^T \int_{\Omega_t} \rho \uvect \cdot \partial_t (\V  \psi)
	\notag \\
	& - \int_0^T  \int_{\Omega_t} \left(\rho[\uvect \otimes \uvect]    : \nabla_x \V  - \Ss : \nabla_x \V  + p_{ \delta}(\rho,\theta)\, \Div_x \V  \right) \psi 	   
	\notag \\
	&
	-	\int_0^T \int_{\Omega_t} \left[\rho s(\rho,\theta) \Big(\partial_t \widetilde \theta  + \uvect \cdot \nabla_x \widetilde \theta  \Big) - \frac{\kappa(\theta, t, x)}{\theta} \nabla_x \theta \cdot \nabla_x \widetilde \theta   \right] \psi  \notag \\
	& - \int_0^T \psi \int_{\Omega_t} \rho \nabla_x \Psi \cdot \V ,
\end{align}
for all $\psi \in \C^1_c([0,T))$ with $\psi\geq 0$ and $\partial_t \psi \leq 0$, where the choice of $\widetilde \theta$ has been made in \eqref{choice-tilde-theta-B}--\eqref{test-function-theta_B}.  

In above, we simply omit the term $\disp \Big(-\int_0^T \int_\B \lambda \theta^{\alpha+1} \psi\Big)$ from the l.h.s. of the energy inequality since this term is non-positive.

\subsection{Conclusion of the proof: vanishing artificial pressure}
In the final step, we proceed with $\delta\to 0$  in a similar way as developed in \cite{Feireisl-NSF-1}; see also the book \cite{Feireisl-Novotny-book}. We skip the details in this paper since the arguments are by now well-understood.

\appendix

\section{Some auxiliary lemmas}

\begin{lemma}[Generalized Poincar\'e inequality]\label{Poincare}
	Let $1\leq p \leq \infty$, $0<\gamma<\infty$, $U_0>0$ and $\Omega\subset \mathbb R^N$ be a bounded Lipschitz domain. 
	Then, there exists a positive constant $C=C(p, \gamma, U_0)$ such that 
	\begin{align}
		\|v \|_{W^{1,p}(\Omega)} \leq C \left[\| \nabla_x v\|_{L^p(\Omega, \mathbb R^N)} + \Big( \int_U |v|^\gamma \Big)^{\frac{1}{\gamma}}  \right] 
	\end{align}
	for any measurable $U\subset \Omega$, $|U|\geq U_0$ and any $v\in W^{1,p}(\Omega)$. 
\end{lemma} 
A formal proof of the above result is given in \cite[Theorem 11.20, Chapter 11.9]{Feireisl-Novotny-book}.

\begin{lemma}[Korn-Poincar\'e inequality]\label{Korn-Poincare}
	Let $\Omega \subset \mathbb R^3$ be a bounded Lipschitz domain. Assume that $r$ is a
	non-negative function such that 
	\begin{align*}
		0< m_0 \leq \int_\Omega r dx , \quad \int_\Omega r^\gamma \leq K,
	\end{align*}
	for some certain $\gamma>1$. Then 
	\begin{align}
		\| v\|_{W^{1,p}(\Omega, \mathbb R^3)} \leq C(p, m_0, K) \left(\left\| \nabla_x  v + \nabla_x^t v -\frac{2}{3} \Div_x v \mathbb I   \right\|_{L^p(\Omega, \mathbb R^3)}  + \int_\Omega r |v| \right) 
	\end{align}
	for any $v\in W^{1,p}(\Omega, \mathbb R^3)$ and $1<p<\infty$.
\end{lemma}
We refer \cite[Theorem 11.22, Chapter 11.10]{Feireisl-Novotny-book} for the proof of the above lemma.

\section*{Acknowledgments} 
The authors  warmly thank the anonymous reviewers for their useful comments to improve the initial version of the
paper. 
K. Bhandari and {\v{S}}. Ne{\v{c}}asov{\'{a}} have been supported by the Czech-Korean project GA\v{C}R/22-08633J 
and the Praemium Academiae of {\v{S}}. Ne{\v{c}}asov{\'{a}}. The work of B. Huang is supported by the grant from National Natural Science Foundation of China (grant No.11901148), the Fundamental Research Funds for the Central Universities (grant no. JZ2022HGTB0257), and the China Scholarship Council (grant no. 202106695016). 
The Institute of Mathematics, CAS is supported by RVO:67985840.


\begin{thebibliography}{10}
	\newcommand{\enquote}[1]{#1}
	
	\bibitem{Danica}
	D.~Basari{\'{c}}, E.~Feireisl, M.~Luk{\'{a}}{\v{c}}ov{\'{a}}-Medvid'ov{\'a},
	H.~Mizerov{\'{a}} and Y.~Yuan, \enquote{Penalization method for the
		{N}avier-{S}tokes-{F}ourier system}, {\it ESAIM Math. Model. Numer. Anal.}
	\textbf{56} (2022) 1911--1938.
	
	\bibitem{buet2004asymptotic}
	C.~Buet and B.~Despres, \enquote{Asymptotic analysis of fluid models for the
		coupling of radiation and hydrodynamics}, {\it J. Quant. Spectrosc. Radiat.
		Transf.} \textbf{85} (2004) 385--418.
	
	\bibitem{Chauduri-Feireisl}
	N.~Chaudhuri and E.~Feireisl, \enquote{Navier-{S}tokes-{F}ourier system with
		{D}irichlet boundary conditions}, {\it Appl. Anal.} \textbf{101} (2022)
	4076--4094.
	
	\bibitem{Cox}
	J.~P. Cox and R.~Giuli, {\it Principles of stellar structure, I.,II.} (New
	York: Gordon and Breach, 1968).
	
	\bibitem{Lions-Diperna}
	R.~J. DiPerna and P.-L. Lions, \enquote{Ordinary differential equations,
		transport theory and {S}obolev spaces}, {\it Invent. Math.} \textbf{98}
	(1989) 511--547.
	
	\bibitem{ducomet2001simplified}
	B.~Ducomet, \enquote{Simplified models of quantum fluids in nuclear physics},
	{\it Math. Bohem.} \textbf{126} (2001) 323--336.
	
	\bibitem{Ducoment-Sarka-1}
	B.~Ducomet, M.~Caggio, {\v{S}}.~Ne{\v{c}}asov{\'{a}} and M.~Pokorn{\'{y}},
	\enquote{The rotating {N}avier-{S}tokes-{F}ourier-{P}oisson system on thin
		domains}, {\it Asymptot. Anal.} \textbf{109} (2018) 111--141.
	
	\bibitem{Eduard-Ducomet-2}
	B.~Ducomet and E.~Feireisl, \enquote{On the dynamics of gaseous stars}, {\it
		Arch. Ration. Mech. Anal.} \textbf{174} (2004) 221--266.
	
	\bibitem{Eduard-Ducomet-1}
	B.~Ducomet, E.~Feireisl, H.~Petzeltov{\'{a}} and I.~Stra{\v{s}}kraba,
	\enquote{Global in time weak solutions for compressible barotropic
		self-gravitating fluids}, {\it Discrete Contin. Dyn. Syst.} \textbf{11}
	(2004) 113--130.
	
	\bibitem{Ghatak-et-al}
	S.~Eliezer, A.~Ghatak and H.~Hora, {\it An introduction to equations of states,
		theory and applications} (Cambridge University Press, 1986).
	
	\bibitem{Feireisl-NSF-2}
	E.~Feireisl, {\it Dynamics of viscous compressible fluids}, volume~26 of {\it
		Oxford Lecture Series in Mathematics and its Applications} (Oxford University
	Press, Oxford, 2004).
	
	\bibitem{Feireisl-NSF-1}
	E.~Feireisl, \enquote{On the motion of a viscous, compressible, and heat
		conducting fluid}, {\it Indiana Univ. Math. J.} \textbf{53} (2004)
	1707--1740.
	
	\bibitem{Sarka-Kreml-Neustupa-Feireisl-Stebel}
	E.~Feireisl, O.~Kreml, {\v{S}}.~Ne{\v{c}}asov{\'{a}}, J.~Neustupa and
	J.~Stebel, \enquote{Weak solutions to the barotropic {N}avier-{S}tokes system
		with slip boundary conditions in time dependent domains}, {\it J.
		Differential Equations} \textbf{254} (2013) 125--140.
	
	\bibitem{Feireisl-Novotny-book}
	E.~Feireisl and A.~Novotn{\'{y}}, {\it Singular limits in thermodynamics of
		viscous fluids}, Advances in Mathematical Fluid Mechanics (Birkh{\"{a}}user
	Verlag, Basel, 2009).
	
	\bibitem{Feireisl2012weak}
	E.~Feireisl and A.~Novotn{\'{y}}, \enquote{Weak-strong uniqueness property for
		the full {N}avier-{S}tokes-{F}ourier system}, {\it Arch. Ration. Mech. Anal.}
	\textbf{204} (2012) 683--706.
	
	\bibitem{Feireisl2001existence}
	E.~Feireisl, A.~Novotn{\'{y}} and H.~Petzeltov{\'{a}}, \enquote{On the
		existence of globally defined weak solutions to the {N}avier-{S}tokes
		equations}, {\it J. Math. Fluid Mech.} \textbf{3} (2001) 358--392.
	
	\bibitem{Feireisl-hana}
	E.~Feireisl and H.~Petzeltov{\'{a}}, \enquote{On integrability up to the
		boundary of the weak solutions of the {N}avier-{S}tokes equations of
		compressible flow}, {\it Comm. Partial Differential Equations} \textbf{25}
	(2000) 755--767.
	
	\bibitem{HBZ-2022}
	B.~Huang, {\v{S}}.~Ne{\v{c}}asov{\'{a}} and L.~Zhang, \enquote{On the
		compressible micropolar fluids in a time-dependent domain}, {\it Ann. Mat.
		Pura Appl. (4)} \textbf{201} (2022) 2733--2795.
	
	\bibitem{Kalousek2023existence}
	M.~Kalousek, S.~Mitra and {\v{S}}.~Ne{\v{c}}asov{\'a}, \enquote{Existence of
		weak solution for a compressible multicomponent fluid structure interaction
		problem}, {\it arXiv preprint arXiv:2301.11216} (to appear in {\em J. Math.
		Pures Appl.}).
	
	\bibitem{Sarka-et-al-ZAMP}
	O.~Kreml, V.~M{\'{a}}cha, {\v{S}}.~Ne{\v{c}}asov{\'{a}} and
	A.~Wr{\'{o}}blewska-Kami{\'{n}}ska, \enquote{Flow of heat conducting fluid in
		a time-dependent domain}, {\it Z. Angew. Math. Phys.} \textbf{69} (2018)
	Paper No. 119, 27.
	
	\bibitem{sarka-aneta-al-JMPA}
	O.~Kreml, V.~M{\'{a}}cha, {\v{S}}.~Ne{\v{c}}asov{\'{a}} and
	A.~Wr{\'{o}}blewska-Kami{\'{n}}ska, \enquote{Weak solutions to the full
		{N}avier-{S}tokes-{F}ourier system with slip boundary conditions in time
		dependent domains}, {\it J. Math. Pures Appl. (9)} \textbf{109} (2018)
	67--92.
	
	\bibitem{KNP-20-springer}
	O.~Kreml, {\v{S}}.~Ne{\v{c}}asov{\'a} and T.~Piasecki, \enquote{Compressible
		navier-stokes system on a moving domain in the {$L_p$-$L_q$} framework}, in
	{\it Waves in Flows: The 2018 Prague-Sum Workshop Lectures} (Springer, 2020),
	pp. 127--158.
	
	\bibitem{KNP-20}
	O.~Kreml, {\v{S}}.~Ne{\v{c}}asov{\'{a}} and T.~Piasecki, \enquote{Local
		existence of strong solutions and weak-strong uniqueness for the compressible
		{N}avier-{S}tokes system on moving domains}, {\it Proc. Roy. Soc. Edinburgh
		Sect. A} \textbf{150} (2020) 2255--2300.
	
	\bibitem{Lions1996mathematical}
	P.-L. Lions, {\it Mathematical Topics in Fluid Mechanics: Volume 2:
		Compressible Models}, volume~2 (Oxford University Press, 1996).
	
	\bibitem{Macha-Muha-Necasova-Roy-Srjdan}
	V.~M{\'{a}}cha, B.~Muha, {\v{S}}.~Ne{\v{c}}asov{\'{a}}, A.~Roy and
	S.~Trifunovi{\'{c}}, \enquote{Existence of a weak solution to a nonlinear
		fluid-structure interaction problem with heat exchange}, {\it Comm. Partial
		Differential Equations} \textbf{47} (2022) 1591--1635.
	
	\bibitem{S-N-Shorey-book}
	S.~N. Shore, {\it An introduction to astrophysical hydrodynamics} (New York:
	Academic Press, 1992).
	
	\bibitem{StoCar}
	Y.~Stokes and G.~Carrey, \enquote{On generalised penalty approaches for slip,
		free surface and related boundary conditions in viscous flow simulation},
	{\it Inter. J. Numer. Meth. Heat Fluid Flow} \textbf{21} (2011) 668--702.
	
	
	\bibitem{KMNPW-L}
	O.~Kreml, V.~Macha, {\v{S}}.~Ne{\v{c}}asov{\'{a}}, T.~Piasecki and A.~Wr{\'{o}}blewska-Kami{\'{n}}ska, 
	{\it Mathematical theory of
		compressible fluids on moving domains}, Advances in Mathematical Fluid
	Mechanics, Lecture Notes in Mathematical Fluid Mechanics,
	Springer (in preparation).
	
	
\end{thebibliography}

\end{document}